\documentclass[12pt,a4paper]{article}
\usepackage{lmodern}
\usepackage[T1]{fontenc}
\usepackage{hyperref}
\usepackage{amsfonts,amsmath,amssymb,amsopn,amsthm}
\usepackage[all,2cell]{xy}\UseAllTwocells
\usepackage{vmargin}
\usepackage{makeidx}
\usepackage{graphicx}
\usepackage{tikz}
\usepackage{comment}
\usepackage{mathabx}
\usepackage[ style=alphabetic,citestyle=alphabetic,sorting=nyt,sortcites=true,autopunct=true,babel=hyphen,hyperref=true,abbreviate=false,backref=true]{biblatex}
\AtEveryBibitem{ \clearfield{url} \clearfield{urldate} \clearfield{review} \clearfield{series} \clearfield{doi} \clearfield{isbn} \clearfield{issn} } 
\defbibheading{bibempty}{}
\DeclareNolabel{\nolabel{\regexp{[\p{Z}]+}}}
\usepackage{scalerel,stackengine}
\stackMath
\newcommand\reallywidehat[1]{%
\savestack{\tmpbox}{\stretchto{%
 \scaleto{%
 \scalerel*[\widthof{\ensuremath{#1}}]{\kern-.6pt\bigwedge\kern-.6pt}%
 {\rule[-\textheight/2]{1ex}{\textheight}}
 }{\textheight}%
}{0.5ex}}%
\stackon[1pt]{#1}{\tmpbox}%
}
\newcommand{\downsquigarrow}{\mathbin{\rotatebox[origin=c]{-90}{$\rightsquigarrow$}}}
\title{Rigid cohomology of locally noetherian formal schemes. Part 1: Geometry}
\author{Bernard Le Stum}
\date{Version of \today}
\newtheorem{thm}{Theorem}[section]
\newtheorem{prop}[thm] {Proposition}
\newtheorem{cor}[thm] {Corollary}
\newtheorem{lem}[thm] {Lemma}
\theoremstyle{definition}
\newtheorem{dfn}[thm] {Definition}
\newenvironment{xmp}[1][Example]{\begin{trivlist} \item[\hskip \labelsep {\bfseries #1}]}{\end{trivlist}}
\newcommand{\Addresses}{{
 \bigskip
 \footnotesize

Bernard Le Stum, \textsc{IRMAR, Université de Rennes,
Campus de Beaulieu, 35042 Rennes cedex, France}\par\nopagebreak
\texttt{bernard.le-stum@univ-rennes1.fr}

}}

\begin{document}

\maketitle

\bigskip

\begin{abstract}
We set up the geometric background necessary to extend rigid cohomology from the case of algebraic varieties to the case of general locally noetherian formal schemes.
In particular, we generalize Berthelot's strong fibration theorem to adic spaces: we show that if we are given a morphism of locally noetherian formal schemes which is partially proper and formally smooth around a formal subscheme, and we pull back along a morphism from an analytic space which is locally of noetherian type, then we obtain locally a fibration on strict neighborhoods. 
\end{abstract}

\tableofcontents

\addcontentsline{toc}{section}{Introduction}
\section*{Introduction}
\subsection*{State of the art}
Pierre Berthelot developed rigid cohomology in the early 80's in \cite{Berthelot96c*} and \cite{Berthelot83*} (see also \cite{LeStum07}) as a $p$-adic cohomology theory for algebraic varieties $X$ defined over a field $k$.
It takes its values over a complete non archimedean field $K$ of characteristic zero whose residue field is $k$.
The main idea consists in showing that the geometry of $X$ is reflected into the geometry of some subspace of a rigid analytic variety over $K$.
More precisely, one embeds $X$ into some formal scheme $P$ and builds the tube $\,]X[$ of $X$ in $P$ as a subspace of the generic fiber of $P$.
When $P$ is proper and smooth around $X$, then the de Rham cohomology of a small neighborhood $V$ of this tube is essentially independent of the choices.
This is a consequence of the strong fibration theorem of Berthelot: if a morphism $Q \to P$ is proper and smooth around $X$, then it induces locally a fibration in open polydiscs in the neighborhood of the tubes.

There exists an alternative description of rigid cohomology that consists in putting together all the embeddings of $X$ into some $P$ and all the neighborhoods $V$ of $\,]X[$.
As explained in \cite{LeStum11}, this gives rise to the overconvergent site whose cohomology is exactly rigid cohomology.

When the base field is not perfect, Christopher Lazda and Ambrus P\`al showed in \cite{LazdaPal16} that the theory can be refined.
More precisely, in order to study algebraic varieties over $k((t))$, where $k$ is a perfect field, they replace the Amice ring $\mathcal E$, which is the natural field of coefficients in Berthelot's theory, with the bounded Robba ring $\mathcal E^\dagger$.
They also replace Tate's rigid analytic geometry with Huber theory of adic spaces: this is necessary because the bounded robba ring lives on an adic space whose only closed point is a valuation of height 2.

In order to obtain constructible coefficients for rigid cohomology, Berthelot also developed in \cite{Berthelot02} a theory of arithmetic $\mathcal D$-modules.
First of all, on any smooth formal scheme $P$, he builds a ring $\mathcal D^\dagger_{P}$ of overconvergent differential operators on $P$.
Then, when $X$ is a subvariety of $P$, he considers a category of coherent $F-\mathcal D^\dagger_{P \mathbb Q}$-modules with support on $X$.
In the thesis of David Pigeon (\cite{Pigeon14}), and thereafter in the work of Daniel Caro and David Vauclair (\cite{CaroVauclair15*}), the theory is extended to formal schemes that are only differentially $p$-smooth.
More recently, Richard Crew showed in \cite{Crew17*} that Berthelot's theory of arithmetic $\mathcal D$-modules also extends to the case of formally smooth formal schemes that are not necessarily of finite type, but only ``universally'' noetherian (formally of finite type for example).
Surprisingly, these two extensions of Berthelot's original theory are rather orthogonal: anticipating on forthcoming notations, we can say that, when Berthelot considers $\mathbb A := \mathrm{Spec}(k[t])$, then Caro and Vauclair will look at $\mathbb A^\mathrm b := \mathrm{Spec}(k[[t]])$ (which is not universally noetherian) and Crew will be interested in $\mathbb A^- := \mathrm{Spf}(k[[t]])$ (which is not a $p$-adic formal scheme).

\subsection*{The main result}

In this article, we set up the geometric background necessary to extend rigid cohomology from the case of an algebraic variety over a field to any morphism of locally noetherian formal schemes.
More precisely, we extend Berthelot's strong fibration theorem to this setting.
In order to state the theorem, it is convenient to extend the method and the vocabulary of \cite{LeStum11}.
The idea is to use formal schemes as a bridge between schemes and analytic spaces.
We assume here that all formal schemes are locally noetherian and all adic spaces are locally of noetherian type (see below).

An overconvergent adic space is a pair made of a locally closed embedding of formal schemes $X \hookrightarrow P$ and a morphism of adic spaces $P^\mathrm{ad} \leftarrow V$ (or equivalently a morphism $P \leftarrow V^+ $).
We shall write $(X \hookrightarrow P \leftarrow V)$ for short.
A formal morphism of overconvergent adic spaces is a pair of commutative diagrams
\[
\xymatrix{ Y \ar@{^{(}->}[r] \ar[d]^f & Q \ar[d]^v \\ X \ar@{^{(}->}[r] & P}, \quad \xymatrix{Q^{\mathrm{ad}} \ar[d]^{v^{\mathrm{ad}}}& W \ar[l] \ar[d]^u \\ P^{\mathrm{ad}} & V \ar[l] }
\]
The formal morphism is said to be right cartesian if $W$ is a neighborhood of the inverse image of of the tube (see below) of $X$ along $v^{\mathrm{ad}}$.
It is called formally smooth (resp.\ partially proper) around $Y$ if there exists an open subset $U$ (resp.\ a closed subset $Z$) of $Q$ that contains $Y$ and such that the restriction of $v$ to $U$ (resp.\ to $Z$) is formally smooth (resp.\ partially proper).

We need to introduce the notion of a tube $\,]X[_{V}$ of $X$ in $V$: this is obtained by Boolean combination from the closed embedding case where we take the inverse image of the adic space associated to the completion of $P$ along $X$ (this is different from the inverse image under specialization).
We call a formal morphism a strict neighbourhood if $f$ is an isomorphism, $v$ is locally noetherian, $u$ is an open immersion and the map induced on the tubes is surjective (or equivalently a homeomorphism).
The overconvergent adic site is obtained by turning strict neighborhoods into isomorphisms and using the topology inherited from $V$.
We will simply denote by $(X,V)$ the corresponding overconvergent adic space viewed as an object of the overconvergent adic site (and forget about $P$ in the notations).

The strong fibration theorem states that if we are given a right cartesian morphism as above which is formally smooth and partially proper around $Y$ and such that $f$ is an isomorphism, then it induces locally an isomorphism $(Y, W) \simeq (X, \mathbb D^{n-}_{V})$ in the overconvergent site.
Together with the overconvergent Poincar\'e lemma, the classic analog of this result is the fundamental tool to prove that rigid cohomology is well defined and functorial.

\subsection*{Content}

In the first section, we briefly review the theory of adic formal schemes. The main purpose is to set up the vocabulary and notations and give elementary but fundamental examples.
It will also be useful as a guideline for the theory of adic spaces that will be discussed thereafter.
We also introduce the important notion of being partially proper for a morphism of formal schemes.

Section two is devoted to adic spaces. We mostly recall the main definitions and give some examples as in the previous section. This may be useful for the reader because the theory is not fully documented.

In section three, we recall how one can associate an adic space to a formal scheme and show that most classical properties of adic morphisms of formal schemes are reflected in the corresponding morphism of adic spaces.
In particular, we check that this holds for formal smoothness or (faithful) flatness.
More interesting, we also introduce the notion of a property being analytic and show that if a morphism of formal schemes is partially proper then the corresponding morphism of adic spaces is analytically partially proper.

In section four, we first discuss the notion of a formal embedding and prove the formal fibration theorem.
Then, we extend the notion of an overconvergent space (that was initially developed using Berkovich analytic spaces) to the adic world.
This is essential for the development of a site-theoretic approach to rigid cohomology.
Nevertheless, it should be noticed that one might also have followed Berthelot's original approach and completely avoided the notion of overconvergent space.
Then, we introduce the notion of a tube in the adic world which is sensibly different form the classical one which we call the naive tube. We show that in the analytic situation, the tube of closed (resp.\ open) subset is equal to the closure (resp.\ interior) of the naive tube and we also introduce the notion of a tube of finite radii.

In the last section, we define the overconvergent site as the localization of the category of overconvergent adic spaces (and formal morphisms) with respect to strict neighborhoods.
We show that, even if the topology comes from the analytic side, the category is, in some sense, also local on the (formal) scheme side.
Finally, we use all the material obtained so far in order to prove the strong fibration theorem.

\subsection*{Many thanks}

Several parts of this paper were influenced by some conversations that I had with many mathematicians and I want to thank in particular Ahmed Abbes, Tomoyuki Abe, Richard Crew,  Veronika Ertl, Kazuhiro Fujiwara, Michel Gros, Fumiharu Kato, Christopher Lazda, Vincent Mineo-Kleiner, Laurent Moret-Bailly, Matthieu Romagny and Alberto Vezzani.

\subsection*{Notations/conventions}

\begin{enumerate}
\item
A \emph{monoid} $G$ is a set endowed with an action on itself which is associative with unit.
Unless otherwise specified, all monoids and groups are assumed to be \emph{commutative}, but we may insist for emphasis and call them \emph{abelian} when the law is additive.
A \emph{ring} $A$ is an abelian group endowed with a linear action on itself which is associative with unit.
Unless otherwise specified, all rings and fields are assumed to be \emph{commutative}, but again, we may insist for emphasis.
\item If $X$ is a $G$-set, $E \subset G$ and $F \subset X$, we will write
\[
EF := \{gx \colon g \in E, x \in F\}.
\]
By induction, if $E \subset G$, we will have
\[
E^n = \{g_{1}\cdots g_{n} \colon g_{i} \in E\}.
\]
If $M$ is an $A$-module, $E \subset A$ and $F \subset M$, we will denote by
\[
E \cdot F := \left\{\sum f_{i}x_{i} \colon f_{i} \in E, x_{i} \in F\right\}
\]
the additive subgroup generated by $EF$.
If $E \subset A$, we will denote by
\[
E^{\cdot n} := \left\{\sum f_{i_{1}}\cdots f_{i_{n}} \colon f_{i_{j}} \in E \right\}
\]
the additive subgroup generated by $E^n$.
We shall make an exception and use the standard notation $IJ$ or $I^n$ instead of $I\cdot J$ or $I^{\cdot n}$ when $I$ and $J$ are ideals in $A$.
\item A \emph{(topologically) ringed space} is a topological space $X$ endowed with a sheaf of (topological\footnote{It might be clever to impose $\mathbb Z$-linearity on the topology in order to handle colimits.}) rings $\mathcal O_X$.
A morphism is made of a continuous map $f : Y \to X$ and a (continuous) morphism of (topological) rings $\mathcal O_X \to f_*\mathcal O_Y$.
There exists an obvious forgetful functor from topologically ringed spaces to ringed spaces that commutes with all limits and colimits.
Conversely, any ringed space may be seen as a topologically ringed space by using the fully faithful adjoint to the forgetful functor (discrete topology and sheaffifying).
Note that a (topologically) ringed space is a sheaf on the big site of (topologically) ringed spaces: in other words, the site is subcanonical.
\item A \emph{(topologically) locally ringed space} is a (topologically) ringed space $X$ whose stalks are all local rings.
A morphism of (topologically) locally ringed spaces is a morphism of (topologically) ringed spaces that induces a local homomorphism on the stalks.
The above forgetful functor and its adjoint are compatible with these new conditions.
And again, the sites are subcanonical.
If $x \in X$, we will denote the residue field of the local ring $\mathcal O_{X,x}$ by $\kappa(x)$.
If $f$ is a section of $\mathcal O_{X}$ in a neighborhood of $x$, we will denote by $f(x)$ the image of $f$ in $\kappa(x)$.
A morphism of (topologically) locally ringed spaces $f \colon Y \to X$ is said to be \emph{flat} (resp.\ \emph{faithfully flat}) if $f^*$ is exact (resp.\ and $f$ is surjective).
\item \label{valsp} A \emph{(topologically) valued ringed space} is a (topologically) locally ringed space $X$ whose stalks $\mathcal O_{X,x}$ are endowed with a valuation $v_{x}$ defined up to equivalence whose support is exactly $\mathfrak m_{X,x}$.
This is actually equivalent to choosing a valuation ring $\mathcal V(x)$ of $\kappa(x)$.
By definition, the \emph{height} of $x$ is the height of $v_{x}$ and $x$ is called \emph{trivial} when $v_{x}$ is trivial (height = $0$).
A morphism of (topologically) valued ringed spaces is a morphism of (topologically) locally ringed spaces such that the induced morphism on the stalks is compatible up to equivalence with the valuations.
This is equivalent to require that it induces a morphism between the valuation rings.
There exists an obvious functor that forgets the valuation.
It possesses a fully faithful adjoint which is simply obtained by endowing $\mathcal O_{X,x}$ with the valuation coming from the trivial valuation on $\kappa(x)$.
In particular, any (topologically) locally ringed space may be seen as a (topologically) valued ringed space.
\item \label{plus} A \emph{doubly (topologically) (locally) ringed space} is a triple $(X, \mathcal O_{X}, \mathcal O_{X}^+)$ were both $(X, \mathcal O_{X})$ and $(X, \mathcal O_{X}^+)$ are (topologically) (locally) ringed spaces and $\mathcal O_{X}^+ \subset \mathcal O_{X}$.
Morphisms are defined in the obvious way.
We will still denote by $X$ the (topologically) (locally) ringed space $(X, \mathcal O_{X})$ and then write $X^+ := (X, \mathcal O_{X}^+)$.
This provides and adjoint and a coadjoint to the obvious functor $(X, \mathcal O_{X}) \mapsto (X, \mathcal O_{X}, \mathcal O_{X})$.
\item
If $X$ is a (topologically) valued ringed space, then we shall denote by $\mathcal O_{X}^+$ the subsheaf locally defined by the conditions $v_{x}(f) \geq 0$.
This defines a functor from (topologically) valued ringed spaces to doubly (topologically) locally ringed spaces which happens to be fully faithful.
\item Let $X$ be a topological space and $x,y \in X$.
We write $y \rightsquigarrow x$, and we say that $x$ is a \emph{specialization} of $y$ or that $y$ is a \emph{generalization} of $x$, if any neighborhood of $x$ is also a neighborhood of $y$.
Alternatively, it means that $x \in \overline {\{y\}}$ - the topological closure.
This defines a partial order ($y \geq x$) on the points of $X$.
\item A topological space is said to be \emph{coherent} if it is quasi-compact, quasi-separated - diagonal map is quasi-compact - and admits a basis of quasi-compact open subsets.
It is said to be \emph{sober} if any irreducible subset has a unique maximal point.
It is called \emph{spectral} if it is coherent and sober.
\item A subset of a topological space $X$ is \emph{constructible} (resp.\ \emph{ind-constructible}, resp.\ \emph{pro-constructible}) if it is a (resp.\ a union of, resp.\ an intersection of) boolean combination(s) of retrocompact - inclusion is quasi-compact - open subspaces.
If $X$ is coherent, a locally constructible (resp.\ locally ind-constructible, resp.\ locally pro-constructible) subset is automatically constructible (resp.\ ind-constructible, resp.\ pro-constructible).
The \emph{constructible topology} on a topological space $X$ is the coarsest topology for which the locally constructible subsets are open.
If $Y$ is a pro-constructible (resp.\ ind-constructible) subset of a spectral space $X$, then the closure (resp.\ the interior) of $Y$ in $X$ has the following description
\[
\overline Y = \{x \in X \colon \exists y \in Y, y \rightsquigarrow x\} \quad \mathrm{(resp.}\ \mathring Y = \{x \in X \colon \forall y \rightsquigarrow x, y \in Y\}).
\]

\item If $X$ is a locally spectral space, then a subset is open (resp.\ closed, resp.\ open and closed) in the constructible topology if and only if it is locally ind-constructible (resp.\ locally pro-constructible, resp.\ locally constructible).
A subset of a locally spectral space is open (resp.\ closed) if and only if it is locally ind-constructible (resp.\ locally pro-constructible) and stable under generalization (resp.\ specialization).
Any quasi-compact quasi-separated continuous map of locally spectral spaces is continuous and closed for the constructible topology.
\item
Assume that $g : \mathcal C' \to \mathcal C$ is a given functor. 
Then,
\begin{enumerate}
\item the topology \emph{induced} on $\mathcal C'$ by a topology on $\mathcal C$ is the finest topology making $g$ continuous,
\item the topology \emph{inherited} (or \emph{coinduced}) by $\mathcal C'$ from a topology on $\mathcal C$ is the coarsest topology making $g$ cocontinuous,
\item the \emph{image} topology on $\mathcal C$ of a topology on $\mathcal C'$ is the coarsest topology making $g$ continuous,
\item the \emph{coimage} topology on $\mathcal C$ of a topology on $\mathcal C'$ is the finest topology making $g$ cocontinuous.
\end{enumerate}

\item We will use the term \emph{embedding} as an equivalent to a \emph{locally closed immersion} or to a \emph{fully faithful functor}, depending on the context.

\item
When $Y \subset X$ is a locally closed subset of a topological space, we denote by $\infty_Y$ and call \emph{locus at infinity} of $Y$ in $X$, the complement of $Y$ in its closure $\overline Y$ in $X$.

\end{enumerate}

\section{Adic formal schemes}

In this section, we briefly review the basics of adic formal schemes, give some examples and generalize some standard notions to the case of non-adic morphisms.
We send the reader to chapter 2 of \cite{Abbes10} and chapter 1 of \cite{FujiwaraKato18} for a recent presentation of the subject (see also chapter 10 of \cite{EGAI}).
Note that the terminology may vary from one reference to another and we tend to prefer a vocabulary which is compatible with the usual conventions in the theory of adic spaces of Huber.

At some point, we will impose noetherian conditions but we can make some general statements first.

\subsection{Adic rings}

We do not assume adic rings to be complete but we require them to have a finitely generated ideal of definition (this will be automatic when we stick to the noetherian world).
More precisely:

\begin{dfn}
A topological ring $A$ is called \emph{adic} if there exists a \emph{finitely generated} ideal $I$ whose powers $I^n$ form a cofinal system of neighborhoods of $0$.
Such an ideal is called an \emph{ideal of definition} for $A$.
A homomorphism $A \to B$ of adic rings is called \emph{adic} when $IB$ is an ideal of definition of $B$.
\end{dfn}

Following Huber, we do \emph{not} assume the topology to be complete or Hausdorff in general and will denote by $\hat A := \varprojlim A/I^{n+1}$ the (Hausdorff) completion of $A$.
On the other hand, we always require $I$ to be \emph{finitely generated}.
Abbes calls such a topological ring \emph{strongly} adic and Fujiwara and Kato say: \emph{of finite ideal type}.
Under this assumption, the ring $\widehat A$ will be a complete adic ring with ideal of definition $I\widehat A$ and we will have $A/I = \widehat A/I\widehat A$ (and we may in practice replace $A$ with $\widehat A$).

If we were ready to always assume completeness, we could (but we will not) consider the more general notion of \emph{admissible} ring.
On the other side, we will always see a usual ring as a complete adic ring for the discrete topology: the discrete topology is adic (and complete) with respect to the zero (or any nilpotent) ideal.

A morphism of adic rings $A \to B$ is simply a continuous homomorphism of rings.
Equivalently, it means that it sends any ideal of definition into some ideal of definition.
We will always make it clear if we assume that the morphism is actually adic (which is a much stronger condition).

\begin{dfn}
If $A$ is a topological ring, then the \emph{formal spectrum} of $A$ is the set $P := \mathrm{Spf}(A)$ of open prime ideals of $A$.
\end{dfn}

If $A$ is a topological ring, $f \in A$ and $\mathfrak p \in P := \mathrm{Spf}(A)$, then we will denote by $f(\mathfrak p)$ the image of $f$ in $\kappa(\mathfrak p) := \mathrm{Frac}(A/\mathfrak p)$.
The set $P$ is endowed with the topology for which the subsets
\[
D(f) = \{\mathfrak p \in P \colon f(\mathfrak p) \neq 0\}
\]
form a basis of open subsets.
This is a spectral space.

When $A$ is an adic ring, there exists a unique sheaf of topological rings $\mathcal O_{P}$ such that (up to a canonical isomorphism)
\[
\Gamma(D(f), \mathcal O_{P}) = \widehat {A[1/f]}
\]
(this is the completion for the $I$-adic topology if $I$ denotes an ideal of definition of $A$).
This turns $P$ into a topologically locally ringed space.

If $A$ is an adic ring, then there exists a natural isomorphism of topologically locally ringed spaces $\mathrm{Spf}(\widehat A) \simeq \mathrm{Spf}(A)$ and this is why we may usually only consider complete adic rings.
If one is ready to work only with complete rings, one can define more generally the topologically ringed space associated to an admissible ring $A$.

Note that we always have $\Gamma(P, \mathcal O_{P}) = \widehat A$ but Cartan's theorem A and B do not hold without further hypothesis on $A$.
We have:

\begin{prop}
If $A$ is a \emph{noetherian} adic ring and $P := \mathrm{Spf}(A)$, then the functors
\[
M \mapsto \mathcal O_{P} \otimes^{\mathrm L}_{\widehat A} M \quad \mathrm{and} \quad \mathcal F \mapsto \mathrm R\Gamma(P, \mathcal F)
\]
induce an equivalence between finite $\widehat A$-modules and coherent $\mathcal O_{V}$-modules.
\end{prop}

\begin{proof}
This is shown in \cite{EGAI} for example.
\end{proof}

The result also holds without the noetherian hypothesis if we require the topology on $A$ to be discrete (in which case it reduces to the analogous result on usual spectra).

\subsection{Formal schemes}

Working with formal schemes will give us a lot of flexibility to travel between the algebraic and the analytic world.

\begin{dfn}
An \emph{adic formal scheme} is a topologically locally ringed space $P$ which is locally isomorphic to $\mathrm{Spf}(A)$ where $A$ is an adic ring.
It is said to be \emph{affine} if it is actually isomorphic to some $\mathrm{Spf}(A)$.
\end{dfn}

An adic formal scheme is a locally spectral space.
Note that one can define the more general notion of \emph{formal scheme} by using \emph{admissible} rings.

If $X$ is a (usual) scheme, then it is a locally ringed space and the topologically locally ringed space associated to $X$ is an adic formal scheme.
This gives rise to a fully faithful functor.
Note that we recover $X$ simply by forgetting the topology on the sheaf of rings.
In practice, we will \emph{identify} $X$ with the corresponding adic formal scheme.
It means that all statements involving formal schemes may also be applied to usual schemes.
Note also that, with this identification, if $A$ is an adic ring with ideal of definition $I$, we have
\[
\mathrm{Spf}(A) = \varinjlim \mathrm{Spec}(A/I^n).
\]

The functor
\[
A \mapsto \mathrm{Spf}(A)
\]
is fully faithful on complete adic rings.
Better, there exists an adjunction
\[
\mathrm{Hom}(X, \mathrm{Spf}(A)) \simeq \mathrm{Hom}(A, \Gamma(X, \mathcal O_{X}))
\]
(continuous homomorphisms on the right hand side).
As a consequence, the category of adic formal schemes has finite limits (but this will not be the case anymore when we make noetherian hypothesis) and we always have
\[
\mathrm{Spf}(A) \times_{\mathrm{Spf}(R)} \mathrm{Spf}(B) = \mathrm{Spf}(A \otimes_{R} B).
\]

\begin{xmp}
If $S$ is any adic formal scheme, then we may consider:
\begin{enumerate}
\item the relative \emph{affine space}
\[
\mathbb A^n_{S} = \underbrace{\mathbb A \times \cdots \times \mathbb A}_{n\ \mathrm{times}} \times S
\]
in which $\mathbb A := \mathrm{Spec}(\mathbb Z[T])$ (seen as a formal scheme).
This is the same thing as the usual formal affine space which is usually denoted with an extra hat.
If $A$ is an adic ring with ideal of definition $I$, $S = \mathrm{Spf}(A)$ and $A[T_{1}, \ldots, T_{n}]$ has the $I$-adic topology, then we have
\[
\mathbb A^n_{S} = \mathrm{Spf}\left(A[T_{1}, \ldots, T_{n}]\right)
\]
and we will sometimes write $\widehat{\mathbb A}^n_{A}$ in this case (the hat is meant to insist on the fact that we take the topology of $A$ into account).
Note that we could as well use the $I$-adic completion $A\{T_{1}, \ldots, T_{n}\}$ of $A[T_{1}, \ldots, T_{n}]$ (sometimes also written $\widehat A\langle T_{1}, \ldots, T_{n} \rangle$).
This gives the same adic formal scheme.
\item the relative \emph{projective space}
\[
\mathbb P^n_{S} := \mathbb P^n \times S
\]
(which is also usually denoted with an extra hat)
in which
\[
\mathbb P^n = \mathrm{Proj}(\mathbb Z[T_{0}, \ldots, T_{n}])
\]
(scheme seen as a formal scheme).
We will also write $\widehat{\mathbb P}^n_{A}$ when $S = \mathrm{Spf}(A)$.
We may drop the exponent $n$ when $n=1$ and simply denote by $\mathbb P$ the projective line.
\item the relative \emph{open affine space}
\[
\mathbb A^{n,-}_{S} := \underbrace{\mathbb A^- \times \cdots \times \mathbb A^-}_{n\ \mathrm{times}} \times S
\]
in which $\mathbb A^- := \mathrm{Spf}(\mathbb Z[T])$ where $\mathbb Z[T]$ has the $T$-adic topology.
If $A$ is an adic ring with ideal of definition $I$, $S = \mathrm{Spf}(A)$ and $A[T_{1}, \ldots, T_{n}]$ has the $I[T_{1}, \ldots, T_{n}] + (T_{1}, \ldots, T_{n})$-adic topology, then
\[
\mathbb A^{n,-}_{S} = \mathrm{Spf}\left(A[T_{1}, \ldots, T_{n}]\right) = \mathrm{Spf}\left(\widehat A[[T_{1}, \ldots, T_{n}]]\right)
\]
and we may also write $\widehat{\mathbb A}^{n,-}_{A}$.
\item the relative \emph{bounded affine space}
\[
\mathbb A^{n,\mathrm{b}}_{S} = \mathrm{Spf}\left(A[[T_{1}, \ldots, T_{n}]]\right),
\]
(also denoted by $\widehat{\mathbb A}^{n,\mathrm b}_{A}$)
in which $S = \mathrm{Spf}(A)$ and $A[[T_{1}, \ldots, T_{n}]]$ has the $I$-adic topology.
Be careful that $B \widehat \otimes_{A} A[[T_{1}, \ldots, T_{n}]] \not\simeq B[[T_{1}, \ldots, T_{n}]]$ if completion is meant with respect to the topology of $B$ (unless $B$ is finite over $A$ or $n = 0$).
It follows that the bounded affine space is not of local nature (we cannot glue).
\end{enumerate}

There exists a sequence of ``inclusions''
\[
\mathbb A^{n,-}_{S} \hookrightarrow \mathbb A^{n,\mathrm{b}}_{S} \hookrightarrow \mathbb A^n_{S} \hookrightarrow \mathbb P^n_{S}.
\]
We may also notice that $\mathbb A$ represents the sheaf of rings $P \mapsto \Gamma(P, \mathcal O_{P})$ and that $\mathbb A^-$ represents the ideal of topologically nilpotent elements.
\end{xmp}

\subsection{Morphisms of adic formal schemes}

We briefly review the basic properties of morphisms of adic formal schemes.

By definition, a \emph{morphism} of adic formal schemes $Q \to P$ is simply a morphism of topologically locally ringed spaces.
It is said to be \emph{adic} if it comes locally from an adic morphism $A \to B$ of adic rings.
In the affine case, it will then come \emph{globally} from an adic morphism of adic rings.
A morphism of adic formal schemes $Q \to P$ is said to be \emph{affine} if it comes locally \emph{on $P$} from a (not necessarily adic) morphism of adic rings $A \to B$.
Then, in the case $P$ is affine, $Q$ will be affine also.

An \emph{open immersion} of adic formal schemes is simply an open immersion of topologically locally ringed spaces.
A morphism of adic formal schemes $Q \to P$ is called a \emph{closed immersion} (resp.\ \emph{finite}) if it comes locally \emph{on $P$} from an \emph{adic} surjective (resp.\ finite) map $A \to B$.
Again, in the affine case, it will come globally from an adic surjective (resp.\ finite) map.
A \emph{locally closed immersion} is the composition of a closed immersion with an open immersion (in that order).
This gives in particular rise to the notion of a \emph{locally closed formal subscheme}.

A morphism $Q \to P$ of adic formal schemes is said to be \emph{locally of finite type} if $Q$ is locally isomorphic to a closed formal subscheme of some relative affine space $\mathbb A^n_{P}$.
This is always an \emph{adic} morphism.
If moreover, the morphism is quasi-compact, then it is said to be \emph{of finite type}.
A morphism of affine formal schemes is of finite type if and only if it is \emph{globally} a closed subscheme of some affine space.
There also exists a stronger notion of morphism \emph{(locally) finitely presented} but this will not matter to us when we enter the noetherian world.

A morphism $Q \to P$ of adic formal schemes is \emph{locally quasi-finite} if it is locally of finite type with discrete fibers (locally on $P$ \emph{and} $Q$, this is the composition of an open immersion and a finite map).
It is said to be \emph{quasi-finite} if moreover it is quasi-compact.

If $u : Q \to P$ is any morphism of adic formal schemes, then the diagonal map $\Delta : Q \to Q \times_{P} Q$ is a locally closed immersion.
The morphism $u$ is said to be \emph{quasi-separated} (resp.\ \emph{separated}) if $\Delta$ is quasi-compact (resp.\ closed).
The morphism $u$ is said to be \emph{proper} if it is separated of finite type and universally closed.

A morphism of adic formal schemes $v \colon Q \to P$ is said to be \emph{flat} (resp.\ \emph{faithfully flat}) if it is flat (resp.\ faithfully flat) as a morphism of locally ringed spaces.
Thus, it means that $v^*$ is exact (resp.\ and $v$ is surjective).
Flatness is a local notion that can be checked on stalks: it means that $\mathcal O_{P,v(y)} \to \mathcal O_{Q,y}$ is flat at all $y \in Q$.
This is only under noetherian hypothesis that there exists a global interpretation.
More precisely, if $A \to B$ is a morphism of \emph{complete noetherian} adic rings and $\mathrm{Spf}(B) \to \mathrm{Spf}(A)$ is (faithfully) flat then the morphism of rings $A \to B$ is (faithfully) flat.
Conversely, if $A \to B$ is flat, then $\mathrm{Spf}(B) \to \mathrm{Spf}(A)$ is flat.
The same holds for faithful flatness but only when $A \to B$ is \emph{adic}.

\begin{xmp}
\begin{enumerate}
\item If both $\mathbb Z$ and $\mathbb F_p$ are endowed with the $1$-adic topology, then the adic morphism $\mathbb Z \twoheadrightarrow \mathbb F_p$ is not flat but the corresponding morphism of formal schemes $\emptyset \simeq \emptyset$ is faithfully flat.
\item The identity $\mathbb Z_p \to \mathbb Z_p$ is obviously faithfully flat, but the morphism $\mathrm{Spf}(\mathbb Z_p) \to \mathrm{Spec}(\mathbb Z_p)$ is not surjective and therefore not faithfully flat.
\end{enumerate}
\end{xmp}

A morphism $u : Q \to P$ is said to be \emph{formally unramified} (resp.\ \emph{formally smooth}, resp.\ \emph{formally \'etale}) if any commutative diagram
\[
\xymatrix{Q \ar[r] & P \\ \mathrm{Spec}(R/\mathfrak a) \ar@{^{(}->}[r] \ar[u] & \mathrm{Spec}(R) \ar[u] \ar@{-->}[ul] }
\]
with $\mathfrak a$ nilpotent may be completed by the diagonal arrow in at most (resp.\ at least, resp.\ exactly) one way.
When $P = \mathrm{Spf}(A)$ and $Q = \mathrm{Spf}(B)$, it means that the morphism of rings $A \to B$ is formally unramified (resp.\ formally smooth, resp.\ formally \'etale).
Be careful that being formally smooth is \emph{not} a local property in general (but formally unramified and formally \'etale are).
The morphism $u$ is called \emph{unramified} (resp \emph{smooth}, resp.\ \emph{\'etale}) if it is formally unramified (resp.\ formally smooth, resp.\ formally \'etale) and locally finitely presented (which is equivalent to locally of finite type in our noetherian situation below).
There exists an intermediate notion of smoothness: a morphism $u : Q \to P$ of adic formal schemes is \emph{differentially smooth} if there exists, locally on $Q$, a formally \'etale morphism $Q \to \mathbb A^{n}_{P}$.

If $Q \hookrightarrow P$ is a locally closed immersion defined (on some open subset of $P$) by an ideal $\mathcal I_{Q}$, we may consider the \emph{first infinitesimal neighborhood} $Q^{(1)}$ of $Q$ in $P$ defined by $\mathcal I_{Q}^2$ (which is an adic formal scheme with the same underlying space as $Q$) and the corresponding short exact sequence
\[
0 \to \check{\mathcal N}_{Q/P} \to \mathcal O_{Q_{(1)}} \to \mathcal O_{Q} \to 0
\]
where $\check{\mathcal N}_{Q/P}$ is by definition the \emph{conormal sheaf}.
In the case of the diagonal immersion $Q \hookrightarrow Q \times_{P} Q$ associated to a morphism $Q \to P$, we obtain the \emph{sheaf of differential forms} $\Omega^1_{Q/P}$.
When $P = \mathrm{Spf}(A)$ and $Q = \mathrm{Spf}(B)$ are affine, we have $\Gamma(Q, \Omega^1_{Q/P}) = \widehat \Omega^1_{B/A}$.

\subsection{Locally noetherian formal schemes}

From now on and unless otherwise specified, \textbf{all \emph{adic formal schemes} will be assumed to be \emph{locally noetherian}} (according to the forthcoming definition) and simply called \emph{formal schemes}.

\begin{dfn}
An adic formal scheme $P$ is said to be \emph{locally noetherian} if it is locally isomorphic to $\mathrm{Spf}(A)$ where $A$ is a noetherian adic ring.
It is said to be \emph{noetherian} if, moreover, it is quasi-compact.
\end{dfn}

In this situation, it is shown in proposition 1.1.3 of \cite{Crew17*} that a formally smooth morphism of formal schemes if always flat.

Unfortunately, the product of two locally noetherian adic formal schemes is not necessarily locally noetherian (see example below).
In order to get around this problem, the following notion was introduced by Richard Crew in section 1.2 of \cite{Crew17*} (he says (locally) \emph{universally} noetherian):

\begin{dfn}
A morphism of adic formal schemes $u \colon P' \to P$ is said to be \emph{locally noetherian} if, whenever $Q \to P$ is a morphism of adic formal schemes with $Q$ locally noetherian, then $u^{-1}(Q) := P' \times_{P} Q$ is locally noetherian.
In the case $u$ is quasi-compact, the morphism is said to be \emph{noetherian}.
\end{dfn}

The notion of locally noetherian morphism is stable under product and composition.

It is also shown in section 2 of \cite{Crew17*} that if $u : Q \to P$ is a locally noetherian morphism of formal schemes, then $\Omega^1_{Q/P}$ is a coherent sheaf and we have $\Omega^1_{Q/P} = 0$ if and only if $u$ is unramified.
This is also the correct condition in order to have a Jacobian criterion:

\begin{thm}[Crew]
Assume that $P \to S$ and $Q \to S$ are two locally noetherian morphisms of (locally noetherian) formal schemes.
\begin{enumerate}
\item If $u : Q \to P$ is an $S$-morphism, then there exists a right exact sequence
\[
u^*\Omega^1_{P/S} \to \Omega^1_{Q/S} \to \Omega^1_{Q/P} \to 0.
\]
Assume moreover that $Q$ is formally smooth over $S$.
Then $u$ is formally smooth (resp.\ \'etale) if and only if $u^*\Omega^1_{P/S}$ is locally a direct summand in (resp.\ is isomorphic to) $\Omega^1_{Q/S}$.
\item If $i : Q \hookrightarrow P$ is an $S$-immersion, then there exists a right exact sequence
\[
\check{\mathcal N}_{Q/P} \to i^*\Omega^1_{P/S} \to \Omega^1_{Q/S} \to 0.
\]
Assume moreover that $P$ is formally smooth over $S$.
Then $Q$ is formally smooth (resp.\ \'etale) over $S$ if and only if $\check{\mathcal N}_{Q/S}$ is is locally a direct summand in (resp.\ is isomorphic to) $i^*\Omega^1_{P/S}$.
\end{enumerate}
\end{thm}

\begin{proof}
This is proved in proposition 2.2.5 of \cite{Crew17*}.
\end{proof}

It is also shown in corollary 2.2.10 of \cite{Crew17*} that a locally noetherian morphism $u$ is formally smooth if and only if it is differentially smooth.

\begin{xmp}
We let $S$ be a formal scheme (that we assume to be affine when we consider a bounded affine space over $S$).
\begin{enumerate}
\item All the above examples of formal schemes are locally noetherian (as long as the base $S$ is so).
\item
The formal scheme $\mathbb A^b_{S} \times_{S} \mathbb A^b_{S}$ is \emph{not} locally noetherian in general: the ring $A[[T]] \widehat \otimes_{A} A[[T]]$, where completion is meant with respect to the topology of $A$, is not noetherian in general (e.g. $A$ a big field with the discrete topology).
As a consequence, the morphism of (noetherian) formal schemes $\mathbb A^b_{S} \to S$ is not locally noetherian in general.
\item
Assume $S = \mathrm{Spf}(\mathcal V)$ where $\mathcal V$ is a discrete valuation ring with perfect residue field $k$ of characteristic $p > 0$.
Then the formal scheme $\mathbb A^b_{S}$ is \emph{not} smooth over $S$ but it is differentially smooth.
The same thing happens more generally in the case $S = \mathrm{Spf}(A)$ where $A$ has the $p$-adic topology and the absolute Frobenius of $A/pA$ is finite.
\item 
The formal scheme $\mathbb A^-_{S}$ is differentially smooth over $S$.
In particular, it is formally smooth.
Note however that this is not a smooth formal scheme over $S$ because it is not even adic over $S$.
\item
Of course, the formal scheme $\mathbb A_{S}$ is smooth over $S$.
\end{enumerate}
\end{xmp}

\subsection{Partially proper maps}

We introduce here the notion of a partially proper map of formal schemes and more generally discuss some properties that are stable under deformation.
Recall that all adic formal schemes are now assumed to be locally noetherian and simply called formal schemes.

If $P$ is a (locally noetherian) formal scheme, then there exists a unique reduced closed \emph{subscheme} $P_{\mathrm{red}} \subset P$ having the same underlying space as $P$ (in other words, there exists a biggest ideal of definition).
This provides an adjoint to the embedding of the category of locally noetherian (usual) schemes into the category of locally noetherian formal schemes.

When we consider a morphism of formal schemes $u \colon Q \to P$, we may wonder how some properties of the morphism of schemes $u_{\mathrm{red}}\colon Q_{\mathrm{red}} \to P_{\mathrm{red}}$ might impact the morphism $u$.
For example, $u$ is \emph{separated} or \emph{affine} if and only if $u_{\mathrm{red}}$ is separated or affine.
Also, $u$ is locally noetherian if and only if $u_{\mathrm{red}}$ is (and this last condition may be tested with usual noetherian schemes).

\begin{dfn}
A morphism of formal schemes $u \colon Q \to P$ is said to be
\begin{enumerate}
\item
\emph{formally (locally) of finite type} (resp.\ \emph{formally (locally) quasi-finite}) if $u_{\mathrm{red}}$ is (locally) of finite type (resp.\ (locally) quasi-finite).
\item
\emph{partially proper} (resp.\ \emph{partially finite}) if $u$ is formally locally of finite type and induces a proper (resp.\ finite) morphism of schemes $Z_{\mathrm{red}} \to T_{\mathrm{red}}$ whenever $Z$ and $T$ are irreducible components of $Q$ and $P$ such that $u(Z) \subset T$.
\end{enumerate}
\end{dfn}

One may also call a locally closed immersion $u$ a \emph{formal thickening} when $u_{\mathrm{red}}$ is an isomorphism, in which case all the above properties are satisfied.

For an \emph{adic} morphism $u$, being formally (locally) of finite type or quasi-finite, is equivalent to being (locally) of finite type or quasi-finite, but this is not the case in general (see examples below).
When $u$ is quasi-compact, being partially proper or finite is equivalent to $u_{\mathrm{red}}$ being proper or finite.
If moreover, $u$ is adic, this is equivalent to $u$ itself being proper or finite.
In other words, these notions are relevant only in the case of non-adic maps.

It is important to notice that a morphism which is locally formally of finite type is automatically locally noetherian.
In particular, if the morphism is also formally smooth, then it is automatically differentially smooth.

\begin{xmp}
\begin{enumerate}
\item
If $S$ is any formal scheme, then the structural map $\mathbb{A}^-_{S} \to S$ is a formal thickening.
This provides an example of a morphism which is partially finite, and therefore partially proper and formally quasi-finite, and in particular formally of finite type.
However, this morphism is \emph{not} locally of finite type (and not locally quasi-finite, not proper, not finite either) because it is not even adic.
\item
Classically, the formal model of $\mathbb G_{m}$ whose reduction is an infinite union of projective lines, each meeting one another in exactly one point, is a partially proper formal scheme.
However, it is not quasi-compact.
\end{enumerate}
\end{xmp}

When we write $\mathbb A^{\pm,N}$, we mean a product of $N$ copies of $\mathbb A$ or $\mathbb A^-$.
Of course, they may be reordered so that
\[
\mathbb A^{\pm,N} = \mathbb A^{n} \times \mathbb A^{-,m}
\]
with $N = n + m$.
In other words, if $S = \mathrm{Spf}(A)$, we have
\[
\mathbb A^{\pm,N}_{S} = \mathrm{Spf}(A[T_{1}, \ldots, T_{n}, U_{1}, \ldots, U_{m}])
\]
with the $I[T_{1}, \ldots, T_{n}] + (U_{1}, \ldots, U_{m})$-topology.
Alternatively, we can write
\begin{align*}
\mathbb A^{\pm,N}_{S} & = \mathrm{Spf}(A\{T_{1}, \ldots, T_{n}\}[[U_{1}, \ldots, U_{m}]])
\\
&= \mathrm{Spf}(A[[U_{1}, \ldots, U_{n}]]\{T_{1}, \ldots, T_{n}\}).
\end{align*}

\begin{lem}
A morphism of formal schemes $u \colon Q \to P$ is formally locally of finite type if and only if it factors, locally on $P$ and $Q$, through a closed immersion $Q \hookrightarrow \mathbb A^{\pm,N}_{P}$.
\end{lem}

\begin{proof}
The condition is clearly sufficient and, in order to prove that it is necessary, we may assume that $u$ comes from a continuous morphism of noetherian adic rings $A \to B$ with $B$ complete.
If $J$ is the biggest ideal of definition of $B$, then our hypothesis tells us that the composite map $A \to B \to B/J$ is of finite type and extends therefore to a surjective map $A' := A[T_{1}, \ldots, T_{n}] \twoheadrightarrow B/J$.
Assume that $J$ is generated by $g_{1}, \ldots, g_{m}$.
Then, since $B$ is complete, there exists a unique continuous map $A'' := A'[[U_{1}, \ldots, U_{m}]] \to B$ that sends $U_{i}$ to $g_{i}$ for all $i = 1, \ldots, m$.
After completion, this map is surjective because it is adic and surjective onto $B/I''B$ where $I''$ is the biggest ideal of definition of $A''$.
\end{proof}

As a particular case of the former lemma, we see that a formal scheme is formally locally of finite type over a discrete valuation ring if and only if it is \emph{special} in the sense of Berkovich (\cite{Berkovich96b}).

\begin{dfn}
A morphism of formal schemes $u \colon Q \to P$ is said to satisfy the \emph{valuative criterion} for properness if, given a field $F$ and a valuation ring $R$ of $F$, then any commutative diagram
\[
\xymatrix{Q \ar[r] & P \\ \mathrm{Spec}(F) \ar[r] \ar[u] & \mathrm{Spec}(R) \ar[u] \ar@{-->}[ul]}
\]
may be uniquely completed by the dotted arrow.
\end{dfn}

When $Q = \mathrm{Spf}(B)$ is affine, the lifting property simply means that the image of $B \to F$ is contained in $R$.

\begin{prop}
Let $u : Q \to P$ be morphism of formal schemes.
Then $u$ is partially proper if and only if it is formally locally of finite type and satisfies the valuative criterion for properness.
\end{prop}

\begin{proof}
This follows directly from the valuative criterion of properness for usual schemes because properness is not sensitive to nilpotent immersions (and $F$ and $R$ have the discrete topology).
\end{proof}

Let us finish by recalling the notion of a blowing-up in the setting of formal schemes.
If $i : X \hookrightarrow P$ is a closed immersion of formal schemes, then the \emph{ideal of $X$ in $P$} is the kernel $\mathcal I_{X}$ of the canonical map $\mathcal O_{P} \to i_{*}\mathcal O_{X}$.
When $X$ is a \emph{usual} scheme, then there exists an \emph{adic} morphism $u : P' \to P$ which is universal for making $\mathcal I_{X}\mathcal O_{P'}$ an invertible ideal of $\mathcal O_{P'}$.
This is the \emph{blowing up} of $P$ centered at $X$.
A blowing-up is always a proper map and it induces an isomorphism between the open complement of $X' := u^{-1}(X)$ and the open complement of $X$.

\begin{xmp}
Assume $X := \{0\}$ is the origin in the special fiber of $P := \widehat {\mathbb A}_{\mathbb Z_{p}}$ (with the $p$-adic topology).
In other words, $P = \mathrm{Spf}(\mathbb Z_{p}\{T\})$ and $X = \mathrm V(T,p)$.
Then
\[
P' := \widehat{\mathrm{Proj}}(\mathbb Z_{p}[T][U, V]/TU-pV) \quad \mathrm{and} \quad X' := \mathrm{Proj}(\mathbb F_{p}[U,V]) \simeq \mathbb P_{\mathbb F_{p}}.
\]
The formal scheme $P'$ has an affine covering $P' = P'_{1} \cup P'_{2}$ with
\[
P'_{1} := \mathrm{Spf}(\mathbb Z_{p}\{T,U\}/TU-p) \quad \mathrm{and} \quad P'_{2} := \mathrm{Spf}(\mathbb Z_{p}\{T,V\}/T-pV)
\]
and $X'= X'_{1} \cup X'_{2}$ where $X'_{1} = \mathrm{Spec}(\mathbb F_{p}[U]) \simeq \mathbb A_{\mathbb F_{p}}$ and $X'_{2} := \mathrm{Spec}(\mathbb F_{p}[V]) \simeq \mathbb A_{\mathbb F_{p}}$.
\end{xmp}

\section{Adic spaces}

In this section, we give a brief description of the theory of adic spaces and some standard examples.
We mostly follow the notes of Wedhorn in \cite{Wedhorn19} but the vocabulary is also influenced by Scholze's lectures in Berkeley (\cite{Scholze*}) and Conrad's lectures in Stanford (\cite{Conrad*}) (see also the more recent notes by Sophie Morel (\cite{Morel*}) and Judith Ludwig (\cite{Ludwig*}).
Anyway, almost everything can be found in the original papers \cite{Huber96} and \cite{Huber93}.
We do not follow the fashion and stick to the additive notation from Huber's thesis \cite{Huber93b}.

\subsection{Valuations}

The main idea in the theory of adic spaces consists in using valuations instead of prime ideals in order to describe the physical points.

\begin{dfn}
If $A$ is any ring and $G$ is a totally ordered \emph{additive} group, a \emph{valuation}
\[
v \colon A \to G \cup \{+\infty\}
\]
is a map that satisfies
\begin{enumerate}
\item $v(1) = 0,\quad v(0) = +\infty$,
\item $\forall f, g \in A, \quad v(fg) = v(f) + v(g)$,
\item $\forall f, g \in A, \quad v(f+g) \geq \min(v(f), v(g))$.
\end{enumerate}
\end{dfn}

The \emph{height} of $v$ is the height of the subgroup $G_{v}$ generated by the image of $v$ (supremum length of a chain of convex subgroups).
We call $v$ \emph{trivial} if $G_{v} = \{0\}$ and \emph{discrete} if $G_{v} \simeq \mathbb Z$.

Two valuations $v$ and $v'$ on the same ring $A$ are said to be \emph{equivalent} if they define the same preorder on $A$:
\[
\forall f, g \in A, \quad v(f) \leq v(g) \Leftrightarrow v'(f) \leq v'(g)
\]
(or equivalently, if they both factor through the same valuation).

\begin{xmp}
\begin{enumerate}
\item
Up to equivalence, a \emph{trivial} valuation corresponds to a unique prime $\mathfrak p \subset A$:
\[
v_{\mathfrak p}(f) = 0 \Leftrightarrow f(\mathfrak p) \neq 0 \quad \mathrm{and} \quad v_{\mathfrak p}(f) = + \infty \Leftrightarrow f(\mathfrak p) = 0.
\]
Recall from our conventions that $f(\mathfrak p)$ denotes the image of $f$ in $\kappa(\mathfrak p) := \mathrm{Frac}(A/\mathfrak p)$ so that the conditions also read $f \not \in \mathfrak p$ and $f \in \mathfrak p$ respectively.
\item
If $v$ is a valuation on a ring $A$, then the \emph{Gauss valuation} (associated to $v$) on $A[T]$ is defined for $F := \sum_{i=0}^d f_{i} T^i \in A[T]$ by
\[
v(F) = \min_{i=0}^d v(f_{i}) \in G \cup \{+ \infty\}.
\]
It has the same height as the original valuation (geometrically, it corresponds to the maximal point of the unit disk).
\item
With the same notations, one may always define valuations of higher height by
\[
v^-(F) = \min_{i=0}^d (v(f_i), i) = \left(v(F), \min_{v(F)=v(f_{i})}{i}\right) \in (G \times \mathbb Z) \cup \{+ \infty\}
\]
or
\[
v^+(F) = \min_{i=0}^d (v(f_i),-i) = \left(v(F), -\max_{v(F)=v(f_{i})}{i}\right) \in (G \times \mathbb Z) \cup \{+ \infty\},
\]
where $G \times \mathbb Z$ has the lexicographical order (geometrically, they correspond to the specializations of the maximal point inside and outside the unit disk).
When we start from a trivial valuation on an integral domain, then $v^- = \mathrm{val}$ (standard valuation on polynomials) and $v^+ = - \deg$ (where $\deg$ denotes the degree map on polynomials).
\end{enumerate}
\end{xmp}

The \emph{support} of a valuation $v$ on $A$ is the prime ideal 
\[
\mathrm{supp}\,(v) := \{f \in A \colon v(f) = + \infty\}.
\]
The residue field of $v$ is $\kappa(v) := \kappa(\mathrm{supp}\,v)$ and the image of $f \in A$ in $\kappa(v)$ will be denoted by $f(v)$ ($= f(\mathrm{supp}\, v)$).

Some authors favor the multiplicative notation for valuations.
We will do that only for fields (but this shouldn't create any confusion).
In general, one can move back and forth between the additive and multiplicative notation by writing formally
\[
\Gamma := \exp (-G) = \{\exp(-g) \colon g \in G\} \quad \mathrm{and} \quad G := -\ln(\Gamma) = \{-\ln \gamma \colon \gamma \in \Gamma\}.
\]

\begin{dfn}
If $K$ is a field and $\Gamma$ is a totally ordered (commutative) \emph{multiplicative} group (of any height), then an \emph{(ultrametric) absolute value\footnote{Also called \emph{multiplicative semi-norm}.}}
\[
|-| \colon K \to \{0\} \cup \Gamma,
\]
is a map that satisfies
\begin{enumerate}
\item $\forall a \in K, \quad |a| = 0 \Leftrightarrow a= 0$,
\item $\forall a,b \in K, \quad |ab| = |a||b|$,
\item $\forall a,b \in K, \quad |a+b| \leq \max(|a|, |b|)$.
\end{enumerate}
\end{dfn}

Besides the multiplicative notation and reversing order, this is nothing but a valuation on $K$ and we will apply systematically the vocabulary of valuations to ultrametric absolute values.
Note that we do \emph{not} require an absolute value to have height at most one (this is \emph{not} standard).
This will not matter much later when $K$ is a Huber field because there will then exist an absolute value of height at most one on $K$ that induces the same topology.

When $K$ is endowed with an ultrametric absolute value (of any height), it is called a \emph{valued field} and
\[
K^+ := \{a \in K \colon |a| \leq 1\}
\]
is its \emph{valuation ring}.
Up to equivalence, $|-|$ is uniquely determined by $K^+$, the value group $\Gamma_{|-|}$ is isomorphic to $K^\times/K^{+\times}$ and the height of $|-|$ is the same thing as the (Krull) dimension of $K^+$.
The ultrametric absolute value induces a topology on $K$ with basis of open subsets given by the open ``disks''
\[
D(c,\gamma) = \{a \in K \colon \quad |a -c| < \gamma\}
\]
(this is the topology inherited from $\Gamma$ - see below).

Any valuation $v$ on a ring $A$ will induce an ultrametric absolute value of the same height $| - |$ on $\kappa(v)$ (and conversely).
The formulas are simply given by
\[
|f(v)| = \exp(-v(f)) \quad \mathrm{and} \quad v(f) = -\ln(|f(v)|).
\]

\begin{xmp}
\begin{enumerate}
\item
The equivalence between prime ideals and trivial valuations sounds more natural in term of absolute value:
\[
|f(v_{\mathfrak p})|= 0 \Leftrightarrow f(\mathfrak p) = 0 \quad \mathrm{and} \quad |f(v_{\mathfrak p})| \neq 0 \Leftrightarrow f(\mathfrak p) \neq 0.
\]
\item
The formula for the Gauss valuation also sounds more familiar: if $F := \sum_{i=0}^d f_{i} T^i \in A[T]$, then
\[
|F(v)| = \max_{i=0}^d |f_{i}(v)|.
\]
\item
Finally the inner and outer valuations read
\[
|F(v^-)| = \max_{i=0}^d |f_i(v)|\rho_-^i \quad \mathrm{and} \quad |F(v^-)| = \max_{i=0}^d |f_i(v)|\rho_+^i 
\]
with the convention $\epsilon < \rho_- < 1 < \rho_+ < \lambda$ whenever $\epsilon < 1 < \lambda$ in $\mathbb R$.
\end{enumerate}
\end{xmp}

\subsection{Huber rings}

The notion of an $f$-adic ring was introduced by Huber and will be called here a Huber ring.
This is a generalization of the notion of an affinoid algebra and a Huber ring is also sometimes called an affinoid ring.
More precisely:

\begin{dfn}
A topological ring $A$ is called a \emph{Huber ring} (an $f$-adic ring in the sense of Huber) if there exists an open adic subring $A_{0} \subset A$.
We then call $A_{0}$ a \emph{ring of definition} for $A$ and any ideal of definition $I_{0}$ of $A_{0}$ will be also called an \emph{ideal of definition} for $A$.
A ring homomorphism $A \to B$ between two Huber rings is called \emph{adic} if it induces an adic map between some rings of definition.
\end{dfn}

When $A$ is a Huber ring, then the completion $\widehat A$ of $A$ is also a Huber ring.
Moreover $\widehat A_{0}$ is a ring of definition for $\widehat A$ and we have $\widehat A = A \otimes_{A_{0}} \widehat A_{0}$.
In particular, there is no harm in assuming that a Huber ring is complete but there is a loss in flexibility.

By definition, a morphism of Huber rings $A \to B$ is a continuous homomorphism and it will always send a ring (resp.\ an ideal) of definition into some ring (resp.\ some ideal) of definition.
We will make it clear when we assume that the morphism is actually adic (which is, as this was the case for adic rings, a much stronger condition).

It may be noticed that $A$ has the $I_{0}$-adic topology as an $A_{0}$-module.
This may be confusing in the important case when $I_{0} = (\pi)$ is a principal ideal in $A_0$ and we just say ``$\pi$-adic'' where we mean $\pi A_{0}$-adic and not $\pi A$-adic.
But this is a classical issue since we talk about the ``$p$-adic'' topology of $\mathbb Q_p$ although $\mathbb Q_p$ is not an adic ring.

\begin{dfn}
Let $A$ be a Huber ring.
Then,
\begin{enumerate}
\item $A$ is said to be of \emph{of noetherian type} if $A$ is finitely generated over some noetherian ring of definition.
\item $A$ is called a \emph{Tate ring} if there exists a topologically nilpotent unit $\pi \in A$.
\end{enumerate}
\end{dfn}

Note that if $A$ is of noetherian type or if $A$ is a Tate ring, so is $\widehat A$.

\begin{xmp}
\begin{enumerate}
\item An adic ring $A$ is always a Huber ring which is its own ring of definition.
It is of noetherian type if and only if $A$ is noetherian.
It cannot be a Tate ring unless $\widehat A$ is the zero ring.
As a particular case, a usual ring $A$ (endowed with the discrete topology) is a Huber ring.
\item If $A_{0}$ is a $\pi$-adic ring for some regular $\pi \in A_{0}$, then $A := A_{0}[1/\pi]$ is a Tate ring with ring of definition $A_{0}$.
And conversely, if $A$ is a Tate ring with ring of definition $A_{0}$, then there exists $\pi \in A_{0}$ such that $A_{0}$ is a $\pi$-adic ring and $A = A_{0}[1/\pi]$.
The Tate ring $A$ is of noetherian type if (but not only if) $A_{0}$ is noetherian.
\item Let $(K, |-|)$ be a valued field (of some height).
When $K$ is a Huber ring (for the topology induced by $|-|$), we call $(K,|-|)$ a \emph{Huber valued field}.
Then, either $|-|$ is trivial or $K$ is a Tate ring in which case, we say \emph{Tate valued field}.
Now, a topological field $K$ is said to be \emph{non archimedean} if its topology \emph{may be} defined by some absolute value of height one.
Then, $(K, |-|)$ is a Tate valued field if and only if $K$ is non-archimedean (but the original absolute value $ |-|$ maybe of higher rank).
A non archimedean field is of noetherian type if and only if the topology may be defined by a \emph{discrete} valuation.
\item
Let $A$ be a Huber ring with ring of definition $A_{0}$ and ideal of definition $I_{0}$.
If $(f_{1}, \ldots, f_{r})$ is an \emph{open} ideal in $A$ and $k \in \mathbb N$, then the polynomial ring $A[T]$ is a Huber ring for the topology defined by the subring $A_{0}[f_{1}T^k, \ldots, f_{r}T^k]$ endowed with the ideal of definition $I_{0}[f_{1}T^k, \ldots, f_{r}T^k]$.
Openness condition is necessary for $A[T]$ to be a topological ring (for this topology).
Geometrically, this corresponds to a closed polydisc.
This construction preserves the noetherian type and the Tate properties.
\end{enumerate}
\end{xmp}

If $A$ is a Huber ring, we will denote by $A^\circ$ the subring of power bounded elements of $A$ (union of all rings of definition) and by $A^{\circ\circ}$ the ideal of topologically nilpotent elements in $A^{\circ}$ (union of all ideals of definition).

\subsection{Huber pairs}

A tricky ingredient in Huber theory is that it is necessary to work with pairs of rings.
Actually, a positivity condition will be necessary in order to get rid of unwanted points.

\begin{dfn}
A \emph{Huber pair} (an affinoid ring in the sense of Huber) is a pair $(A, A^+)$ where $A$ is a Huber ring and $A^+$ is \emph{any} subset $A^{\circ}$.
The integral closure $\overline A^+$ of the subring generated by $A^+$ and $A^{\circ\circ}$ is called the \emph{ring of integral elements}. 
\end{dfn}

As we will see below, there is no harm in replacing $A^+$ with $\overline A^+$ (as Huber systematically does).
There is however more flexibility in not making such an assumption.
We will call the Huber pair $(A, A^+)$ \emph{complete} if $A$ is complete and $A^+$ is the full ring of integral elements ($A^+ = \overline A^+$).
We will also call the Huber pair \emph{rigid} when $\overline A^+ = A^\circ$ is the whole subring of power-bounded elements.
A Huber pair $(A, A^+)$ is said to be of \emph{noetherian type} (resp.\ is called a \emph{Tate pair}) when $A$ is of noetherian type (resp.\ is a Tate ring).
Note that $A^+$ plays no role in this last definition.

\begin{xmp}
\begin{enumerate}
\item If $A$ is any Huber ring, then both $(A,A^\circ)$ and $(A, \emptyset)$ are Huber pairs.
The ring of integral elements is $A^\circ$ in the first case and this is the integral closure of the subring $\mathbb Z \cdot 1_{A} + A^{\circ\circ} \subset A$ in the second case.
\item If $A$ is an adic ring, and this includes usual rings (with the discrete topology), then $(A,A)$ is a rigid Huber pair.
\item If $A_{0}$ is a $\pi$-adic ring for some regular $\pi \in A_{0}$ and $A := A_{0}[1/\pi]$, then $(A, A_{0})$ is a Tate pair.
\item Let $(K, |-|)$ be a non trivially valued field with valuation ring $K^+$.
Then $(K, K^+)$ is a Huber pair if and only if $(K, |-|)$ is a Tate value field.
This happens exactly when $K^+$ has a prime ideal of height one ($|-|$ is \emph{microbial} in Huber's words).
Then, $K^+$ is exactly the ring of integral elements but $K^+ \neq K^\circ$ in general and $(K, K^+)$ is not rigid.
Actually, $K^\circ$ is a valuation ring that induces an absolute value of height one which defines the original topology of $K$.
The ring $K^+$ is not noetherian unless the absolute value $|-|$ is discrete, but $(K, K^+)$ will be of noetherian type whenever the topology of $K$ \emph{may} be defined by a discrete valuation, or equivalently when $K^\circ$ is a discrete valuation ring.
\item 
In the case $A = \mathbb Z[1/p]$ (with the discrete topology), there exists essentially two Huber pairs, with rings of integral elements $\mathbb Z[1/p]$ and $\mathbb Z$ respectively.
They both play a role in the theory.
\item Assume that $(A, A^+)$ is a Huber pair, that $A[T]$ has the topology induced by $A_{0}[f_{1}T^k, \ldots, f_{r}T^k]$ for some $f_{1}, \ldots, f_{r}$ generating an open ideal in $A$ and ring of definition $A_{0}$.
Then, we will usually choose $A[T]^+ = A^+ \cup \{f_{1}T^k, \ldots, f_{r}T^k\}$.
\end{enumerate}
\end{xmp}

\subsection{Adic spectra}

Alike the wave-particle duality in physics, there exists a valuation-ideal alternative when we wish to describe the objects of arithmetic geometry.
The former gives rise to the adic spectrum and the latter to the prime spectrum.

If $G$ is a totally ordered additive group, then $G \cup \{+\infty\}$ is endowed with the topology whose open subsets $U$ are defined by the conditions
\[
+ \infty \not\in U\quad \mathrm{or}\quad \exists M \in G, ]M, +\infty[ \subset U.
\]
Note that we have the alternative $ng \to +\infty$ or $ng$ is bounded but the condition $g >0$ does not imply that $ng \to +\infty$ in general (it is possible that $ng$ is bounded even if $g >0$).
Note also that the topology introduced above on a valued field $K$ is the topology inherited by $G$ (coarsest topology making the absolute value continuous).
As a consequence, we see that a valuation $v$ on a topological ring $A$ is continuous if and only if the evaluation map $A \to \kappa(v)$ is continuous.

\begin{dfn}
The \emph{adic spectrum} of a Huber pair $(A, A^+)$ is the set $V := \mathrm{Spa}(A, A^+)$ of equivalence classes of continuous valuations on $A$ that are non-negative on $A^+$.
\end{dfn}

We always have
\[
\overline A^+ = \{f \in A \colon \forall v \in V, v(f) \geq 0\} \quad \mathrm{and}\quad V = \mathrm{Spa}(A, \overline A^+).
\]
This is why we may usually replace $A^+$ with $\overline A^+$ and assume that $A^+$ is the full ring of integral elements.
We also have
\[
 \mathrm{Spa}(A, A^+) = \mathrm{Spa}(\widehat A, \mathrm{image}\ \mathrm{of}\ A^+\ \mathrm{in}\ \widehat A)
\]
and we may therefore replace in practice $A$ with $\widehat A$ and assume that $A$ is complete.

If $A$ is a Huber ring, then there exists a smallest and a biggest adic spectrum associated to $A$ which are given by
\[
\mathrm{Spa}(A) := \mathrm{Spa}(A,A^{\circ}) \quad \mathrm{and} \quad \mathrm{Cont}(A) := \mathrm{Spa}(A, \emptyset).
\]
This notation will also be used when $A$ has the \emph{discrete} topology in which case $\mathrm{Spa}(A)$ denotes the set of all non-negative valuations on $A$, but we would then rather write $\mathrm{Spv}(A)$ instead of $\mathrm{Cont}(A)$ for the set of all valuations (which are automatically continuous in this situation).

In general, the adic spectrum $V := \mathrm{Spa}(A, A^+)$ is endowed with the topology for which a basis of open subsets is given by the \emph{rational} subsets
\[
R(f_{1}/f_{0}, \ldots, f_{r}/f_{0}) = 
\{v \in V \colon \forall i = 1, \ldots, n, v(f_{i}) \geq v(f_{0}) \neq + \infty \}
\]
in which $(f_{0}, \ldots, f_{r})$ is an open ideal.
Note that the condition that $(f_{0}, \ldots, f_{r})$ is \emph{open} is not necessary in order to obtain an open subset of $V$ (although it would not be called \emph{rational} anymore).
As a consequence, the topology of $V$ only depends on the ring $A$ (and not the topology of $A$ or the choice of $A^+$) although the points do.

\begin{xmp}
\begin{enumerate}
\item
If $K$ is any field (endowed with the discrete topology) and $k$ is any subring of $K$, then $\mathrm{Spa}(K, k)$ is the \emph{Riemann-Zariski space} of $K$ over $k$.
As a standard example, note that $\mathrm{Spa}(\mathbb C(X), \mathbb C) \simeq \mathbb P_{\mathbb C}$ (more about this later).
\item
For $n \in \mathbb Z$, set
\[
0(n) = \left\{\begin{array}{ll} 0 & if\ n \neq 0 \\ + \infty & if\ n = 0 \end{array}\right.;
\]
If $p$ is a prime, denote by $v_{p}$ the usual $p$-adic valuation, and set
\[
p(n) = \left\{\begin{array}{ll} 0 & if\ p \nmid n \\ + \infty & if\ p \mid n\end{array}\right..
\]
Then, we have
\[
\mathrm{Spv}(\mathbb Z) = \{0, v_{p} \ \mathrm{for}\ p\ \mathrm{prime}, p\ \mathrm{for}\ p\ \mathrm{prime}\}.
\]
Any proper closed subset of $\mathrm{Spv}(\mathbb Z)$ is a finite union of subsets of the form $\{p\}$ or $\{v_{p}, p\}$.
\end{enumerate}
\end{xmp}

We say that the Huber pair $(A, A^+)$, or that $V := \mathrm{Spa}(A, A^+)$, is \emph{sheaffy} if there exists a (necessarily unique) structural sheaf $\mathcal O_{V}$ for which (up to a canonical isomorphism)
\[
\Gamma(A(f_{1}/f_{0}, \ldots, f_{r}/f_{0}) , \mathcal O_{V}) = \widehat {A[1/f_{0}]},
\]
where $A[1/f_{0}]$ is the Huber ring whose definition ring is $A_{0}[f_{1}/f_{0}, \ldots, f_{r}/f_{0}]$ (and completion is meant relatively to this structure) for some ring of definition $A_{0}$ of $A$.
When this is the case, $V := \mathrm{Spa}(A, A^+)$ is a topologically valued ringed space (see item \eqref{valsp} in notations/conventions).

If $V = \mathrm{Spa}(A,A^+)$ is a sheaffy adic spectrum, then the inclusion of valued fields
\[
\mathrm{Frac}(A/\mathrm{supp}\, v) \hookrightarrow \mathcal O_{V,v}/\mathfrak m_{V,v}
\]
(which is not an equality in general) induces an isomorphism on their \emph{completions} and we hope that our common notation $\kappa(v)$ for both fields will not create any confusion.
Following Berkovich, we will denote by $\mathcal H(v)$ their common completion.

\begin{xmp}
\begin{enumerate}
\item 
If $V = \mathrm{Spv}(\mathbb Q)$, then we have an isomorphism of locally ringed spaces $V^+ \simeq \mathrm{Spec}(\mathbb Z)$ (see item \eqref{plus} in notations/conventions for the definition of $V^+$).
\item
If $V = \mathrm{Spa}(\mathbb C(X), \mathbb C)$, then we have an isomorphism of locally ringed spaces $V^+ \simeq \mathbb P_{\mathbb C}$.
\item
There exists isomorphisms (of topologically valued ringed spaces)
\[
\mathrm{Spv}(\mathbb Z) \setminus \{p\} \simeq \mathrm{Spv}(\mathbb Z[1/p])\quad \mathrm{and}
\quad
\mathrm{Spv}(\mathbb Z) \setminus \{v_{p}, p\} \simeq \mathrm{Spa}(\mathbb Z[1/p]).
\]

\end{enumerate}
\end{xmp}

It is important to mention that, when $A$ is of noetherian type, then $V$ is automatically sheaffy.
There are actually many various other conditions that would insure $V$ to be sheaffy.
This is the case for example when $A$ has the discrete topology.

When $A$ is of noetherian type, one can also prove Cartan's theorems A and B as in the case of formal schemes:
\begin{thm}[Huber]
If $A$ is of noetherian type and $V = \mathrm{Spa}(A, A^+)$, then the functors
\[
M \mapsto \mathcal O_{V} \otimes^{\mathrm L}_{\widehat A} M \quad \mathrm{and} \quad \mathcal F \mapsto \mathrm R\Gamma(V, \mathcal F)
\]
induce an equivalence between finite $\widehat A$-modules and coherent $\mathcal O_{V}$-modules.
\end{thm}

\begin{proof}
This is shown in \cite{Huber93}.
\end{proof}

Again, the same result also holds when $A$ has the discrete topology.
\subsection{Specialization}

An important theorem of Huber states that an adic spectrum $V$ is a spectral space (coherent and sober).
Let us say a few words about specialization/generalization on adic spectra (we send the reader back to the section notations/conventions after the introduction for a review of the basics).

\begin{lem}
Let $(A, A^+)$ be a Huber pair and $v, w \in V = \mathrm{Spa}(A, A^+)$.
Then $w \rightsquigarrow v$ if and only the valuations induce a morphism of ordered groups $G_{v} \to G_{w}$ such that 
\begin{equation} \label{infco}
\forall f, g \in A, \quad v(g) \leq v(f) = +\infty \Rightarrow w(g) \leq w(f).
\end{equation}
\end{lem}

\begin{proof}
Recall that, by definition, $w \rightsquigarrow v$ if and only if any neighborhood of $v$ is also a neighborhood of $w$.
It means that
\begin{equation} \label{consp}
\forall f, g \in A, \quad v(f) \geq v(g) \neq +\infty \Rightarrow w(f) \geq w(g) \neq +\infty.
\end{equation}
Thus, we see that (by symmetry) if $v(f) = v(g) \neq +\infty$, then necessarily $w(f) = w(g) \neq +\infty$ and there exists therefore a well defined map
\[
G_{v} \cap v(A) \to G_{w} \cap w(A), \quad v(f) \mapsto w(f).
\]
The condition also shows that this map preserves the order.
Moreover, since we always have $v(fg) = v(f) + v(g)$ and $w(fg) = w(f) + w(g)$, this map extends uniquely to a group homomorphism which is automatically order preserving because the group laws are compatible with the orders.
Conversely, if the valuations induce a morphism of ordered groups, then condition \eqref{consp} is automatically satisfied whenever $v(f) \neq + \infty$ and only condition \eqref{infco} had to be checked.
\end{proof}

Note that the support map is compatible with specialization: we have
\[
w \rightsquigarrow v \Rightarrow \mathrm{supp}\, w \rightsquigarrow \mathrm{supp}\, v
\]
(which means that $\mathrm{supp}\, w \subset \mathrm{supp}\, v$).
There are two special kinds of specializations:

\begin{dfn}
Let $V = \mathrm{Spa}(A,A^+)$ be an adic spectrum and $v,w \in V$.
Then,
\begin{enumerate}
\item $v$ is a \emph{horizontal} (or primary) specialization of $w$ if the valuations induce an injective morphism of ordered groups $G_{v} \hookrightarrow G_{w}$ and
\[
\forall f,g \in A, \quad v(g) < v(f) = +\infty \Rightarrow w(g) < w(f).
\]
\item $v$ is a \emph{vertical} (or secondary) specialization of $w$ if the valuations induce a surjective morphism of ordered groups $G_{v} \twoheadrightarrow G_{w}$ and
\[
\forall f \in A, \quad v(f) = +\infty \Rightarrow w(f) = + \infty.
\]
\end{enumerate}
\end{dfn}

Note that a specialization $w \rightsquigarrow v$ is \emph{vertical} if and only if $\mathrm{supp}\, w = \mathrm{supp}\, v$ (it lives in a fiber of the support map).
On the other hand, given an inclusion $\mathrm{supp}\, w \hookrightarrow \mathfrak p$ where $\mathfrak p$ is some prime ideal in $A$, there will exist at most one horizontal specialization $w \rightsquigarrow v$ with $\mathrm{supp}\, v = \mathfrak p$.
A basic (but not trivial) theorem of the theory states that any specialization is the composition of a vertical one and a horizontal one (in that order).

\begin{xmp}
\begin{enumerate}
\item
In $\mathrm{Spv}(\mathbb Z)$ we have the following (non trivial) vertical and horizontal specializations (for any prime $p$)
\[
\begin{array}{ccc} 0 \\ \downsquigarrow \\ v_{p} & \rightsquigarrow & p. \end{array}
\]
\item
If $A$ is a Tate ring, then all specializations are vertical and any valuation $v$ has a unique (vertical) generalization $w$ of height 1.
\item
Recall that if $K$ is a field, then there exists a bijection between $\mathrm{Spv}(K)$ and the set of valuation rings of $K$.
If $v$ corresponds to $\mathcal V$ (so that $v(a) \geq 0 \Leftrightarrow a \in \mathcal V$), then the (automatically vertical) generalizations of $v$ correspond bijectively to the subrings of $K$ containing $\mathcal V$.
There also exists a recursive way to describe the specializations of $v$: if $k$ denotes the residue field of $\mathcal V$, then pulling back along $\mathcal V \to k$ induces a bijection between valuation rings of $k$ and valuation rings of $K$ corresponding to specializations of $v$.
\end{enumerate}
\end{xmp}

For further use, we need to show some elementary properties of specialization:

\begin{lem} \label{prlem}
Let $V = \mathrm{Spa}(A,A^+)$ be an adic spectrum, $v, w \in V$ and $f \in A$.
Assume that $w \rightsquigarrow v$ in $V$.
Then, we have
\begin{enumerate}
\item $w(f) = + \infty \Rightarrow v(f) = + \infty$, and conversely if the specialization is vertical,
\item $w(f^n) \to + \infty \Rightarrow v(f^n) \to + \infty$,
\item $w(f) > 0 \Rightarrow v(f) > 0$ and conversely if the specialization is horizontal.
\end{enumerate}
\end{lem}

\begin{proof}
The first assertion follows from the existence of the map $G_{v} \to G_{w}$ and the very definition of a vertical specialization.
For the second one, let us assume that the set $\{v(f^n)\}_{n \in \mathbb N}$ is bounded.
It means that there exists $\gamma_{v} \in G_{v}$ such that for all $n \in \mathbb N$ we have $v(f^n) \leq \gamma_{v}$.
Then, if we denote by $\gamma_{w} \in G_{w}$ the image of $\gamma_{v}$, we will have for all $n \in \mathbb N$, $w(f^n) \leq \gamma_{w}$ which means that the set $\{w(f^n)\}_{n \in \mathbb N}$ is also bounded.
The last implication follows from the fact that the map $G_{v} \to G_{w}$ always preserves the order and that it even strictly preserves the order when it is injective.
\end{proof}

We may also notice that we always have $v(f) > 0 \Leftrightarrow v(f^n) \to + \infty$ when $v$ has height as most one but it is important to notice that this is \emph{not} true anymore for higher height.

\subsection{Adic spaces}

Glueing sheaffy adic spectra will provide adic spaces:

\begin{dfn}
An \emph{adic space} (or Huber space) is a topologically valued ringed space $V$ which is locally isomorphic to some (sheaffy) adic spectrum $\mathrm{Spa}(A, A^+)$.
It is said to be \emph{affinoid} if it is actually isomorphic to some $\mathrm{Spa}(A, A^+)$.
\end{dfn}


Note that an adic space $V$ is a locally spectral space.

Alternatively, we could have defined an adic space as a doubly topologically locally ringed space $(V, \mathcal O_{V}, \mathcal O_{V}^+)$ locally isomorphic to an adic spectrum (see notations/conventions after the introduction).

The functor
\[
(A, A^+) \mapsto \mathrm{Spa}(A, A^+)
\]
is fully faithful on complete sheaffy pairs.
Better, there exists an adjunction
\[
\mathrm{Hom}(X, \mathrm{Spa}(A, A^+)) \simeq \mathrm{Hom}((A, A^+), (\Gamma(X, \mathcal O_{X}), \Gamma(X, \mathcal O_{X}^+))
\]
(compatible pairs of continuous homomorphisms on the right hand side).

It is important to notice that there is no fibered product of adic spaces in general and that, even when it exists, a (fibered) product of affinoid spaces is not necessarily affinoid ($\mathrm{Spv}(\mathbb Z[T]) \times \mathrm{Spa}(\mathbb Q_{p})$ for example).
Whenever, we write down a product of adic spaces, we implicitly assume that it is representable by an adic space.
This is the case in the examples below when $O$ is locally of noetherian type (definition \ref{noethtyp} below).

\begin{xmp}
If $O$ is an adic space, then we may consider:
\begin{enumerate}
\item the \emph{closed unit polydisc}:
\[
\mathbb D^n_{O} = \underbrace{\mathbb D \times \cdots \times \mathbb D}_{n\ \mathrm{times}} \times O
\]
in which $\mathbb D := \mathrm{Spa}(\mathbb Z[T])$ (with the discrete topology).
In the case $O = \mathrm{Spa}(A, A^+)$, we have
\[
\mathbb D^n_{O} = \mathrm{Spa}\left(A[T_{1}, \ldots, T_{n}], A^+[T_{1}, \ldots, T_{n}]\right)
\]
(with the topology coming from $A$).
\item the \emph{affine space}:
\[
\mathbb A^n_{O} = \underbrace{\mathbb A^{\mathrm{val}} \times \cdots \times \mathbb A^{\mathrm{val}}}_{n\ \mathrm{times}} \times O
\]
in which $\mathbb A^{\mathrm{val}} := \mathrm{Spv}(\mathbb Z[T])$.
Note that $O$ being affinoid will \emph{not} imply that $\mathbb A^n_{O}$ is affinoid.
\item the \emph{open unit polydisc}:
\[
\mathbb D^{n,-}_{O} := \underbrace{\mathbb D^- \times \cdots \times \mathbb D^-}_{n\ \mathrm{times}} \times O
\]
in which $\mathbb D^- := \mathrm{Spa}(\mathbb Z[T]) = \mathrm{Spa}(\mathbb Z[[T]])$ with the $T$-adic topology.
Again, the fact that $O$ is affinoid will \emph{not} imply that the open unit polydisc is affinoid.
\item the \emph{proper unit polydisc}:
\[
\overline{\mathbb D}^{n}_{O} := \underbrace{\overline {\mathbb D}_{O} \times_{O} \cdots \times_{O} \overline {\mathbb D}_{O}}_{n\ \mathrm{times}}
\]
in which $O = \mathrm{Spa}(A,A^+)$ is supposed to be affinoid and $\overline {\mathbb D}_{O} := \mathrm{Spa}(A[T], A^+)$ (with the topology coming from $A$).
Note that, when $O = \mathrm{Spv}(A)$ (with $A$ discrete), then $\overline{\mathbb D}^{n}_{O} = \mathbb A^n_O$.
\item the \emph{closed polydisc of radii $\{f_{i,j}^{-\frac 1k}\}$}:
\[
\mathbb D^n_{O}(0, \{f_{i,j}^{-\frac 1k}\}) = \mathbb D_{O}(0, \{f_{1,j}^{-\frac 1k}\}) \times_{O} \cdots \times_{O} \mathbb D_{O}(0, \{f_{n,j}^{-\frac 1k}\})
\]
where $\mathbb D_{O}(0, \{f_{j}^{-\frac 1k}\})$ denotes the closed disk of radius $\{f_{j}^{-\frac 1k}\}$ defined as
\[
\mathbb D_{O}(0, f_{1}^{-\frac 1k}, \ldots, f_{r}^{-\frac 1k}) = \left\{v \in \mathbb A_{O} \colon v(f_{1}T^k), \ldots, v(f_{r}T^k) \geq 0\right\}.
\]
In this definition, we assume that $O$ lives over some affinoid space $\mathrm{Spa}(A,A^+)$ and that $f_{1}, \ldots, f_{r} \in A$ generate an open ideal of $A$. 
In the case $O = \mathrm{Spa}(A, A^+)$, then a closed polydisc of some radii is always affinoid and
\[
\mathbb D_{O}(0, f_{1}^{-\frac 1k}, \ldots, f_{r}^{-\frac 1k}) = \mathrm{Spa}(A[T], A^+ \cup \{f_{1}T^k, \ldots, f_{r}T^k\})
\]
(with the topology induced by $A_{0}[f_{1}T^k, \ldots, f_{r}T^k]$ if $A_{0}$ denotes a ring of definition of $A$).

\end{enumerate}

The affine space is the union of all closed polydiscs of the same dimension.
Actually, if $O = \mathrm{Spa}(A, A^+)$ and $f_{1}, \ldots, f_{r}$ are some generators of an ideal of definition, then it is easily checked that
\[
\mathbb A_{O} := \bigcup_{n} \mathbb D_{O}\left(0, \{f_{\underline i}^{-1}\}_{|\underline i|=n}\right).
\]
Here, we use the usual multiindex notation $f_{\underline i} := f_{1}^{i_{1}}\cdots f_{r}^{i_{k}}$ and $|\underline i| := i_{1} + \cdots + i_{r}$ if $\underline i = (i_{1}, \ldots, i_{r}) \in \mathbb N^n$.

There exists no analogous result for $\mathbb D^{-}_{O}$ in general because $\mathbb D^{-}_{O}$ is not even open in $\mathbb D_{O}$.
However, if $O$ is an \emph{analytic} space (see below), then the open (resp.\ proper) unit polydisc is a union (resp.\ intersection) of closed polydiscs.
Actually, if $O$ is Tate with topologically nilpotent unit $\pi$, then we have
\[
\mathbb D^{-}_{O} := \bigcup_{k} \mathbb D_{O}(0, \pi^{\frac 1k}) \quad \left(\mathrm{resp.} \overline {\mathbb D}_{O} := \bigcap_{k} \mathbb D_{O}(0, \pi^{-\frac 1k})\right).
\]
Be careful however that $\mathbb D^{-}_{O}$ is usually strictly smaller than $\{v \in \mathbb D_{O} \colon v(T) > 0\}$ which is actually \emph{not} a disk even if the condition translates multiplicatively in $|T| < 1$.
Also, $\overline {\mathbb D}_{O}$ is usually strictly bigger than ${\mathbb D}_{O}$.
\end{xmp}

\subsection{Morphisms of adic spaces}

We review here some basic properties of morphisms of adic spaces.

By definition, a \emph{morphism of adic spaces} $W \to V$ is simply a morphism of topologically valued ring spaces between adic spaces.
It is said to be \emph{adic} if it comes locally from an adic morphism of Huber rings $A \to B$ (equivalently, it sends any analytic point to an analytic point - see below).

An \emph{open immersion} of adic spaces is an open immersion of topologically valued spaces between adic spaces.
A \emph{closed immersion} of adic spaces $W \hookrightarrow V$ is a morphism that comes locally on $V$ from a surjective \emph{adic} morphism $A \twoheadrightarrow B$ sending $A^+$ \emph{onto} $B^+$ (but not necessarily $\overline A^+$ onto $\overline B^+$).
A \emph{locally closed immersion} is the composite of a closed immersion and an open immersion.
We can also talk about \emph{locally closed adic subspaces}.

A morphism of adic spaces $W \to V$ is \emph{affinoid} if it comes locally on $V$ from a morphism of Huber pairs. 
One also defines a \emph{finite} morphism of adic spaces as a morphism $W \to V$ that comes locally on $V$ from a finite \emph{adic} morphism $A \to B$ sending $A^+$ onto $B^+$ (or equivalently such that $\overline A^+ \to \overline B^+$ is integral).
And a \emph{locally quasi-finite} morphism is a morphism which is locally of finite type with discrete fibers (\emph{quasi-finite} when it is quasi-compact).

We call a morphism of adic spaces $u \colon W \to V$ \emph{flat} (resp.\ \emph{faithfully flat}) when the underlying morphism of locally ringed spaces is flat (resp.\ faithfully flat).
It means that $u^*$ is exact (resp.\ and $u$ is surjective).
Flatness can be checked on stalks: $u$ is flat if and only if $\mathcal O_{W,u(v)} \to \mathcal O_{V,v}$ is flat at all $v \in W$.
As this was the case with formal schemes, flatness globalizes under noetherian hypothesis:
if $A \to B$ is a morphism of \emph{complete noetherian} Huber rings and $\mathrm{Spa}(B, B^+) \to \mathrm{Spa}(A, A^+)$ is (faithfully) flat, then the morphism of rings $A \to B$ is (faithfully) flat.
The converse holds for flatness but \emph{not} for faithful flatness in general.

\begin{xmp}
\begin{enumerate}
\item If $K \subset K'$ is a field extension, then the morphism $\mathrm{Spv}(K') \to \mathrm{Spv}(K)$ is faithfully flat: any valuation on $K$ will extend to $K'$.
\item The identity $\mathbb Z_p \to \mathbb Z_p$ is obviously faithfully flat, but the morphisms $\mathrm{Spa}(\mathbb Z_p) \to \mathrm{Spv}(\mathbb Z_p^\mathrm{triv})$ or $\mathrm{Spa}(\mathbb Z_p, \mathbb Z_p) \to \mathrm{Spa}(\mathbb Z_p, \mathbb Z)$ are not surjective, and therefore not faithfully flat.
\end{enumerate}
\end{xmp}

An adic space $V$ is said to be \emph{locally of finite type} over $O$ if it is locally isomorphic to a locally closed adic subspace of some $\mathbb A^n_{O}$.
It is said to be \emph{of finite type} if, moreover, it is quasi-compact.
Equivalently, $V$ is locally of finite type over $O$ if it is locally isomorphic to a closed adic subspace of a closed polydisc $\mathbb D^n_{O}(0, \{f_{i,j}^{-1}\})$.
In particular, we recover the original definition of Huber.
One could also define the notion of a morphism \emph{locally finitely presented} but this would be of no use to us when we enter the noetherian world.

\begin{xmp}
Let $O$ be an adic space.
Then,
\begin{enumerate}
\item the closed unit polydisc $\mathbb D^n_{O} $, the affine space $\mathbb A^n_{O}$ as well as the closed polydisc $\mathbb D^n_{O}(0, \{f_{i,j}^{-\frac 1k}\})$ of finite radii are all locally of finite type.
\item the open unit polydisc $\mathbb D^{n,-}_{O}$ is locally of finite type when $O$ is analytic (see below).
\item The proper unit disk $\overline {\mathbb D}_{O}$ is not (locally) of finite type over $O$ in general (although it is \emph{weakly of finite type} in the sense of Huber).
\end{enumerate}
\end{xmp}

Although fibered products of adic spaces do not exist in general, on can always pull back an adic space which is locally of noetherian type (definition \ref{noethtyp} below) along a morphism $V \to O$ which is locally of finite type.
This is a particular instance of proposition 1.1.2 in \cite{Huber96} (or theorem 8.56 in \cite{Wedhorn19}).

\begin{xmp}
\begin{enumerate}
\item 
If $v \in O = \mathrm{Spa}(A,A^+)$, we can then consider the Gauss point that we still denote by $v \in \mathbb D_{O}$ as well as the inner Gauss point $v^- \in \mathbb D_{O}$.
Then, there exists a unique point that we may still denote by $(v,v) \in \mathbb D^2_{O}$ (resp.\ $(v,v^-) \in \mathbb D^2_{O}$) over $(v,v)$ (resp.\ $(v,v^-)$).
There exists however two different points (of height three) over $(v^-,v^-)$ (depending on the order chosen to go inwards).
\item Using the previous example, one can show that the number of points in the fibers of a finite map is not stable under specialization/generalization.
Consider the finite flat morphism of degree $2$
\[
\mathbb D^2 \to \mathbb D^2, \quad (X, Y) \mapsto (XY, X+Y)
\]
Only the point $(v,v)$ is sent to $(v,v)$ but both points $(v,v^-)$ and $(v^-,v)$ are sent to $(v,v^-)$.
\end{enumerate}
\end{xmp}

At some point, we will also encounter pseudo-adic spaces (\cite{Huber96}, section 1.10).
First of all, a \emph{prepseudo-adic space} is simply a couple $(V,T)$ where $V$ is an adic space and $T$ is a subset of $V$.
By definition, its structural sheaf is $i^{-1}\mathcal O_V$ where $i : T \hookrightarrow V$ denotes the inclusion map.
$(V,T)$ is called a \emph{pseudo-adic space} when $T$ is convex under specialization and locally pro-constructible.
A \emph{morphism} $(V',T') \to (V,T)$ of prepseudo-adic spaces is a morphism of adic spaces from $V'$ to $V$ that sends $T'$ inside $T$.
We call it a \emph{strict neighborhood} if $V' \hookrightarrow V$ is an open immersion inducing a surjection $T' \twoheadrightarrow T$ (automatically a homeomorphism).
Strict neighborhoods form a right multiplicative system.
The category of prepseudo-adic spaces localized with respect to strict neighborhoods is the category of \emph{germs of adic spaces}.
We will usually denote by $T^\dagger$ the germ of $(V,T)$ and by $\mathcal O_T^\dagger$ its structural sheaf because $V$ plays now a secondary role.
More precisely, in the category of germs, we can always replace $V$ with any other neighborhood of $T$.
In particular, when $T$ is locally closed in $V$, we can always assume that $T$ is actually closed in $V$ (we recover the definition 2.2.1 of Abe and Lazda in \cite{AbeLazda20}).

\subsection{Adic space associated to a scheme}

In section \ref{adform}, we will recall how one can associate an adic space to a (locally noetherian) formal scheme.
This should not be confused with the following construction that only applies to usual schemes and provides a different object.

We used above the notation $\mathbb A^{\mathrm{val}}$.
Actually, to any locally ringed space $X$, one can associate a (topologically) valued ringed space $X^{\mathrm{val}}$ as follows.
One sets
\[
X^{\mathrm{val}} := \{(x,v) \colon x \in X, v \in \mathrm{Spv}(\kappa(x))\}
\]
(which is also sometimes written $\mathrm{Spv}(X)$).
It is made into a topological space by choosing as basis of open subsets, the subsets
\[
\{(x,v) \colon x \in U, v(f(x)) \geq v(g(x)) \neq 0\} \subset X^{\mathrm{val}},
\]
in which $U$ is an open subset of $X$ and $f,g \in \Gamma(U, \mathcal O_{X})$.
There exists an obvious continuous map
\[
\mathrm{supp}\, : X^{\mathrm{val}} \to X, \quad (x,v) \mapsto x,
\]
and $X^{\mathrm{val}}$ is endowed with the sheaf of (topological) rings $\mathcal O_{X^{\mathrm{val}}} := \mathrm{supp}\,^{-1}(\mathcal O_{X})$.
In particular, we have $\mathcal O_{X^{\mathrm{val}},(x,v)} = \mathcal O_{X,x}$ for all $(x,v) \in X^{\mathrm{val}}$ and this local ring is endowed with the valuation induced by $v$.

\begin{prop} \label{adjalg}
If $X$ is a scheme, then $X^{\mathrm{val}}$ is an adic space and the functor $X \mapsto X^{\mathrm{val}}$ is fully faithful (on schemes).
Moreover, for any adic space $V$, there exists a natural bijection
\[
\mathrm{Hom}(V, X) \simeq \mathrm{Hom}(V, X^{\mathrm{val}})
\]
(morphisms of locally ringed spaces on one side and morphisms of adic spaces on the other).
Finally, if $X = \mathrm{Spec}(A)$, then there exists a natural isomorphism $X^{\mathrm{val}} \simeq \mathrm{Spv}(A)$.
\end{prop}

\begin{proof}
We start with the last assertion:
if $A$ is any ring (endowed with the discrete topology), then there exists an obvious map
\[
\mathrm{Spv}(A) \to \mathrm{Spec}(A)^{\mathrm{val}}, \quad v \mapsto (\mathrm{supp}\, v, \overline v)
\]
in which $\overline v$ denotes the valuation induced by $v$ on $\kappa(\mathrm{supp}\, v)$.
This is easily seen to be an isomorphism of (topologically) valued ringed spaces.
Actually, by definition, we will always have
\[
\Gamma(R(f_{1}/f_{0}, \ldots, f_{r}/f_{0}) , \mathcal O_{V}) = A[1/f_{0}]
\]
(with the discrete topology).
It follows that if $X$ is a scheme, then $X^{\mathrm{val}}$ is an adic space.
Now, both other questions are local.
More precisely, for the fake adjunction, we may assume that $X = \mathrm{Spec}(A)$ and $V = \mathrm{Spa}(B, B^+)$ with $B$ complete.
Then the result follows from the equality
\[
\mathrm{Hom}(A, B) = \mathrm{Hom}((A, \emptyset), (B, B^+))
\]
((automatically continuous) homomorphisms of rings on one side and compatible pairs of (automatically continuous) homomorphisms of rings on the other).
Full faithfulness then results from the fact that $\Gamma(X^{\mathrm{val}}, \mathcal O_{X^{\mathrm{val}}}) = A$ when $X = \mathrm{Spec}(A)$.
\end{proof}

\begin{cor}
The functor $X \mapsto X^{\mathrm{val}}$ commutes with finite limits of schemes.
\end{cor}

\begin{proof}
This is simply because finite limits of schemes are also finite limits in the whole category of locally ringed spaces.
\end{proof}

In general, $X^{\mathrm{val}}$ is way too big and it is more convenient to rely on a relative version:

\begin{cor}[Huber] \label{relsch}
Let $X \to S$ be a morphism of schemes which is locally of finite type, $V$ an adic space locally of noetherian type and $V \to S$ a morphism of locally ringed spaces.
Then, the functor
\[
W \mapsto \mathrm{Hom}(W, X) \times_{\mathrm{Hom}(W, S) } \mathrm{Hom}(W,V)
\]
is representable by the adic space
\[
X_{V} := X \times_{S} V := X^{\mathrm{val}} \times_{S^{\mathrm{val}}} V.
\]
\end{cor}

\begin{proof}
The only thing to check is that the fibered product on the right hand side is representable and this follows from the fact that the morphism $X^{\mathrm{val}} \to S^{\mathrm{val}}$ is locally of finite type (because $X \to S$ is).
\end{proof}

\begin{xmp}
\begin{enumerate}
\item
If $V$ is an adic space, then we can consider
\[
\mathbb A^n_{V} := \mathbb A^{n,\mathrm{val}} \times V
\quad \mathrm{and} \quad
\mathbb P^n_{V} := \mathbb P^{n,\mathrm{val}} \times V.
\]
\item More generally, if $X$ is a affine (resp.\ quasi-projective) scheme over $S$, $V$ is an adic space and $V \to S$ a morphism of locally ringed spaces, then one may call $X_V$ an \emph{affine space} (resp.\ a \emph{quasi-projective space} over $V$.
\end{enumerate}
\end{xmp}

\subsection{Analytic points}

We are ultimately interested in analytic spaces that we shall shortly define.
General adic spaces will mostly serve as a bridge between (formal) schemes and analytic spaces.

\begin{dfn}
A point $v$ of an adic space $V$ is \emph{analytic} if it has an open neighborhood which is Tate.
The space $V$ is said to be \emph{analytic} if all its points are analytic.
\end{dfn}

Be careful that an affinoid space which is analytic is not necessarily the adic spectrum of a Tate ring (it's only true locally).

When $V:= \mathrm{Spa}(A, A^+)$ is affinoid, then a point $v \in V$ is non-analytic if and only if $\mathrm{supp}(v)$ is open in $A$ (or equivalently $\forall f \in I_0, v(f) = +\infty$ if $I_0$ is some ideal of definition for $A$).
In general, a point is non-analytic if and only if it has a trivial (vertical) generalization.
We will denote by $V^{\mathrm{an}}$ (instead of Huber's $V_{a}$) the open subset of analytic points of $V$.
More generally, if $T \subset V$ is any subset, we will simply write $T^{\mathrm{an}}$ instead of $T \cap V^{\mathrm{an}}$ for the set of analytic points of $T$.

This is the non-analytic locus which is functorial.
Actually, a morphism is adic if and only if it sends analytic points to analytic points.
As a consequence, if we are given a morphism $u : W \to V$ with $V$ analytic, then $W$ is automatically analytic and $u$ is automatically an adic morphism.
We may say that an adic space $V$ is \emph{Tate} if there exists a morphism $V \to \mathrm{Spa}(A,A^+)$ where $A$ is Tate. Then, $V$ is automatically analytic and there exists a ``topologically nilpotent unit'' $\pi$ on $V$.

\begin{dfn}
Let $(K, |-|)$ be a Huber valued field with valuation ring $K^+$.
Then, $\mathrm{Spa}(K^{\mathrm{triv}}, K^+)$ is called a \emph{non-analytic} Huber point.
If $|-|$ is not trivial, then $\mathrm{Spa}(K, K^+)$ is called an \emph{analytic Huber point}.
\end{dfn}

The topological space of a Huber point is totally ordered by generalization with one maximal point of height $0$ (non-analytic) or $1 $ (analytic) and one minimal point which is (the valuation defined by) $|-|$.

Let $V$ be an adic space and $v \in V$ a non-analytic point (resp.\ an analytic point).
Then, there exists a canonical morphism
\[
\mathrm{Spa}(\kappa(v)^\mathrm{triv}, \kappa(v)^+) \hookrightarrow V\quad \mathrm{(resp.\ }\ \mathrm{Spa}(\kappa(v), \kappa(v)^+) \hookrightarrow V\mathrm{)}
\]
that identifies the Huber point with the set of vertical generalizations of $v$.
This is important because the inverse image of a usual point under a morphism of adic spaces (i.e. its fiber) is not an adic space in general and we need to pull all the generalizations back together.

When $V$ is an \emph{analytic} space, then the maximal points for generalization in $V$ are exactly the points of height $1$ (the Berkovich points of $V$).
We shall denote their set by $[V]$.
There exists a retraction map $\mathrm{sep} \colon V \to [V]$ sending any point $v$ to its maximal generalization $[v]$ and $[V]$ is endowed with the \emph{quotient} topology (and \emph{not} the induced topology).
Note that the canonical map $\kappa(v) \hookrightarrow \kappa([v])$ on the residue fields is dense and induces therefore an isomorphism $\mathcal H(v) \simeq \mathcal H([v])$ on the completed residue fields.
Also, since we always have $\mathrm{sep}^{-1}(v) = \overline {\{v\}}$ for $v \in [V]$, we see that the maps $\mathrm{sep}$ and $\mathrm{sep}^{-1}$ induce a bijection between the (locally closed, open, closed) subsets of $[V]$ and the (locally closed, open, closed) subsets of $V$ that are stable under both specialization and generalization.
Moreover, $[V]$ is a \emph{Fr\'echet} topological space (i.e. $T_{1}$) and the map $\mathrm{sep}$ is the adjunction map for a fake adjunction
\[
\mathrm{Hom}([V], T) \simeq \mathrm{Hom}(V, T)
\]
whenever $T$ is a Fr\'echet topological space.
Finally, when $V$ is quasi-compact and quasi-separated, then the space $[V]$ is compact (Hausdorff) and the map $\mathrm{sep}$ is proper.

\subsection{Properness and smoothness}

We will review here the important notions of proper and smooth morphisms of adic spaces and some related properties.

Unless otherwise specified, we will always assume that separated, proper or partially proper morphisms of adic spaces are locally of finite type as in \cite{Huber93b}.
In his latter work, Huber makes some more sophisticated finiteness assumptions that are not relevant to us at the moment.
More precisely:

\begin{dfn}
A morphism of adic spaces $u \colon W \to V$ is said to be
\begin{enumerate}
\item \emph{separated} if $u$ is locally of finite type and the diagonal map $\Delta \colon W \hookrightarrow W \times_{V} W$ is closed.
\item \emph{partially proper} if $u$ is separated and universally specializing with respect to \emph{adic} pullback.
\item \emph{proper} if $u$ is partially proper and quasi-compact.
\end{enumerate}
\end{dfn}

You may visit section 1.3 of \cite{Huber96} for the details (even if some of his finiteness conditions there are slightly weaker than ours).
We may however recall that \emph{specializing} means that if $w \mapsto v$ and $v \rightsquigarrow v'$, then there exits $w \rightsquigarrow w'$ such that $w' \mapsto v'$.
For example, when $u$ is quasi-compact, (universally) specializing is equivalent to (universally) closed.
Universally here is always meant with respect to \emph{adic} pull-back.

\begin{xmp}
\begin{enumerate}
\item If $V$ is any adic space, then $\mathbb P^n_{V}$ is proper over $V$.
\item
It is not difficult to see that the open disk $\mathbb D^-_{V}$ is partially proper over $V$ when $V$ is analytic.
Since $\mathbb P_{V}$ is proper over $V$, it is sufficient to show that the open immersion $\mathbb D^-_{V} \hookrightarrow \mathbb P_{V}$ is specializing.
And we may also assume that $V$ is Tate with topologically nilpotent unit $\pi$.
But then, we have
\[
\mathbb D^-_{V} = \bigcup_{n \in \mathbb N} \mathbb D_{V}(0, \pi^{\frac 1n}) \quad \mathrm{with} \quad \overline {\mathbb D_{V}(0, \pi^{\frac 1n})} \subset \mathbb D_{V}(0, \pi^{\frac 1{n+1}}).
\]
\end{enumerate}
\end{xmp}

\begin{dfn}
A morphism $W \to V$ of adic spaces is said to satisfy the \emph{analytic valuative criterion} for properness if, given a non-archimedean field $F$ and valuation rings $R \subset F^+ \subsetneq F$, then any commutative diagram
\[
\xymatrix{W \ar[r] & V \\ \mathrm{Spa}(F, F^+) \ar[r] \ar[u] & \mathrm{Spa}(F,R) \ar[u] \ar@{-->}[ul]}
\]
may be uniquely completed by the dotted arrow.
\end{dfn}

When $W = \mathrm{Spa}(B, B^+)$ is affine, the lifting property means that the image of the map $B^+ \to \widehat F^+$ is contained in $\widehat R$.

\begin{prop}
Let $u : W \to V$ be a morphism of \emph{analytic} spaces which is locally of finite type and quasi-separated.
Then $u$ satisfies the analytic valuative criterion for properness if and only if $u$ is partially proper.
\end{prop}

\begin{proof}
This is exactly corollary 1.3.9 of \cite{Huber96}).
\end{proof}

There also exists a more subtle valuative criterion when $V$ is \emph{not} assumed to be analytic (proposition 3.12.2 of \cite{Huber93b}).

We now turn to smoothness.
A morphism $u : W \to V$ is said to be \emph{formally unramified} (resp.\ \emph{formally smooth}, resp.\ \emph{formally \'etale}) if any commutative diagram
\[
\xymatrix{W \ar[r] & V \\ \mathrm{Spa}(R/\mathfrak a, (R/\mathfrak a)^+) \ar@{^{(}->}[r] \ar[u] & \mathrm{Spa}(R, R^+) \ar[u] \ar@{-->}[ul] }
\]
with $\mathfrak a$ nilpotent may be completed by the diagonal arrow in at most (resp.\ at least, resp.\ exactly) one way.
The morphism $u$ is called \emph{unramified} (resp \emph{smooth}, resp.\ \emph{\'etale}) if it is formally unramified (resp.\ smooth, resp.\ \'etale) and locally finitely presented.

If $W \hookrightarrow V$ is an immersion defined (on some open subset of $V$) by an ideal $\mathcal I_{W}$, then we may consider the \emph{first infinitesimal neighborhood} $W^{(1)}$ of $W$ in $V$ defined by $\mathcal I_{W}^2$ (which is always an adic space with the same underlying subspace as $W$) and the corresponding short exact sequence
\[
0 \to \check{\mathcal N}_{W/V} \to \mathcal O_{W_{(1)}} \to \mathcal O_{W} \to 0
\]
where $\check{\mathcal N}_{W/V}$ is by definition the \emph{conormal sheaf}.
In the case of the diagonal immersion $W \hookrightarrow W \times_{V} W$ associated to a morphism $W \to V$, we obtain the \emph{sheaf of differential forms} $\Omega^1_{W/V}$.
One can show that a morphism of finite type $u : W \to V$ is unramified if and only if $\Omega^1_{W/V} = 0$.

At some point, the adic spaces will only serve as a bridge between formal schemes and analytic spaces.
The following concept will therefore be quite useful:

\begin{dfn} \label{propan}
Let $\mathcal P$ be a property of morphisms of analytic spaces which is local on the base, stable under pullback and stable under composition.
A morphism $W \to V$ of adic spaces is said to be \emph{analytically $\mathcal P$} if for any analytic space $V'$ over $V$, the pull-back $W' \to V'$ is $\mathcal P$.
\end{dfn}

We may choose for $\mathcal P$ the property of being (locally) of finite type (resp.\ separated, resp.\ (partially) proper, resp.\ unramified, resp.\ smooth, resp.\ \'etale, resp.\ an open or a (locally) closed immersion).

Note that the fact that $W'$ is representable is part of the definition (and not automatic).
We also insist on the fact that representability condition is also implicit when we require $\mathcal P$ to be stable under pullback.

Actually, it is sufficient to check the condition when $V'$ is Tate affinoid because the condition is local.
Note also that, when $V$ is analytic, the definition is equivalent to $u$ itself satisfying $\mathcal P$ because the property is assumed to be stable under pull-back.

Be careful however that the image of $V'$ into $V$ is \emph{not} necessarily contained inside the analytic locus $V^{\mathrm{an}}$ of $V$, simply because the analytic locus is not functorial (this is the non-analytic locus which is functorial).
In particular, it is not sufficient to consider the case $V' = V^{\mathrm{an}}$ in the above definition.

\begin{xmp}
The absolute open unit disk (we use the $T$-adic topology here)
\[
\mathbb{D}^- : = \mathrm{Spa}(\mathbb Z[[T]])
\]
is analytically partially proper and analytically smooth over $\mathrm{Spv}(\mathbb Z)$.
Note however that $\mathbb{D}^-$ is \emph{not} partially proper and \emph{not} smooth over $\mathrm{Spv}(\mathbb Z)$ because the structural map is not even an adic morphism.
\end{xmp}

\subsection{Adic spaces locally of noetherian type}

In order to deal with coherent sheaves on adic spaces, il will be necessary to make some noetherian assumptions.

\begin{dfn} \label{noethtyp}
An adic space $V$ is said to be \emph{locally of noetherian type} if it is locally isomorphic to some $\mathrm{Spa}(A, A^+)$ with $A$ of noetherian type.
If moreover, $V$ is quasi-compact, it is said to be of \emph{noetherian type}. 
\end{dfn}

If $u : W \to V$ is a morphism of adic spaces which is locally of finite type and $V$ is locally of noetherian type, then $W$ is also locally of noetherian type.
Moreover, in this situation, $\Omega^1_{W/V}$ is a coherent sheaf.
On can also show that when $u$ is smooth, it is necessarily flat.
Moreover, there exists a Jacobian criterion:

\begin{thm}[Huber]
Let $O$ be an adic space locally of noetherian type.
\begin{enumerate}
\item Let $u : W \to V$ be a morphism between adic spaces that are locally of finite type over $O$.
Then, there exists a right exact sequence
\[
u^*\Omega^1_{V/O} \to \Omega^1_{W/O} \to \Omega^1_{W/V} \to 0.
\]
Assume moreover that $W$ is smooth over $O$.
Then $u$ is smooth (resp.\ \'etale) if and only if $u^*\Omega^1_{V/O}$ is locally a direct summand in (resp.\ is isomorphic to) $\Omega^1_{W/O}$
\item Let $i : W \hookrightarrow V$ be an immersion of adic spaces locally of finite type over $O$.
Then, there exists a right exact sequence
\[
\check{\mathcal N}_{W/V} \to i^*\Omega^1_{V/O} \to \Omega^1_{W/O} \to 0.
\]
Assume moreover that $V$ is smooth over $O$.
Then $W$ is smooth (resp.\ \'etale) over $O$ if and only if $\check{\mathcal N}_{W/O}$ is is locally a direct summand in (resp.\ is isomorphic to) $i^*\Omega^1_{V/O}$.
\end{enumerate}
\end{thm}

\begin{proof}
The existence of the right exact sequences follow from proposition 1.6.3 in \cite{Huber96} and the other assertions are proved in proposition 1.6.9 of \cite{Huber96}.
\end{proof}

If one has to deal with formal schemes that are \emph{not} locally noetherian, then it would be necessary to replace Huber adic spaces with the \emph{generalized adic spaces}\footnote{Another solution is to rely on condensed mathematics but this would be another story.} of Scholze and Weinstein (section 2.1 of \cite{ScholzeWeinstein13} - see also Peter Scholze's lectures \cite{Scholze*}).
The main point is that $\mathrm{Spa}(A, A^+)$ need not be sheaffy in general.
In Scholze-Weinstein theory, the \emph{affinoid space} $\mathrm{Spa}(A, A^+)$ is redefined to be the sheaf associated to the presheaf of sets
\[
(B, B^+) \mapsto \mathrm{Hom}((A, A^+), (B, B^+)) 
\]
on the category $\mathcal C$ opposite to the category of complete Huber pairs.
Here, we endow $\mathcal C$ with the topology generated by rational coverings (be careful that this is not a pretopology in the usual sense: the inverse image of a rational open immersion is only a filtered colimit of rational open immersions).
A \emph{generalized adic space} is a sheaf $V$ on $\mathcal C$ which is locally ind-representable by rational open immersions.

There exists a functor $V \mapsto |V|$ from generalized adic spaces to topological spaces which may be defined by glueing from the affinoid case.
If $V = \mathrm{Spa}(A, A^+)$, then
\[
|V| := \{\mathrm{continuous}\ \mathrm{valuations}\ \mathrm{on}\ A\ \mathrm{non-negative}\ \mathrm{on}\ A^+\}/\sim
\]
is simply the underlying space of the \emph{former} $\mathrm{Spa}(A, A^+)$ (in practice, one still writes $V$ instead of $|V|$).
The sheaf $\mathcal O$ (resp.\ $\mathcal O^+$) on $\mathcal C$ is defined as the sheaf \emph{associated to} the presheaf
\[
(A, A^+) \mapsto A \quad (\mathrm{resp.\ }\ A^+).
\]
By restriction, they both induce a sheaf of topological rings on $|V|$ that we may denote by $\mathcal O_{V}$ and $\mathcal O_{V}^+$ respectively
(be careful however that we may have $\Gamma(V, \mathcal O_{V}) \neq \widehat A$ when $V = \mathrm{Spa}(A, A^+)$ is not \emph{sheaffy}).

We can identify the category of Huber adic spaces (Scholze and Weinstein call them \emph{honest} adic spaces) with the full subcategory of generalized adic spaces $V$ that satisfy for all Huber pairs $(A, A^+)$,
\[
\mathrm{Hom}(V, \mathrm{Spa}(A, A^+)) \simeq \mathrm{Hom} ((A, A^+), (\Gamma(V, \mathcal O_{V}), \Gamma(V, \mathcal O_{V}^+))).
\]

We will however stick to adic spaces that are locally of noetherian type for various reasons.

\section{Adic spaces and formal schemes} \label{adform}

From now on and unless otherwise specified, \textbf{all formal schemes (resp.\ adic spaces) are locally noetherian (resp.\ locally of noetherian type)}.

In this section, we recall how a formal scheme may be seen as an adic space in such a way that most geometric properties translate directly into the new world.

\subsection{Trivial points}

We start with the reverse construction that associates a formal scheme to an adic space.

If $I_{0}$ is an ideal of definition in a Huber ring $A$, then $A^{/}$ will denote the ring $A$ endowed with the $I_{0}A$-adic topology (and $\widehat {A^{/}}$ will denote the completion of $A$ for the $I_{0}A$-adic topology).
The topology of $A^/$ is coarser than the topology of $A$ but an ideal $\mathfrak a$ of $A$ is open for one topology if and only if it is open for the other.
Concerning completions, we have
\[
\widehat A^{/} = \varprojlim A/(I_{0}A)^n \neq \widehat A = \varprojlim A/I_{0}^n
\]
in general.
\begin{xmp}
\begin{enumerate}
\item
If $A$ is an adic ring, then $A^{/} = A$ as topological rings.
\item
If $A$ is a Tate ring, then $A^{/}$ has the coarse topology and $\widehat {A^{/}} = \{0\}$.
\end{enumerate}
\end{xmp}

A point of an adic space is said to be \emph{trivial} if the corresponding valuation is trivial (has height zero).

\begin{prop}
Let $V$ be an adic space, $V_{0}$ the subset of trivial points and $i = V_{0} \hookrightarrow V$ the inclusion map.
Then, we have $i^{-1}\mathcal O^+_{V} = i^{-1}\mathcal O_{V}$ and $(V_{0}, i^{-1}\mathcal O_{V})$ is a formal scheme.
If $V = \mathrm{Spa}(A, A^+)$, then there exists a natural isomorphism $V_{0} \simeq \mathrm{Spf}(A^{/})$.
\end{prop}

\begin{proof}
Clearly, $\mathcal O^+_{V,v} = \mathcal O_{V,v}$ when $v$ is trivial.
Now, the second statement is local and it is therefore sufficient to prove the last assertion.
We first notice that the support map induces a bijection
\begin{equation} \label{kermap}
\mathrm{supp}\, : V_{0} \simeq \mathrm{Spf}(A^{/}).
\end{equation}
More precisely, the inverse map sends an open prime $\mathfrak p$ of $A$ to $v_{\mathfrak p}$ defined as follows:
\[
v_{\mathfrak p}(f) = \left\{ \begin{array}{cl} 0 & \mathrm{if}\ f(\mathfrak p) \neq 0 \\ + \infty & \mathrm{otherwise}. \end{array} \right.
\]
Now, let $(f_{0}, \ldots, f_{r})$ be an open ideal in $A$.
Since any $v \in V_{0}$ will only take the values $0$ and $+ \infty$, we see that
\[
v(f_{1}), \ldots, v(f_{r}) \geq v(f_{0}) \neq + \infty \Leftrightarrow f_{0}(v) \neq 0.
\]
In other words, we have
\[
R(f_{1}/f_{0}, \ldots, f_{r}/f_{0}) \cap V_{0} = D(f_{0}) \subset V_{0}
\]
and it immediately follows that the support map \eqref{kermap} is a homeomorphism.

The sheaf $i^{-1}\mathcal O_{V}$ is the sheaf associated to
\[
D(f_{0}) \mapsto \varinjlim \Gamma(R(f_{1}/f_{0}, \ldots, f_{r}/f_{0}) , \mathcal O_{V}).
\]
In order to see that the support map \eqref{kermap} is an isomorphism, we will compute this limit.
After localizing, we may assume that $f_{0} = 1$ and then, after adding some functions, that $A = A_{0}[f_{0}, \ldots, f_{r}]$ where $A_{0}$ is some ring of definition.
Then by definition, we have
\[
\Gamma\left(R(f_{1}/f_{0}, \ldots, f_{r}/f_{0}) , \mathcal O_{V}\right) = \widehat {A^/}. \qedhere
\]
\end{proof}

Note that we used our general assumption that $A$ is finitely generated over $A_{0}$ in order to finish this proof.

One may also remark that $V_{0} \cap V^{\mathrm{an}} = \emptyset$.
Actually, the closure $\overline V_{0}$ of $V_{0}$ in $V$ is exactly the set of non-analytic points.

\subsection{Adic space associated to a formal scheme}

It is shown in \cite{Huber94}, section 4 (or \cite{Wedhorn19}, chapter 9.2) that one can associate functorially an adic space to any (locally noetherian) formal scheme $P$.
We will denote it by $P^{\mathrm{ad}}$ (and not $t(P)$ as Huber does).
More precisely, we have the following:

\begin{prop}\label{adzer}

The functor $V \mapsto V_{0}$ has an adjoint $P \mapsto P^{\mathrm{ad}}$ which is fully faithful.
When $P = \mathrm{Spf}(A)$, we have $P^{\mathrm{ad}} = \mathrm{Spa}(A)$.
\end{prop}

\begin{proof}
We have to show that, if $P$ a formal scheme, then there exists a unique adic space $P^{\mathrm{ad}}$ such that, given an adic space $V$, there exists a natural bijection
\[
\mathrm{Hom}(P^{\mathrm{ad}}, V) \simeq \mathrm{Hom}(P, V_{0}).
\]
The question is local on $P$ and $V$ and we may therefore assume that $P = \mathrm{Spf}(A)$ is affine and $V := \mathrm{Spa}(B, B^+)$ is affinoid.
We may also assume that $A$ is complete.
Then, it follows from lemma \ref{addoub} below that if we set $P^{\mathrm{ad}} = \mathrm{Spa}(A)$, we have
\[
\mathrm{Hom}(P^{\mathrm{ad}}, V) \simeq \mathrm{Hom}(P, V_{0}).
\]
It only remains to show that the functor is fully faithful, or equivalently, that the adjunction map is an isomorphism $P^{\mathrm{ad}}_{0} \simeq P$.
Again this is a local question.
But if $A$ is an adic ring, then we have $A^/ = A$.
\end{proof}

We used in the above proof the following straightforward lemma:

\begin{lem} \label{addoub}
\begin{enumerate}
\item
The functor $A \mapsto A^{/}$ is adjoint to the forgetful functor from adic rings to Huber rings.
\item
The functor $(A, A^+) \mapsto A^{/}$ is adjoint to the functor $A \mapsto (A, A)$ from adic rings to Huber pairs. $\Box$
\end{enumerate}
\end{lem}

As a consequence of the proposition, the functor $P \mapsto P^{\mathrm{ad}}$ commutes with all colimits, the functor $V \mapsto V_{0}$ commutes with all limits and there exists a natural inclusion $P \hookrightarrow P^{\mathrm{ad}}$ (identification of the points of $P$ with the trivial points of $P^{\mathrm{ad}}$).
We may also notice that
\[
(V_{0})^{\mathrm{ad}} = \{v \in V \colon \forall f \in \mathcal O_{V,v}, v(f) \geq 0\} \subset V
\]
is the set of all points having a trivial horizontal specialization.
Finally, if $X$ is a usual scheme, then there exists a canonical inclusion $X^{\mathrm{ad}} \hookrightarrow X^{\mathrm{val}}$ which is locally given by the inclusion $ \mathrm{Spa}(A) \subset \mathrm{Spv}(A)$.

\begin{prop} \label{eqco}
If $P$ is a formal scheme, and $i : P \hookrightarrow P^{\mathrm{ad}}$ denotes the inclusion map, then $i^{-1}$ induces an equivalence between coherent modules on both sides.
\end{prop}

\begin{proof}
We already know that the restriction of $\mathcal O_{P^{\mathrm{ad}}}$ to $P$ is exactly $\mathcal O_{P}$.
Moreover, the question is local and we may therefore assume that $P = \mathrm{Spf}(A)$ with $A$ complete.
In this case, we have at our disposal Theorem A and B on both sides and restriction is simply given by $M \otimes_{A} \mathcal O_{P^{\mathrm{ad}}} \mapsto M \otimes_{A} \mathcal O_{P}$ (for some finite $A$-module $M$).
\end{proof}

We will denote by $\mathcal F^{\mathrm{ad}}$ the coherent extension of a coherent $\mathcal O_{P}$-module $\mathcal F$ to $P^{\mathrm{ad}}$.

We also have the following fake adjunction on the other side (recall that we denote by $V^+$ the topologically locally ringed space $(V, \mathcal O_{V}^+)$):

\begin{prop}[Huber] \label{adplus}
If $V$ is an adic space and $P$ is a formal scheme, then there exists a natural bijection
\[
\mathrm{Hom}(V^+, P) \simeq \mathrm{Hom}(V, P^{\mathrm{ad}})
\]
(morphisms of topologically ringed spaces on one side and morphisms of adic spaces on the other).
\end{prop}

\begin{proof}
This is shown by Huber in Proposition 4.1 of \cite{Huber94}.
\end{proof}

When $P = \mathrm{Spf}(A)$ with $A$ complete, we have
\[
\mathrm{Hom}(V, P^{\mathrm{ad}}) \simeq \mathrm{Hom}(A, \Gamma(V, \mathcal O_{V}^+))
\]
(continuous maps on the right hand side).
As a consequence of proposition \ref{adplus}, we have the following:

\begin{cor}
The functor $P \mapsto P^{\mathrm{ad}}$ commutes with all finite limits.
\end{cor}

\begin{proof}
This is not completely formal because we do not have a true adjunction.
Since $\mathrm{Spf}(\mathbb Z)^{\mathrm{ad}} = \mathrm{Spa}(\mathbb Z)$, we only have to show that, given two morphisms of formal schemes $P_{i} \to P$ for $i = 1, 2$ and a commutative diagram
\[
\xymatrix{V \ar@/^/[rrd] \ar@/_/[rdd] \ar@{-->}[rd] \\ &(P_{1} \times_{P} P_{2})^{\mathrm{ad}}\ar[r] \ar[d] & P_{1}^{\mathrm{ad}} \ar[d] \\ & P_{2}^{\mathrm{ad}} \ar[r] & P^{\mathrm{ad}},}
\]
there exits a unique dotted arrow that preserves the commutativity of the diagram.
This is a local question and we may therefore assume that $P, P_{1}, P_{2}$ are affine with rings of functions $A, A_{1}$ and $A_{2}$ respectively.
But then, using the remark following the proposition, we can build a canonical map $A_{1} \otimes_{A} A_{2} \to \Gamma(V, \mathcal O_{V}^+)$ and we are done.
\end{proof}

\begin{xmp}
Let $S$ be a formal scheme.
\begin{enumerate}
\item We have
\[
\mathbb A^{n, \mathrm{ad}}_{S} = \mathbb D^{n}_{S^{\mathrm{ad}}}, \quad\mathbb P^{n, \mathrm{ad}}_{S} = \mathbb P^n_{S^{\mathrm{ad}}}\quad \mathrm{and} \quad
\mathbb A^{n,-, \mathrm{ad}}_{S} = \mathbb D^{n,-}_{S^{\mathrm{ad}}}.
\]
\item When $S = \mathrm{Spf}(A)$ and $V = S^{\mathrm{ad}}$, we will also consider the \emph{relative bounded disk}
\[
\mathbb D^{\mathrm b,n}_{V} = \mathbb A^{\mathrm b, n, \mathrm{ad}}_{S}.
\]
\end{enumerate}
There exists a sequence of (strict) inclusions
\[
\mathbb D^- \subset \mathbb D^{\mathrm{b}} \subset \mathbb D \subset \mathbb A^\mathrm{val} \subset \mathbb P^\mathrm{val}.
\]
\end{xmp}

\subsection{Specialization}

We are talking here about specialization from the adic world to the formal world.

\begin{dfn}
If $P$ is a formal scheme, then the composition of the obvious morphism of locally topologically ringed spaces and the adjunction morphism
\[
\mathrm{sp} \colon P^{\mathrm{ad}} \to P^{\mathrm{ad},+} \to P
\]
is the \emph{specialization} map.
\end{dfn}

Recall that, by construction, if $P = \mathrm{Spf}(A)$, then we have
\[
\mathrm{sp}(v) = \{f \in A \colon v(f) > 0\}.
\]
Alternatively, any $v \in P^{\mathrm{ad}}$ has a unique trivial \emph{horizontal} specialization $v_{0}$ and we have $\mathrm{sp}(v) = \mathrm{supp}\, v_{0}$.
This is a good reason for calling this map the ``specialization'' map.

\begin{xmp}
If $X$ is a (usual) scheme, then the diagram
\[
\xymatrix{X^{\mathrm{ad}} \ar@{^{(}->}[rr] \ar[dr]^{\mathrm{sp}} && X^{\mathrm{val}} \ar[dl]_{\mathrm{supp}} \\ & X &
}
\]
is \emph{not} commutative in general.
In the case $X = \mathrm{Spec}(\mathbb Z)$, we have $X^{\mathrm{ad}} = X^{\mathrm{val}} = \mathrm{Spv}(\mathbb Z)$ and $\mathrm{Spec}(\mathbb Z)$ embeds canonically into $\mathrm{Spv}(\mathbb Z)$ by sending the ideal $(p)$ to the trivial valuation modulo $p$ (works also for $p = 0$).
The only other valuations on $\mathbb Z$ are the $p$-adic valuations $v_{p}$ for $p$ prime.
Support and specialization provide two \emph{different} natural sections for this embedding because $\mathrm{supp}\,(v_{p}) = (0)$ but $\mathrm{sp}(v_{p}) = (p)$.
\end{xmp}

\begin{prop} \label{retrac}
Let $P$ be a formal scheme.
Then the specialization map $P^{\mathrm{ad}} \to P$ (resp.\ the adjunction map $P^{\mathrm{ad},+} \to P$) is a faithfully flat retraction for the canonical embedding $P \hookrightarrow P^{\mathrm{ad}}$ (resp.\ $P \hookrightarrow P^{\mathrm{ad},+}$).
\end{prop}

\begin{proof}
One first checks that the composite morphism
\[
P \hookrightarrow P^{\mathrm{ad}} \to P^{\mathrm{ad},+} \to P
\]
is the identity.
This is a local question which becomes trivial in the affine case.
It only remains to show that both specialization and adjunction maps are flat.
This follows from the fact that, if $v \in P^{\mathrm{ad}}$, then the localization maps
\[
\mathcal O^+_{P^{\mathrm{ad}},v} \to \mathcal O^+_{P^{\mathrm{ad}},v_{0}} \quad \mathrm{and} \quad \mathcal O_{P^{\mathrm{ad}},v} \to \mathcal O_{P^{\mathrm{ad}},v_{0}}
\]
are flat (noetherian hypothesis used here) and we have by definition
\[
\mathcal O^+_{P^{\mathrm{ad}},v_{0}} = \mathcal O_{P^{\mathrm{ad}},v_{0}} = \mathcal O_{P,\mathrm{sp}(v)}. \qedhere
\]

\end{proof}

\begin{cor}
If $P$ is a formal scheme, then $\mathrm {Rsp}_{*}$ and $\mathrm {Lsp}^{*}$ induce an equivalence of categories between coherent sheaves on $P^{\mathrm{ad}}$ and coherent sheaves on $P$.
\end{cor}

\begin{proof}
Follows from proposition \ref{eqco}.
\end{proof}

Alternatively, such an equivalence may be shown directly: it follows from the construction of the fake adjunction of proposition \ref{adplus} that, if $Q \subset P$ is an formal open subscheme, then $\mathrm{sp}^{-1}(Q) = Q^{\mathrm{ad}}$.
In particular, the inverse image of an affine open subset is always an affinoid open subset and we can use Cartan's theorems A and B directly.

\subsection{Persistence of properties}

We will show that most geometric properties of formal schemes correspond perfectly to their analog in the adic world.

We insist on the fact that the functor $P \mapsto P^{\mathrm{ad}}$ is fully faithful and commutes with all finite limits and all colimits.
We can also list the following properties:

\begin{prop} \label{listpr}
\begin{enumerate}
\item \label{opn}
A morphism of formal schemes $Q \to P$ is a locally closed (resp.\ an open, resp.\ a closed) immersion if and only if $Q^{\mathrm{ad}} \to P^{\mathrm{ad}}$ is a locally closed (resp.\ an open, resp.\ a closed) immersion.
Actually, there exists a bijection between closed formal subschemes of $P$ and closed adic subspaces of $P^{\mathrm{ad}}$. 
\item \label{cov} A family $\{P_{i} \to P\}_{i \in I}$ of morphisms of formal schemes is an open covering if and only if $\{P_{i}^{\mathrm{ad}} \to P^{\mathrm{ad}}\}_{i \in I}$ is an open covering.
\item \label{aff}
A morphism of formal schemes $Q \to P$ is affine (resp.\ quasi-compact, resp.\ quasi-separated) if and only if $Q^{\mathrm{ad}} \to P^{\mathrm{ad}}$ is affinoid (resp.\ quasi-compact, resp.\ quasi-separated).
\item \label{Hub} A morphism of formal schemes $Q \to P$ is adic (resp.\ (locally) of finite type, resp.\ (locally) quasi-finite, resp.\ locally of finite type and separated, resp.\ proper, resp.\ finite, resp. flat, resp. faithfully flat quasi-compact) if and only if $Q^{\mathrm{ad}} \to P^{\mathrm{ad}}$ is so.
\item \label{smet} A morphism of formal schemes $Q \to P$ is unramified (resp.\ smooth, resp.\ \'etale) if and only if $Q^{\mathrm{ad}} \to P^{\mathrm{ad}}$ is so.
\end{enumerate}
\end{prop}

\begin{proof}
By definition, if $Q \subset P$ is an open subset, then $Q^{\mathrm{ad}}$ is open in $P^{\mathrm{ad}}$.
And since the functor $V \mapsto V_{0}$ has the analogous property, we see that the converse is also true.
The question about closed immersions is local on $P$ and we may therefore assume that $P = \mathrm{Spf}(A)$.
In this case, closed immersions on both sides correspond bijectively to ideals in $A$.
Here again, noetherianity is crucial.
This finishes the proof of assertion \ref{opn}).

Assertion \ref{cov}) results from the fact that both functors $P \mapsto P^{\mathrm{ad}}$ and $V \mapsto V_{0}$ preserve open coverings.

We know that, when $P$ is affine, $P^{\mathrm{ad}}$ is affinoid.
And conversely, if $V$ is affinoid, then $V_{0}$ is affine.
Assertion \ref{aff}) follows easily.

In assertion \ref{Hub}), the adic and (locally of) finite type cases follow easily from the definitions (see also proposition 4.2 of \cite{Huber94}).
For the finite case, we use the fact that there exists a bijection between coherent sheaves on both sides.
Flatness is local upstairs and therefore follows from the affine case which is trivial (since it is then a global condition).
Concerning faithful flatness, the converse implication follows from proposition \ref{retrac}.
For the direct implication, we can assume that $P = \mathrm{Spf}(A)$ is affine with $A$ complete.
Since $Q$ is then quasi-compact, we can replace it with $\coprod_{i=1}^r Q_i$ where $Q = \bigcup _{i=1}^r Q_i$ is a finite affine open covering and therefore also assume that $Q = \mathrm{Spf}(B)$ is affine with $B$ complete.
Our assertion then follows from lemma \ref{flatad} below.
Since a locally quasi-finite morphism of formal schemes is locally the composition of an open immersion and a finite map, the direct implication is clear in the (locally) quasi-finite case.
Conversely, if $Q^{\mathrm{ad}} \to P^{\mathrm{ad}}$ has discrete fibers, then the same si true for the induced map $Q \to P$.
The separated case is an automatic consequence of the fact that the functor $P \mapsto P^{\mathrm{ad}}$ preserves fibered products and closed immersions.
For properness, it seems more reasonable to rely on the (adic) valuative criterion (see \cite{Huber93b}, proposition 3.12.5 for a proof).

Assertion \ref{smet}) will follow from lemma \ref{leform} below (formal case) because we already know that $u$ is locally of finite type if and only if $u^{\mathrm{ad}}$ is.
\end{proof}

There exists also a beautiful theorem of Huber (\cite{Huber93b}, proposition 3.12.7) that states that any proper adic space over $P^{\mathrm{ad}}$ comes from a (unique and proper) formal scheme over $P$.
The finite case should be clear.

The following lemma was required at some point in the proof of proposition \ref{listpr}:

\begin{lem} \label{flatad}
Let $A \to B$ be a (continuous) morphism of \emph{not necessarily noetherian} adic rings and $v \in \mathrm{Spa}(A)$.
Assume either that
\begin{enumerate}
\item $v$ is trivial and $\mathrm{Spf}(B) \to \mathrm{Spf}(A)$ is surjective, or
\item $v$ is not trivial and $A \to B$ is faithfully flat,
\end{enumerate}
then $v$ extends to some $w \in \mathrm{Spa}(B)$.
\end{lem}

\begin{proof}
The first case follows immediately from the identification of trivial points and open primes.
We can therefore assume that $v$ is not trivial and $A \to B$ is faithfully flat.
We first replace $A$ with the valuation ring $\mathcal V$ of $v$ and $B$ with $\mathcal V \otimes_{A} B$.
Since the map $\mathcal V \to B$ (the new $B$) is still faithfully flat, we can lift the maximal ideal $\mathfrak m$ of $\mathcal V$ to a prime ideal $\mathfrak n \subset B$.
We may then replace $B$ with its localization $B_{\mathfrak n}$ and assume that $\mathcal V \to B$ is a local morphism of local rings.
Now, the support $\mathfrak p$ of $v$ also lifts to a prime $\mathfrak q \subset B$.
We can then replace $B$ (this new-new $B$) with the quotient $B/\mathfrak q$ which is now a local domain.
As a local domain, $B$ (the new one !) is dominated by some valuation ring $\mathcal W$ and we may finally assume that $B=\mathcal W$.
We are done now because we replaced our original morphism $A \to B$ with a local morphism of valuation rings $\mathcal V \hookrightarrow \mathcal W$.
We can then simply use the natural valuation $w$ of $\mathcal W$ which is continuous non-negative (and extends $v$ by construction).
\end{proof}

In this last proof, even if $A$ and $B$ are noetherian, then it usually happens that $\mathcal V$ and $\mathcal W$ are not.
It is also important to rule out the trivial case in the beginning to make sure that $B \to \mathcal W$ is continuous (when $\mathcal W$ has its natural topology) in the end.

\begin{xmp}
A morphism $Q \to P$ might be surjective although $Q^{\mathrm{ad}} \to P^{\mathrm{ad}}$ is not: if we endow $\mathbb Z_{p}$ with the $p$-adic topology, then the closed immersion $\mathrm{Spec}(\mathbb F_{p}) \hookrightarrow \mathrm{Spf}(\mathbb Z_{p})$ is actually a homeomorphism (of one point spaces) but the closed immersion $\mathrm{Spa}(\mathbb F_{p}) \hookrightarrow \mathrm{Spa}(\mathbb Z_{p})$ identifies the unique point of the first space with the unique \emph{closed} point of the second one (which also has an \emph{open} point).
Note that this counter example is a closed immersion and in particular proper and adic.
The same example shows that assertion \ref{cov}) is not true for closed or locally closed coverings in general.
\end{xmp}

In the course of the proof of proposition \ref{listpr}, we also used the following:

\begin{lem} \label{leform}
A locally noetherian morphism of formal schemes $u \colon Q \to P$ is formally unramified (resp.\ formally smooth, resp.\ formally \'etale) if and only if $u^{\mathrm{ad}} \colon Q^{\mathrm{ad}} \to P^{\mathrm{ad}}$ is.
\end{lem}

\begin{proof}
The question is local on $P$ and $Q$ (thanks to the locally noetherian hypothesis in the smooth case) that we may assume to be affine, say $P = \mathrm{Spf}(A)$ and $Q = \mathrm{Spf}(B)$.
We have to check that the lifting property on the formal side is equivalent to the corresponding lifting property on the adic side (complete rings and complete pairs) :
\[
\xymatrix{B \ar[d] \ar@{-->}[rd] & \ar[l] \ar[d] A \\ R/\mathfrak a & R \ar@{->>}[l] } \quad \mathrm{and} \quad \xymatrix{(B, B) \ar[d] \ar@{-->}[rd] & \ar[l] \ar[d] (A, A) \\ ( S/\mathfrak b, (S/\mathfrak b)^+) & (S, S^+) \ar@{->>}[l] }.
\]
This follows from the fact that $(S/\mathfrak b)^+ = S^+/(\mathfrak b \cap S^+)$ when $\mathfrak b$ is nilpotent.
\end{proof}

We also want to mention the following.
As we saw above, when $P$ is a formal scheme, there exists an equivalence of categories $\mathcal F \mapsto \mathcal F^{\mathrm{ad}}$ between coherent sheaves on $P$ and $P^{\mathrm{ad}}$.
One easily sees that, if $u \colon Q \to P$ is a locally noetherian morphism, then $(\Omega^1_{Q/P})^{\mathrm{ad}} = \Omega^1_{Q^{\mathrm{ad}}/P^{\mathrm{ad}}}$.
More generally, if $Q \hookrightarrow P$ is any immersion of formal schemes, then $\check{\mathcal N}_{Q/P}^{\mathrm{ad}} = \check{\mathcal N}_{Q^{\mathrm{ad}}/P^{\mathrm{ad}}}$.

For the pleasure, we can also state a beautiful criterion of properness for usual schemes.
Let us first recall that, when $X$ is a usual scheme, then there exists a natural inclusion $X^{\mathrm{ad}} \subset X^{\mathrm{val}}$ which is locally given by $\mathrm{Spa}(A) \subset \mathrm{Spv}(A)$.
Moreover, if $Z$ is a closed subset of $X$, then $Z^{\mathrm{ad}} = Z^{\mathrm{val}} \cap X^{\mathrm{ad}}$.

\begin{prop}
A morphism of schemes $Y \to X$ satisfies the valuative criterion for properness if and only if
\[
Y^{\mathrm{ad}} \simeq Y_{X^{\mathrm{ad}}}.
\]
\end{prop}

\begin{proof}
The condition means that any commutative diagram
\[
\xymatrix{V \ar@/^/[rrd] \ar@/_/[rdd] \ar@{-->}[rd] \\ &Y^{\mathrm{ad}}\ar[r] \ar[d] & X^{\mathrm{ad}} \ar[d] \\ & Y^{\mathrm{val}} \ar[r] & X^{\mathrm{val}},}
\]
may be uniquely completed.
Using the universal properties of our functors, we see that this is equivalent to the same assertion about the commutative diagram
\[
\xymatrix{V \ar[d] \ar[r] & V^+ \ar[d] \ar@{-->}[ld]\\ Y \ar[r] & X.}
\]
Using the very definition of $V^+$, we may actually assume that $V = \mathrm{Spa}(K,K^+)$ is a Huber point.
We are then reduced to the same assertion relative to the commutative diagram
\[
\xymatrix{\mathrm{Spec}(K) \ar[d] \ar[r] & \mathrm{Spec}(K^+) \ar[d] \ar@{-->}[ld]\\ Y \ar[r] & X.}
\]
This is exactly the valuative criterion for properness.
\end{proof}

\subsection{Analytic space associated to a formal scheme}

We consider here the analytic locus of the adic space associated to a formal scheme.
Be careful that this is not functorial in general.

We will denote by $P^{\mathrm{an}}$ the analytic part of $P^{\mathrm{ad}}$ (instead of ${P^{\mathrm{ad}}}^{\mathrm{an}}$).

\begin{prop} \label{bup}
If $u \colon P' \to P$ is the blowing up of a \emph{usual} closed subscheme $Z$ in a formal scheme $P$ and $Z' = u^{-1}(Z) := P' \times_{P} Z$, then $u$ induces an isomorphism $P'^{\mathrm{ad}} \setminus Z'^{\mathrm{ad}}\simeq P^{\mathrm{ad}} \setminus Z^{\mathrm{ad}}$.
In particular, it induces an isomorphism $P'^{\mathrm{an}} \simeq P^{\mathrm{an}}$.
\end{prop}

\begin{proof}
Assume first that $u \colon X' \to X$ is the blowing up of a closed subscheme $Z$ in a \emph{scheme} $X$.
Then, it induces an isomorphism $X' \setminus Z' \simeq X \setminus Z$ from which we can deduce an isomorphism
\[
X'^{\mathrm{val}} \setminus Z'^{\mathrm{val}} = (X' \setminus Z')^{\mathrm{val}} \simeq (X \setminus Z)^{\mathrm{val}} = X^{\mathrm{val}} \setminus Z^{\mathrm{val}}.
\]
We want to show now that the more restrictive isomorphism $X'^{\mathrm{ad}} \setminus Z'^{\mathrm{ad}}\simeq X^{\mathrm{ad}} \setminus Z^{\mathrm{ad}}$ also holds.
Only surjectivity needs to be checked, and since $X'^{\mathrm{ad}} \to X^{\mathrm{ad}}$ is proper, it is sufficient to consider maximal points.
But the maximal points of $X^{\mathrm{ad}}$ and $X'^{\mathrm{ad}}$ are in bijection with the maximal points of $X$ and $X'$ respectively, and we can use our isomorphism $X' \setminus Z' \simeq X \setminus Z$ in order to lift them.
In general, the question is local and we may assume that $P = X^{/Y}$ is the completion of a scheme $X$ along a closed subscheme $Y$ containing $Z$.
Since the map $u$ is adic, we obtain an isomorphism $P'^{\mathrm{ad}} \setminus Z'^{\mathrm{ad}}\simeq P^{\mathrm{ad}} \setminus Z^{\mathrm{ad}}$.
The last assertion follows from the fact that $Z^{\mathrm{an}} = \emptyset$ and $Z'^{\mathrm{an}} = \emptyset$ (note that $Z'$ is a \emph{usual} subscheme because a blowing up is an \emph{adic} map).
\end{proof}

Alternatively, in the previous proof, one may assume from the beginning that $P$ is the completion of an affine scheme and do the computations locally.
Anyway, it should be noticed that the inclusion $(P \setminus Z)^{\mathrm{ad}} \subset P^{\mathrm{ad}} \setminus Z^{\mathrm{ad}}$ is \emph{strict} in general (take $P = \mathrm{Spec}(\mathbb Z)$ and $Z = \mathrm{Spec}(\mathbb F_{p})$ for example) and this is why the above proof needs some care.

The next step is the theorem of Raynaud that allows a direct study of the analytic space $P^{\mathrm{an}}$ associated to $P$ (see \cite{Abbes10} and \cite{FujiwaraKato18}) (without introducing $P^{\mathrm{ad}}$):

\begin{thm}[Raynaud]
If $P$ is a quasi-compact formal scheme, then specialization induces an isomorphism of locally topologically ringed spaces
\[
P^{\mathrm{an},+} \simeq \varprojlim_{P' \to P} P'
\]
where $P' \to P$ runs through all blowing ups of usual subschemes.
\end{thm}

\begin{proof}
For the details, we refer to part II, appendix A of \cite{FujiwaraKato18}.
First of all, the existence of the map follows from proposition \ref{bup}. 
In order to show that this is an isomorphism, after blowing up an ideal of definition and localizing, we may assume that $P = \mathrm{Spf}(A)$ for some adic ring $A$ and that there exists a principal ideal of definition $(\pi)$.
In this case, we have $P^{\mathrm{an}} = \mathrm{Spa}(A[\frac 1\pi])$.
Now, if we denote by $V$ the right hand side of our isomorphism, and pick up some $v \in V$, then one first show that the completion $\widehat{\mathcal O}_{V,v}$ of the local ring is a valuation ring (theorem 0.9.11 of \cite{FujiwaraKato18}).
By composition, this provides a valuation on $A[\frac 1\pi] $. 
This way, we obtain an inverse for our map.
It is not hard to finish the proof because blowing up an open ideal in $A$ provides a rational covering on the other side (and conversely).
\end{proof}

Actually, Raynaud's theorem says a lot more.
Recall that we call \emph{rigid} an adic space which is locally of the form $\mathrm{Spa}(A)$ (meaning that $A^+ = A^\circ$) for some Huber ring $A$.
One can show that an adic space is rigid if and only if it is locally isomorphic to an open subset of some $P^{\mathrm{ad}}$ where $P$ is a formal scheme.
In its strong form, Raynaud's theorem states that the functor $P \mapsto P^{\mathrm{an}}$ induces an equivalence between the category of quasi-compact formal schemes and adic maps up to blowing up and the category of quasi-compact quasi-separated rigid analytic spaces (in our sense).

\subsection{Analytic properties of formal schemes}

As already mentioned, the ``functor'' $P \mapsto P^{\mathrm{an}}$ is only functorial in \emph{adic} maps.
However, non adic maps will play an important role in our constructions and this is why we cannot simply rely on this more restrictive notion.

Before stating the next lemma, recall that the (resp.\ the analytic) valuative criterion for properness requires the existence of a lifting in the following situation ($R$ a valuation ring)
\[
\xymatrix{Q \ar[r] & P \\ \mathrm{Spec}(F) \ar[r] \ar[u] & \mathrm{Spec}(R) \ar[u] \ar@{-->}[ul]} \quad (\mathrm{resp.}\ \xymatrix{W \ar[r] & V \\ \mathrm{Spa}(F, F^+) \ar[r] \ar[u] & \mathrm{Spa}(F,R) \ar[u] \ar@{-->}[ul]).}
\]

\begin{lem} \label{valcri}
If $u \colon Q \to P$ is a morphism of formal schemes that satisfies the valuative criterion for properness, then $u^{\mathrm{ad}} \colon Q^{\mathrm{ad}} \to P^{\mathrm{ad}}$ satisfies the analytic valuative criterion for properness.
\end{lem}

\begin{proof}
We give ourselves a non-archimedean field $F$ and valuation rings $R \subset F^+ \subsetneq F$.
We start with a morphism $\mathrm{Spa}(F, F^+) \to Q^{\mathrm{ad}}$ and we denote by $w$ the image of the closed point.
Without loss of generality, we may assume that $F$ is complete.
Then, by definition, there exists a morphism $\mathcal O_{w} \to F$ which is local and compatible with the valuations (when $F$ is endowed with the valuation associated to $F^+$).
It follows that $\kappa(w) \subset F$ and
\[
F^+ \cap \kappa(w) = \kappa(w)^+ = \{\alpha \in \kappa(w) \ /\ w(\alpha) \geq 0\}
\]
is the valuation ring of $\kappa(w)$.
In particular, we may as well assume from now on that $F = \kappa(w)$ and $F^+ = \kappa(w)^+$ (and we replace $R$ with $R \cap \kappa(w) \subset \kappa(w)^+$).

We consider now the composition of the local morphisms of local rings
\[
\mathcal O_{\mathrm{sp}(w)} \to \mathcal O^+_{w} \twoheadrightarrow \kappa(w)^+.
\]
We denote by $S \subset \mathcal O_{\mathrm{sp}(w)}$ the inverse image of $R\subset \kappa(w)^+$ (see diagram below) and by $\overline S \subset \kappa(\mathrm{sp}(w))$ the image of $S$ which is a valuation ring of $\kappa(\mathrm{sp}(w))$.
In order to make sure that $\overline S$ is a valuation ring, one may consider the image $\overline R$ of $R$ in the residue field of $\kappa(w)^+$, which is a valuation ring, and notice that $\overline S = \overline R \cap \kappa(\mathrm{sp}(w))$.

We suppose now that the composite map $\mathrm{Spa}(\kappa(w), \kappa(w)^+) \to Q^{\mathrm{ad}} \to P^{\mathrm{ad}}$ factors through some morphism $\mathrm{Spa}(\kappa(w), R) \to P^{\mathrm{ad}}$ and we denote by $v$ the image of the closed point under this last map.
Then, as before, we have $\kappa(v) \subset \kappa(w)$ and $\kappa(v)^+ = R \cap \kappa(v)$.
Thus, if we consider the commutative diagram of local morphisms of local rings
\[
\xymatrix{\mathcal O_{\mathrm{sp}(v)} \ar[r] \ar[d] & \mathcal O^+_{v} \ar@{->>}[r] \ar[d] & \kappa(v)^+ \ar@{=}[r] \ar[d] & R \cap \kappa(v) \\ \mathcal O_{\mathrm{sp}(w)} \ar[r] & \mathcal O^+_{w} \ar@{->>}[r] & \kappa(w)^+ \\ S \ar[rr] \ar@{^{(}->}[u] && R, \ar@{^{(}->}[u]}
\]
we see that the image of $\mathcal O_{\mathrm{sp}(v)}$ in $\mathcal O_{\mathrm{sp}(w)}$ is contained in $S$.
It follows that its image in $\kappa(\mathrm{sp}(w))$ is contained in $\overline S$.
From this, we deduce the existence of a morphism $\mathrm{Spec}(\overline S) \to P$ that we may uniquely lift as $\mathrm{Spec}(\overline S) \to Q$ because we assume that $u$ satisfies the valuative criterion for properness.
Let us denote by $\mathfrak q$ the image of the closed point (so that, by construction, $\mathrm{sp}(v) = u(\mathfrak q)$).

After replacing $Q$ with an affine neighborhood of $\mathfrak q$, we may assume that $Q = \mathrm{Spf}(B)$.
The lifting property tells us that the image of the natural map $B \to \kappa(\mathrm{sp}(w)$ is contained in $\overline S$.
A quick look at the above diagram and we see that the image of the natural map $B \to \kappa(w)^+$ is contained in $R$ and we are done.
\end{proof}

Recall that the concept of being analytically something was introduced in definition \ref{propan}.
Then we have the following:

\begin{thm} \label{parpr}
If $u \colon Q \to P$ is a morphism of formal schemes which is formally locally of finite type (resp.\ formally locally of finite type and separated, resp.\ partially proper) then $u^{\mathrm{ad}} \colon Q^{\mathrm{ad}} \to P^{\mathrm{ad}}$ is analytically locally of finite type, resp.\ analytically separated, resp.\ analytically partially proper).
\end{thm}

\begin{proof}
We consider the first assertion.
This is a local question and we may therefore assume that $u$ splits into a closed immersion followed with the projection $\mathbb A^n \times \mathbb A^{-,m} \times P \to P$.
We are therefore reduced to the case where $u$ is either a closed immersion, the projection of the affine line $\mathbb A$ onto $\mathrm{Spec}(\mathbb Z)$, or the projection of the formal affine line $\mathbb A^-$ onto $\mathrm{Spec}(\mathbb Z)$.
In the first two cases, we already know that $u^{\mathrm{ad}}$ is of finite type.
It is therefore sufficient to recall that $\mathbb D^-_{V}$ is locally of finite type over $V$ when $V$ is analytic.
More precisely, we may assume that $V$ is Tate ring with topologically nilpotent unit $\pi$ in which case
\[
\mathbb D^-_{V} = \bigcup_{n \in \mathbb N} \mathbb D_{V}(0, \pi^{\frac 1n})
\]
is a (increasing) union of affinoid open subsets of finite type.

Assume now that $u$ is also separated which means that the diagonal embedding $Q \hookrightarrow P \times_{Q} P$ is a closed immersion.
First of all, since $u$ is locally formally of finite type, this is a locally noetherian morphism, and the fibered product is representable by a locally noetherian formal scheme.
It follows that the map $Q^{\mathrm{ad}} \hookrightarrow P^{\mathrm{ad}} \times_{Q^{\mathrm{ad}}} P^{\mathrm{ad}}$ is also a closed immersion (but we may not call $u^{\mathrm{ad}}$ separated because our definition requires a finiteness condition).
Pulling back preserves products and closed immersions.
It follows that $u^{\mathrm{ad}}$ is analytically separated.

We assume now that $u$ is partially proper and we consider the pullback $(u^{\mathrm{ad}})^{-1}(V) \to V$ of the morphism $u^{\mathrm{ad}} \colon Q^{\mathrm{ad}} \to P^{\mathrm{ad}}$ along some morphism $V \to P^{\mathrm{ad}}$ with $V$ analytic.
We already know that our map is locally of finite type and separated and we apply the \emph{analytic valuative criterion}.
It is actually sufficient to show that the map $u^{\mathrm{ad}}$ itself satisfies the analytic valuative criterion (using the universal property of fibered products) and this was done in lemma \ref{valcri}.
\end{proof}

Let us end this chapter with the following remark.
The above functors $P \mapsto P^{\mathrm{ad}}$, $V \mapsto V_{0}$ and $V \mapsto V^+$ (or even $X \mapsto X^{\mathrm{val}}$) extend naturally to the context of generalized adic spaces of Scholze-Weinstein (without any noetherian hypothesis) but the adjunctions of proposition \ref{adzer} and \ref{adplus} (as well as proposition \ref{adjalg}) are not valid anymore in this full generality.
And there are many other issues as well.

\begin{xmp}
If $\mathcal V$ is a (non discrete) valuation ring with fraction field $K$ and $P$ is a formal scheme which is formally of finite type over $\mathcal V$, then $P^{\mathrm{ad}}$ is only a generalized adic space but $P^{\mathrm{ad}}_{K} := P^{\mathrm{ad}} \times_{\mathrm{Spa}(\mathcal V)} \mathrm{Spa}(K)$ is a (honest) analytic Huber space.
In other words, in the non noetherian case, we would mostly use generalized adic spaces as a bridge between formal schemes and analytic Huber spaces.
Note that one could also use directly Raynaud's generic fiber and completely avoid adic spaces but we would then have to stick to adic morphisms and this is not our philosophy here.
\end{xmp}

\section{Overconvergent spaces}

In this section, we first introduce the notion of a formal embedding and then extend the concept of an overconvergent space from \cite{LeStum11} to the absolute setting.
Recall that all formal schemes (resp.\ adic spaces) are supposed to be locally noetherian (resp.\ locally of noetherian type).

\subsection{Formal embeddings}

In the end, we are interested in usual schemes but we need to embed them into formal schemes in order to obtain and adic space and then an analytic space.
However, the theory works as well if we start directly from a general formal scheme and this provides actually a lot more flexibility.

\begin{dfn}
A \emph{formal embedding} is a locally closed embedding of formal schemes $X \hookrightarrow P$.
\end{dfn}

We will often identify $X$ with its image in $P$ and then call it a \emph{formal subscheme}.
Note that we do not require $X$ to be a usual scheme in this definition.
A formal embedding is a locally noetherian morphism.
Formal embeddings form a category with compatible pairs of morphisms
\[
\xymatrix{ Y \ar@{^{(}->}[r] \ar[d]^f & Q \ar[d]^u \\ X \ar@{^{(}->}[r] & P.}
\]
Of course, $f$ is uniquely determined by $u$ when it exists.
The forgetful functor $(X \hookrightarrow P) \mapsto P$ commutes with all limits and all colimits because it has obvious adjoint $P \mapsto (\emptyset \hookrightarrow P)$:
\[
\mathrm{Hom}(\emptyset \hookrightarrow Q, X \hookrightarrow P) \simeq \mathrm{Hom}(Q, P),
\]
and coadjoint $P \mapsto (P = P)$:
\[
\mathrm{Hom}(Q, P) \simeq \mathrm{Hom}(Y \hookrightarrow Q, P = P).
\]
And the forgetful functor $(X \hookrightarrow P) \mapsto X$ commutes with all limits because it has an obvious adjoint $X \mapsto (X = X)$:
\[
\mathrm{Hom}(Y = Y, X \hookrightarrow P) \simeq \mathrm{Hom}(Y, X)
\]
(but no coadjoint).

If we are given a topology on the category of formal schemes, then we will implicitly endow the category of formal embeddings with the inherited topology (coarsest topology making the forgetful functor $(X \hookrightarrow P) \to P$ cocontinuous).
We will usually consider the coarse topology so that the inherited topology is also the coarse topology but we may as well consider any other such as the Zariski or $h$-topology.
For example, the Zariski topology will be generated by the pretopology made of families $\{(X_{i} \hookrightarrow P_{i}) \to (X \hookrightarrow P)\}_{i \in I}$ where $P = \bigcup_{i \in I} P_{i}$ is a Zariski open covering and, for all $i \in I$, $X_{i} = X \cap P_{i}$.
One easily sees that this topology is subcanonical.
Moreover, both functors $(X \hookrightarrow P) \mapsto P$ and $(X \hookrightarrow P) \mapsto X$ are left exact continuous and cocontinuous.
There are similar descriptions and properties for the other topologies.

If $X$ is a locally closed subspace of a formal scheme $P$, then there always exists a structure of a formal subscheme on $X$ and we will often consider $X$ as endowed with such a structure.
Actually, there always exists a unique structure of reduced subscheme on $X$ but it might be more convenient sometimes to use another one.
This applies in particular to the closure $\overline X$ of a formal subscheme $X$ in $P$ that will always be seen as a formal subscheme for some structure (for example, the (formal) scheme theoretic closure).
Note that the closure map is functorial in the sense that any morphism of formal embeddings
\[
\xymatrix{ Y \ar@{^{(}->}[r] \ar[d]^f & Q \ar[d]^u \\ X \ar@{^{(}->}[r] & P.}
\]
will induce a morphism $\overline f : \overline Y \to \overline X$.

We will now prove two technical lemmas.

\begin{lem} \label{locnei}
Assume that we are given a morphism of formal embeddings
\[
\xymatrix{ X' \ar@{^{(}->}[r] \ar[d]^f & P' \ar[d]^u \\ X \ar@{^{(}->}[r] & P}
\]
with $f$ formally \'etale and $u$ differentially smooth in the neighborhood of $X'$.
Then, locally on $P$ and $P'$, $u$ factors through a morphism $v \colon P' \to \mathbb A^n_{P}$ which is formally \'etale in the neighborhood of $X$ and extends both $f$ and $\overline f$ when $X$ and $\overline X$ are embedded into $\mathbb A^n_{P}$ through the zero section:
\[
\xymatrix{ X' \ar@{^{(}->}[r] \ar[d]^f & \overline X' \ar@{^{(}->}[r] \ar[d]^{\overline f} & P' \ar[d]^v \ar@/^1cm/[dd]^u \\
X \ar@{^{(}->}[r] \ar@{=}[d] & \overline X \ar@{^{(}->}[r] \ar@{=}[d] & \mathbb A^n_{P} \ar[d] \\
X \ar@{^{(}->}[r] & \overline X \ar@{^{(}->}[r] & P.}
\]
\end{lem}

\begin{proof}
We follow the proof of theorem 1.3.7 of \cite{Berthelot96c*}.
Since $f$ is formally \'etale, there exists an isomorphism
\[
\check{\mathcal N}_{X'/u^{-1}(X)} \simeq \left(\Omega^1_{u^{-1}(X)/X}\right)_{|X'}
\]
between the conormal sheaf on one side and the restriction of the sheaf of differential forms on the other.
Now, since the question is local and $u$ is differentially smooth in the neighborhood of $X'$, we may assume that there exists a basis of the conormal sheaf which is induced by global sections $f_{1}, \ldots, f_{n}$ of the ideal $\mathcal I_{X'}$ that defines $X'$ in $u^{-1}(X)$.
After multiplication by a common factor, we may assume that the sections $f_{1}, \ldots, f_{n}$ also induce global sections of the ideal $\mathcal I_{\overline X'}$ defining $\overline X'$ in $u^{-1}(\overline X)$.
Lifting these sections to $P'$ provides a morphism $v \colon P' \to \mathbb A^n_{P}$ which extends $f$ and $\overline f$ (when $X$ and $\overline X$ are seen as a formal subschemes of $\mathbb A^n_{P}$ via the zero section).
By construction, the morphism $v^*\Omega^1_{\mathbb A^n_{P}/P} \to \Omega^1_{P'/P}$ is an isomorphism in the neighborhood of $X'$.
This implies that the map $v$ is formally \'etale in the neighborhood of $X'$.
\end{proof}

With a projective version of the same argument, we can prove the following:

\begin{lem} \label{smet2}
Assume that we are given a commutative diagram
\[
\xymatrix{Y \ar@{^{(}->}[r]^j \ar[d]^f & Z \ar[dr]^g \\ X \ar@{^{(}->}[rr] && P}
\]
where $f$ is \'etale, $g$ is locally quasi-projective and $j$ is a dense immersion.
It will extend locally on $Y$ and $P$ to a morphism of formal embeddings
\[
\xymatrix{Y \ar@{^{(}->}[r]^j \ar[d]^f & Z \ar[dr]^g \ar@{^{(}->}[r] & Q \ar[d]^u \\ X \ar@{^{(}->}[rr] && P}
\]
where $u$ is a projective morphism which is \'etale in the neighborhood of $Y$.
\end{lem}

\begin{proof}
This is taken from the proof of theorem 2.3.5 of \cite{Berthelot96c*}.
The question being local on $P$, we may assume that $Z$ is a formal subscheme of the projective space $\mathbb P^N_{P}$.
We may also assume that $P$ is affine.
As in the proof of lemma \ref{locnei}, there exists an isomorphism
\[
\check{\mathcal N}_{Y/\mathbb P^N_{X}} \simeq \left(\Omega^1_{\mathbb P^N_{X}/X}\right)_{|Y}.
\]
Since the question is local on $Y$ (but not on $Z$), we may assume that there exists a basis of the conormal sheaf which is defined on some open subset $U$ of $\mathbb P^N_{X}$.
We may actually assume that $U = V \cap \mathbb P^N_{\overline X}$ where $V := D^+(s)$ for some $s \in \Gamma(\mathbb P^N_{P}, \mathcal O(m))$ and that our basis comes from some sequence $f_{1}, \ldots, f_{N} \in \Gamma(\mathbb P^N_{P}, \mathcal O(n))$.
We set $Q := V(f_{1}, \ldots, f_{N}) \subset \mathbb P^N_{P}$.
The Jacobian criterion shows that $Q$ is \'etale in the neighborhood of $Y$.
\end{proof}

\subsection{Overconvergent spaces}

Overconvergent spaces encompass in the same object a usual (more generally a formal) scheme, which is what we want to understand and an analytic (more generally an adic) space on which we intend to do the computations.
They are the bricks on which we will build the theory.

\begin{dfn} \label{defov}
An \emph{overconvergent space} is a pair
\[
(X \hookrightarrow P, P^{\mathrm{ad}} \stackrel \lambda\leftarrow V)
\]
where $X \hookrightarrow P$ is a formal embedding and $\lambda \colon V \to P^{\mathrm{ad}}$ is a morphism of adic spaces.
We call it a \emph{convergent space} if $X$ is closed in $P$.
We will say \emph{analytic (over-) convergent space} when $V$ is analytic.
\end{dfn}

We may write
\[
\xymatrix{ X \ar@{^{(}->}[r] & P \ar@<2pt>@{^{(}->}[r] & P^{\mathrm{ad}} \ar@<2pt>[l] & V \ar[l]_\lambda}
\]
but we will usually make it shorter as $(X \hookrightarrow P \leftarrow V)$.
Recall from proposition \ref{adplus}, that giving the morphism $\lambda \colon V \to P^{\mathrm{ad}}$ is equivalent to giving a morphism of locally topologically ringed spaces $V^+ \to P$ (as usual, we write $V^+ := (V, \mathcal O_{V}^+)$).
In particular, we could remove the locally noetherian condition in definition \ref{defov} by replacing the morphism $V \to P^{\mathrm{ad}}$ with the equivalent morphism $V^+ \to P$ (which always exists).

We will call \emph{specialization} the composition $\mathrm{sp}_{V}$ of the morphism $\lambda : V \to P^{\mathrm{ad}}$ (seen as a morphism of locally topologically ringed spaces) with usual specialization map $\mathrm{sp} \colon P^{\mathrm{ad}} \to P$ (or equivalently the composite map $V \to V^+ \to P$).

Note that, in the definition of an analytic overconvergent space, we only require that $V$ is analytic and not at all that the map $\lambda$ factors through $P^{\mathrm{an}}$.
It may even happen that the analytic locus of $P$ is empty.

\begin{xmp}
Let $K$ be a complete discretely valued field with topologically nilpotent unit $\pi$ (e.g.\ a uniformizer).
We denote by $\mathcal V$ the valuation ring of $K$ and by $k$ its residue field.
\begin{enumerate}
\item
We endow $\mathcal V$ with the $\pi$-adic topology.
We embed $\mathrm{Spec}(k)$ into $\mathrm{Spf}(\mathcal V)$ and let $V = \mathrm{Spa}(K)$ ($\lambda$ is the inclusion map).
Then,
\[
(\mathrm{Spec}(k) \hookrightarrow \mathrm{Spf}(\mathcal V) \leftarrow \mathrm{Spa}(K))
\]
is an analytic convergent space which is the usual basis for rigid cohomology.
\item Now, we endow $\mathcal V[[t]]$ with the \emph{$\pi$-adic topology} (and \emph{not} the $(\pi,t)$-adic topology).
We embed $\eta_{k} := \mathrm{Spec}(k((t)))$ into $\mathbb A^{\mathrm{b}}_{\mathcal V} := \mathrm{Spf}(\mathcal V[[t]])$ and recall that
\[
\mathbb D^{\mathrm{b}}_{K} = \mathrm{Spa}(K \otimes_{\mathcal V} \mathcal V[[t]])
\]
denotes the bounded unit disk.
Then
\[
(\eta_{k} \hookrightarrow \mathbb A^{\mathrm{b}}_{\mathcal V} \leftarrow \mathbb D^{\mathrm{b}}_{K})
\]
is an analytic overconvergent space which is the refined basis for rigid cohomology over a Laurent series field (see \cite{LazdaPal16}).
\end{enumerate}
\end{xmp}

An overconvergent space can be used to transfer a formal scheme into the adic world.
The following will need to be refined later:

\begin{dfn}
The \emph{naive tube} of an overconvergent space $(X \hookrightarrow P \leftarrow V)$ is the subset
\[
]X[_V^{\mathrm{naive}} := \mathrm{sp}_V^{-1}(X) \subset V.
\]
\end{dfn}

Since $\mathrm{sp}$ is continuous, $]X[_V^{\mathrm{naive}}$ is a locally closed subset of $V$ that will be endowed with the induced topology.
This is not an adic space in general unless $X$ is open in $V$.
This is an open (resp. a closed) subset when $X$ is open (resp. closed) in $P$.
It is also completely formal to see that
\[
\left]\bigcup_i X_i\right[_V^{\mathrm{naive}} = \bigcup_i ]X_i[_V^{\mathrm{naive}}\quad,  \left]\bigcap_i X_i\right[_V^{\mathrm{naive}} = \bigcap_i ]X_i[_V^{\mathrm{naive}}
\]
and $]Y[_V^{\mathrm{naive}}  \subset ]X[_V^{\mathrm{naive}}$ when $Y \subset X$.
As we shall see later, this naive tube is \emph{not} the right object and will have to be replaced by the genuine tube.

The following notion also will show up:

\begin{dfn} \label{fiber}
The \emph{fiber} of a \emph{convergent} space $(X \hookrightarrow P \leftarrow V)$ is the closed subspace
\[
X_V := X^{\mathrm{ad}} \times_{P^{\mathrm{ad}}} V
\]
(which should not be confused with the notion of a relative scheme of corollary \ref{relsch}.).
\end{dfn}

The same construction also works in the case of a general overconvergent space (not necessarily convergent) but it should then be called the \emph{naive fiber}.
For example, the naive fiber of an open formal subscheme is identical to the naive tube.
We will generalize later the notion of a fiber to the case of an \emph{overconvergent} space.

Be careful that, unlike the naive tube (or the genuine tube that we will consider below), the fiber depends on the formal scheme structure of $X$.

\begin{xmp}
 We use the ``$p$-adic'' topology everywhere.
\begin{enumerate}
\item
If $X := \mathrm{Spf}(\mathbb Z_p)$, $P := \mathrm{Spf}(\mathbb Z_p)$ et $V := \mathrm{Spa}(\mathbb Q_p)$, then $X_V =\mathrm{Spa}(\mathbb Q_p)$.
\item
If $X := \mathrm{Spec}(\mathbb F_p)$, $P := \mathrm{Spf}(\mathbb Z_p)$ and $V := \mathrm{Spa}(\mathbb Q_p)$, then $X_V = \emptyset$.
\item
If we embed $X := \mathrm{Spf}(\mathbb Z_p)$ into $P := \mathbb A_{\mathbb Z_p}$ using the zero section and we let $V := \mathbb D_{\mathbb Q_p}$, then $X_V =\mathrm{Spa}(\mathbb Q_p)$.
\item
If $X := \mathbb A_{\mathbb Z_p}$, $P := \mathbb P_{\mathbb Z_p}$ and $V := \mathbb P_{\mathbb Q_p}$, then $X_V =\mathbb D_{\mathbb Q_p}$.
\end{enumerate}
\end{xmp}

\subsection{Formal morphisms}

We will study here a rough version of the overconvergent site that will be defined later as a localization of the category of overconvergent spaces and formal morphisms.

\begin{dfn}
A \emph{formal morphism} of overconvergent spaces
\[
(Y \hookrightarrow Q, Q^{\mathrm{ad}} \stackrel \mu\leftarrow W) \to (X \hookrightarrow P, P^{\mathrm{ad}} \stackrel \lambda\leftarrow V)
\]
is a triple of morphisms $f \colon Y \to X, v \colon Q \to P, u \colon W \to V$ making commutative the diagrams
\[
\xymatrix{ Y \ar@{^{(}->}[r] \ar[d]^f & Q \ar[d]^v \\ X \ar@{^{(}->}[r] & P}, \quad \xymatrix{Q^{\mathrm{ad}} \ar[d]^{v^{\mathrm{ad}}}& W \ar[l]_\mu \ar[d]^u \\ P^{\mathrm{ad}} & V \ar[l]_\lambda }
\]
\end{dfn}

We might sometimes draw a full diagram
\[
\xymatrix{ Y \ar@{^{(}->}[r] \ar[d]^f & Q \ar@<2pt>@{^{(}->}[r] \ar[d]^{v}& Q^{\mathrm{ad}} \ar@<2pt>[l] \ar[d]^{v^{\mathrm{ad}}} & W \ar[l]_\mu \ar[d]^{u}\\
X \ar@{^{(}->}[r] & P \ar@<2pt>@{^{(}->}[r] & P^{\mathrm{ad}} \ar@<2pt>[l] & V \ar[l]_\lambda}
\]
or simply write
\[
(f, u, v) : (Y \hookrightarrow Q \leftarrow W) \to (X \hookrightarrow P \leftarrow V).
\]
The condition is actually equivalent to requiring the commutativity of the diagram
\[
\xymatrix{Y \ar@{^{(}->}[r] \ar[d]^f & Q\ar[d]^{v} & W \ar[l]_{\mathrm{sp}_{W}} \ar[d]^-{u} \\ X \ar@{^{(}->}[r] & P & V. \ar[l]_{\mathrm{sp}_{V}} }
\]

\begin{xmp}
Let $K$ be a complete discretely valued field of mixed characteristic $p$ with valuation ring $\mathcal V$ and residue field $k$.
Then, the \emph{Amice ring} $\mathcal A := \widehat{\mathcal V[[t]][\frac 1t]}$ is a complete discrete valuation ring with residue field $k((t))$ and its fraction field is the \emph{Amice field} $\mathcal E = \mathcal A[\frac 1p]$. 
We may then consider the formal morphism of overconvergent spaces
\[
\xymatrix{\mathrm{Spec}(k((t))) \ar@{^{(}->}[r] \ar@{=}[d] & \mathrm{Spf}(\mathcal A)\ar[d] & \mathrm{Spa}(\mathcal E) \ar[l] \ar[d] \\ \eta \ar@{^{(}->}[r] & \mathbb A^{\mathrm{b}}_{\mathcal V} & \mathbb D^{\mathrm{b}}_{K} . \ar[l] }
\]
This is the morphism that refines the usual basis for rigid cohomology over $k((t))$.
The map induced on the tubes (see below) will be given by the inclusion of the bounded Robba ring $\mathcal E^\dagger$ (also denoted by $\mathcal R^{\mathrm{b}}$) into the Amice field $\mathcal E$.
\end{xmp}

The following observation will allow us to split some proofs in two separate cases:

\begin{prop} \label{decom}
Any formal morphism of overconvergent spaces
\[
\xymatrix{Y \ar@{^{(}->}[r] \ar[d]^f & Q\ar[d]^{v} & W \ar[l]_{\mu} \ar[d]^-{u} \\ X \ar@{^{(}->}[r] & P & V \ar[l]_\lambda}
\]
is the composition of a formal morphism with $u = \mathrm{Id}_{V}$ and another one with both $f = \mathrm{Id}_{X}$ and $v = \mathrm{Id}_{P}$.
\end{prop}

\begin{proof}
It is sufficient to split our morphism as follows:
\[
\begin{gathered}
\xymatrix{ Y \ar@{^{(}->}[r] \ar[d]^f & Q \ar[d]^{v}& W \ar[l]_-{\mu} \ar@{=}[d]\\
X \ar@{^{(}->}[r] \ar@{=}[d] & P \ar@{=}[d]& W \ar[l]_-{v^{\mathrm{ad}} \circ \mu} \ar[d]^{u}\\
X \ar@{^{(}->}[r] & P & V. \ar[l]_\lambda}
\end{gathered}
 \qedhere
\]
\end{proof}

Endowed with formal morphisms, overconvergent spaces form a category (that we will need to refine later).
The functor $(X \hookrightarrow P \leftarrow V) \mapsto (X \hookrightarrow P)$ has an adjoint $(X \hookrightarrow P) \mapsto (X \hookrightarrow P \leftarrow \emptyset)$:
\[
\mathrm{Hom}(Y \hookrightarrow Q \leftarrow \emptyset, X \hookrightarrow P \leftarrow V) \simeq \mathrm{Hom}(Y \hookrightarrow Q, X \hookrightarrow P),
\]
and a coadjoint $(X \hookrightarrow P) \mapsto (X \hookrightarrow P \leftarrow P^{\mathrm{ad}})$:
\[
\mathrm{Hom}(Y \hookrightarrow Q, X \hookrightarrow P) \simeq \mathrm{Hom}(Y \hookrightarrow Q \leftarrow W, X \hookrightarrow P \leftarrow P^{\mathrm{ad}}).
\]
In particular, it commutes with all limits and all colimits.
By composition, we see that the functor $(X \hookrightarrow P \leftarrow V) \mapsto P$ has an adjoint $P \mapsto (\emptyset \hookrightarrow P \leftarrow \emptyset)$ and a coadjoint $P \mapsto (P = P \leftarrow P^{\mathrm{ad}})$ and that the functor $(X \hookrightarrow P \leftarrow V) \mapsto X$ has an adjoint $X \mapsto (X = X \leftarrow \emptyset)$ (but no coadjoint).
On the other hand, the functor $(X \hookrightarrow P \leftarrow V) \mapsto V$ has a coadjoint $V \mapsto (\mathrm{Spec}(\mathbb Z) = \mathrm{Spec}(\mathbb Z) \leftarrow V)$:
\[
\mathrm{Hom}(W, V) \simeq \mathrm{Hom}(Y \hookrightarrow Q \leftarrow W, \mathrm{Spec}(\mathbb Z) = \mathrm{Spec}(\mathbb Z) \leftarrow V)
\]
and commutes therefore with all colimits.
There exists no coadjoint to this forgetful functor but one can check directly from the definition that it also commutes with all limits when they exist.
More precisely, if we are given a diagram 
\[
\{(X_{i} \hookrightarrow P_{i} \leftarrow V_{i})\}_{i \in I},
\]
and we assume that $\varprojlim X_{i}$, $\varprojlim P_{i}$ and $\varprojlim V_{i}$ exist, then our diagram has a limit which is
\[
(\varprojlim X_{i} \hookrightarrow \varprojlim P_{i} \leftarrow \varprojlim V_{i}).
\]
We should mostly apply this to fibered products (when they exist).

We endow the category of overconvergent spaces and formal morphisms with the \emph{adic} topology: this is the topology inherited from $V$ (the coarsest topology making cocontinuous the forgetful functor $(X \hookrightarrow P \leftarrow V) \mapsto V$).
It is generated by the pretopology made of families
\[
\{(X \hookrightarrow P \leftarrow V_{i}) \to (X \hookrightarrow P \leftarrow V)\}_{i \in I}
\]
where $V = \bigcup_{j \in j} V_{i}$ is an open covering.
This topology is subcanonical.
Moreover, the functor $(X \hookrightarrow P \leftarrow V) \mapsto V$ is left exact continuous and cocontinuous, giving rise to two morphisms of topoi.

Unless otherwise specified, we always consider the category of adic spaces as a site with respect to the adic topology.
If we use another topology such as the \'etale topology for example, then we can consider the corresponding topology on the category of overconvergent spaces and formal morphisms which is defined exactly as before.
On the other hand, we may also endow our category with the image topology of the functor
\[
(X \hookrightarrow P) \mapsto (X \hookrightarrow P \leftarrow P^{\mathrm{ad}})
\]
where the first category is endowed for example with the Zariski topology (or some other topology such that the $h$-topology if we wish).
We obtain the topology generated by the pretopology made of families
\[
\{(X_{i} \hookrightarrow P_{i} \leftarrow V_{i}) \to (X \hookrightarrow P \stackrel \lambda\leftarrow V)\}_{i \in I}
\]
where $P = \bigcup_{j \in j} P_{i}$ is an open covering, and for each $i \in I$, $X_{i} = X \cap P_{i}$ and $V_{i} = \lambda^{-1}(P_{i}^\mathrm{ad})$.
Finally, note that it is also possible to endow the category of overconvergent spaces and formal morphisms with a topology coming both from the adic and the formal side (coarsest topology finer than both).
This would give rise for example to the Zariski-adic topology.

\subsection{Formal completion}

Formal completion is usually only defined for usual schemes.
We extend it to formal schemes as follows:

\begin{dfn}
Let $X \hookrightarrow P$ be a closed formal embedding.
If $\mathcal I_{X}$ denotes the ideal of $X$ in $P$ and $\mathcal I$ is some ideal of definition of $P$, we let $X_{n}$ denote the closed subscheme of $P$ defined by $(\mathcal I_{X} + \mathcal I)^{n+1}$.
Then the \emph{completion} $P^{/X}$ of $P$ along $X$ is $P^{/X} := \varinjlim X_{n}$.
\end{dfn}

It is not difficult to check that this definition is independent of the choice of the ideal of definition.
Note that if we did not have noetherian hypothesis, it would be necessary to assume that the closed embedding is locally (radically) finitely presented in order to obtain an adic formal scheme.
Note also that completion is usually written with an extra hat as $\widehat P^{/X}$ but we will rather not do that.
Finally, it is not hard to see that $P^{/X}$ is (representable as) a formal scheme: when $P = \mathrm{Spf}(A)$ and $X = \mathrm{Spf}(A/\mathfrak a)$, we will have $ P^{/X} = \mathrm{Spf}(A^{/\mathfrak a})$ where $A^{/\mathfrak a}$ is the ring $A$ endowed with the $\mathfrak a + I$-adic topology where $I$ is some ideal of definition for $A$.

It is possible to define completion in the more general case of a locally closed embedding but we want to avoid this because this would only create confusion later, once we introduce the notion of a tube.

There exists a natural map $ P^{/X} \to P$ and the inclusion map $X \hookrightarrow P$ factors through a closed embedding $X \hookrightarrow P^{/X}$ which is actually a formal thickening (and in particular a \emph{homeomorphism}).
The map $ P^{/X} \to P$ is formally \'etale and locally noetherian.

The functor $(X \hookrightarrow P) \mapsto P^{/X}$ commutes with all limits because it has an adjoint $P \mapsto (P_{\mathrm{red}} \hookrightarrow P)$.
Since this last functor is fully faithful, the map $ P^{/X} \to P$ is a monomorphism (although not an embedding in general).
Finally, one easily checks that the functor $(X \hookrightarrow P) \mapsto P^{/X}$ is left exact continuous and cocontinuous (with respect to the Zariski topology for example).

Let us state some other basic properties of completion:

\begin{prop} \label{formp}
Let $X$ be a closed formal subscheme of a formal scheme $P$.
\begin{enumerate}
\item
The formal scheme $ P^{/X}$ only depends on the underlying subspace of $X$ (and not on the structure of formal scheme).
\item
If $P = \bigcup_{i \in I} P_{i}$ is an open covering and, for all $i \in I, X_{i} := P_{i} \cap X$, then we have an open covering
\[
 P^{/X} = \bigcup_{i \in I} P_{i}^{/X_{i}}.
\]
\item If $u \colon Q \to P$ is a morphism of formal schemes, then
\[
u^{-1}( P^{/X}) = Q^{/u^{-1}(X)}.
\]
\end{enumerate}
\end{prop}

\begin{proof}
All questions are local and the statements are easily checked.
\end{proof}

As a particular case of the last assertion, we obtain transitivity: if $Y$ is a closed formal subscheme of $P$ that contains $X$ and $Q := P^{/Y}$, then
\[
Q^{/X} = P^{/X}.
\]
This will allow us to shrink the ambient space in the future.

When $X$ is only a closed \emph{subspace} of a formal scheme $P$, we will denote by $P^{/X}$ the completion of $P$ with respect to any structure of formal subscheme on $X$.
For example, if $X$ is a formal subscheme and $\overline X$ denotes its closure in $P$, it has a meaning to consider $P^{/\overline X}$ because it does not depend on the structure of formal subscheme of $\overline X$.
Of course, we may as well use the schematic closure but this is not necessary (and sometimes it is more convenient to use another structure of formal subscheme).
We will now apply lemma \ref{locnei} to the case $X = \overline X$, $X' = \overline X'$ and $f$ is an isomorphism:

\begin{prop}[Formal fibration theorem] \label{forfib}
Assume that we are given two \emph{closed} formal embeddings $X \hookrightarrow P$ and $X' \hookrightarrow P'$.
If a differentially smooth morphism $u \colon P' \to P$ induces an isomorphism $X' \simeq X$, then it induces, \emph{locally on $P$ and $P'$}, an isomorphism
\[
 P'^{/X'} \simeq \mathbb A^{n,-} \times P^{/X}.
\]
\end{prop}

\begin{proof}
Using lemma \ref{locnei}, we may assume that there exists a morphism $v \colon P' \to \mathbb A^n_{P}$ which is formally \'etale and induces an isomorphism between $X'$ and the image of $X$ via the zero section.
We conclude with the forthcoming lemma \ref{lemun}.
\end{proof}

\begin{lem} \label{lemun}
In the situation of the proposition,
\begin{enumerate}
\item if $u$ is formally \'etale, then it induces an isomorphism $ P'^{/X'} \simeq P^{/X}$,
\item if $P' = \mathbb A^n_{P}$ and $X'$ is contained in the zero section, then $ P'^{/X'} \simeq \mathbb A^{n,-} \times P^{/X}$.
\end{enumerate}
\end{lem}

\begin{proof}
The first assertion is an immediate consequence of the definition of formal \'etaleness.
The second one follows from left exactness of completion of formal schemes.
\end{proof}

We want now to make precise the notion of a morphism having some property ``around'' a formal subscheme.
Actually, this splits into two cases: properties that are \emph{open} in nature (such as formally smooth for example) and properties that are \emph{closed} in nature (such as partially proper for example).
Note that some important properties are neither open or closed in nature (such as smooth which is at the same time formally smooth (open) and locally of finite type (closed)).

\begin{dfn} \label{formop}
A morphism of formal schemes $u \colon Q \to P$ is said to be \emph{flat} (resp.\ \emph{formally smooth}, resp.\ \emph{formally unramified}, resp.\ \emph{formally \'etale}) \emph{around} a formal subscheme $Y$ of $Q$ if there exists a neighborhood $Q'$ of $Y$ in $Q$ such that the induced map $Q' \to P$ has this property.
\end{dfn}

If we are given a morphism of formal embeddings
\[
\xymatrix{ Y \ar@{^{(}->}[r] \ar[d]^f & Q \ar[d]^u \\ X \ar@{^{(}->}[r] & P,}
\]
then we may also say that the morphism of formal embeddings itself is \emph{flat} (resp.\ \emph{formally smooth}, resp.\ \emph{formally unramified}, resp.\ \emph{formally \'etale}).

This is a standard notion and we turn now to the other case:

\begin{dfn} \label{formcl}
A morphism $u \colon Q \to P$ of formal schemes is said to be \emph{separated} (resp.\ \emph{affine}, resp.\ \emph{(locally) of finite type}, \emph{(locally) quasi-finite}, resp.\ \emph{partially proper}, resp.\ \emph{partially finite}) \emph{around a formal subscheme} $Y$ of $Q$ if there exists a closed formal subscheme $Z$ of $Q$ containing $Y$ such that $Q^{/Z} \to P$ is separated (resp.\ affine, resp.\ formally (locally) of finite type, resp.\ formally (locally) quasi-finite, resp.\ partially proper, resp.\ partially finite).
\end{dfn}

Again, if we are actually given a morphism of formal embeddings
\[
\xymatrix{ Y \ar@{^{(}->}[r] \ar[d]^f & Q \ar[d]^u \\ X \ar@{^{(}->}[r] & P}
\]
then we will say that the morphism of formal embeddings itself is \emph{separated} (resp.\ \emph{affine}, resp.\ \emph{(locally) of finite type}, \emph{(locally) quasi-finite}, resp.\ \emph{partially proper}, resp.\ \emph{partially finite}).
Actually, this is equivalent to $\overline f \colon \overline Y_{\mathrm{red}} \to \overline X_{\mathrm{red}}$ being separated (resp.\ affine, resp.\ (locally) of finite type, resp.\ (locally) quasi-finite, resp.\ locally of finite type and proper on irreducible components, resp.\ locally of finite type and finite on irreducible components).
This follows from the fact that we can replace $P$ and $Q$ by their completions along $\overline X$ and $\overline Y$ respectively without changing the conditions.
Note that we may then also replace $X$ (resp.\ $Y$) with the open formal subscheme of $P$ (resp.\ $Q$) having the same underlying space if we wish.

\subsection{Tubes}

A tube is meant to be an analytic neighborhood of a scheme (or more generally an adic neighborhood of a formal scheme).
We start with the following observation (recall that convergent means that $X \hookrightarrow P$ is closed):

\begin{lem}
If $(X \hookrightarrow P \leftarrow V)$ is a \emph{convergent} space, then the canonical map
\[
P^{/X,\mathrm{ad}} \times_{P^\mathrm{ad}} V \to V
\]
is a homeomorphism onto its image.
\end{lem}

\begin{proof}
We can assume that $V = P^{\mathrm{ad}}$.
Moreover, this is a local question and we may therefore assume that $P = \mathrm{Spf}(A)$ is affine and that $X$ is defined by an ideal $\mathfrak a \subset A$.
It is then sufficient to recall that the topology on $\mathrm{Spa}(A)$ and $\mathrm{Spa}(A^{/\mathfrak a})$ are both induced by the topology of $\mathrm{Spv}(A)$.
\end{proof}

\begin{xmp}
\begin{enumerate}
\item In the simplest non trivial case $X := \mathrm{Spec}(\mathbb F_{p})$, $P := \mathrm{Spec}(\mathbb Z)$ and $V := \mathrm{Spa}(\mathbb Z)$, 
we have $ P^{/X} = \mathrm{Spf}(\mathbb Z_{p})$ (with the $p$-adic topology) and therefore $P^{/X,\mathrm{ad}} = \mathrm{Spa}(\mathbb Z_{p})$.
In particular, (the homeomorphic image of) $P^{/X,\mathrm{ad}}$ is the closed subset of $V$ consisting in the the point $v_{p}$ and its (horizontal) specialization $p$.
\item
If we use the zero section to embed $X := \mathrm{Spec}(\mathbb Z)$ into $P := \mathbb A$ and let $V := \mathbb D$, then we have $P^{/X} = \mathbb A^-$ and therefore $P^{/X,\mathrm{ad}} = \mathbb D^-$.
Recall however that the inclusion $\mathbb D^- \subset \mathbb D$ is not a locally closed embedding (although it is \emph{analytically} an open embedding).
\end{enumerate}
\end{xmp}

\begin{dfn}
\begin{enumerate}
\item
Let $(X \hookrightarrow P \leftarrow V)$ be a convergent space.
The \emph{tube} of $X$ in $V$ is the (homeomorphic) image $\,]X[_V$ of 
\[
P^{/X,\mathrm{ad}} \times_{P^\mathrm{ad}} V \to V.
\]
\item
Let $(X \hookrightarrow P \leftarrow V)$ be an overconvergent space, $\overline X$ the closure of $X$ in $P$ and $\infty_{X} := \overline X \setminus X$ its locus at infinity.
Then, the \emph{tube} of $X$ in $V$ is the subspace
\[
\,]X[_{V} :=\,]\overline X[_{V} \setminus \,]\infty_{X}[_{V} \subset V.
\]
\end{enumerate}
\end{dfn}

Note that this notion is independent of the formal scheme structure on $X$.
Of course, both definitions coincide when the overconvergent space is actually convergent.
Be careful that even if $P$ does not usually appear in the notations, the tube $\,]X[_{V}$ also depends on $P$, and it might sometimes be necessary to write $\,]X[_{P,V}$ in order to remove the ambiguity.
On the contrary, when $V = P^{\mathrm{ad}}$, we will write $\,]X[_{P}$ and call it the \emph{tube of $X$ in $P$}.
In the convergent case, $\,]X[_{P}$ \emph{is} the homeomorphic image of $P^{/X,\mathrm{ad}}$ in $P^{\mathrm{ad}}$.

The tube $\,]X[_{V}$ is endowed with the subspace topology coming from the topology of $V$ (this is not an adic space).
When we consider it as a topologically locally ringed space, we always use the restriction sheaf $i_{X}^{-1}\mathcal O_{V}$ where $i_{X} \colon \,]X[_{V} \hookrightarrow V$ denotes the inclusion map.
Alternatively, we can consider the prepseudo-adic\footnote{This is not a pseudo-adic space in general because the tube need not be pro-constructible.} space $(V, \,]X[_V)$ whose structural sheaf is by definition $i_{X}^{-1}\mathcal O_{V}$.
Note that when $X$ is closed in $P$ (the convergent case), the homeomorphism $P^{/X,\mathrm{ad}} \simeq \,]X[_{P}$ is usually \emph{not} an isomorphism of locally topologically ringed spaces (see proposition \ref{carthat} however).
Be careful also that the \emph{naive tube} $]X[_V^{\mathrm{naive}} :=\mathrm{sp}_{V}^{-1}(X)$ is usually different from the genuine tube just introduced (more about this later).

\begin{xmp}
\begin{enumerate}
\item Recall from above that $]\mathrm{Spec}(\mathbb F_{p})[_{\mathrm{Spec}(\mathbb Z)}$ is a \emph{closed} subset of $\mathrm{Spv}(\mathbb Z)$ which is homeomorphic, but \emph{not} isomorphic, to $\mathrm{Spa}(\mathbb Z_{p})$ (where $\mathbb Z_{p}$ has the $p$-adic topology).
More precisely, a function on the tube is an element of the local ring $\mathbb Z_{(p)}$ which is smaller than the completed local ring $\mathbb Z_{p}$.
\item
With the convergent space
\[
(\mathrm{Spec}(\mathbb F_{p}) \hookrightarrow \mathrm{Spec}(\mathbb Z) \leftarrow \mathrm{Spa}(\mathbb Q_{p})),
\]
in which $\mathbb Q_{p}$ has the ``$p$-adic'' topology, we have
\[
]\mathrm{Spec}(\mathbb F_{p})[_{\mathrm{Spa}(\mathbb Q_{p})} = \mathrm{Spa}(\mathbb Q_{p}).
\]
\item
Let us consider more generally an overconvergent space $(X \hookrightarrow P \leftarrow O)$ where $O := \mathrm{Spa}(K, K^+)$ is a \emph{Huber point}.
Then, we have:
\begin{enumerate}
\item If $O$ is non-analytic, then $\,]X[_{O}$ may be any ``interval'' in $O$ (which is totally ordered by specialization).
\item If $O$ is analytic, then $\,]X[_{O} = \emptyset$ or $\,]X[_{O} = O$.
\end{enumerate}
\item
If we consider the convergent space $(\{0\} \hookrightarrow \mathbb A \leftarrow \mathbb D_{\mathbb Q_{p}})$, then we have
\[
]0[_{\mathbb D_{\mathbb Q_{p}}} = \mathbb D^-_{\mathbb Q_{p}} \subset \mathbb D_{\mathbb Q_{p}},
\]
which is an \emph{open} subset.
\item (Lazda and P\`al setting)
Let $K$ be a discretely valued field of mixed characteristic $p$ with valuation ring $\mathcal V$ and residue field $k$.
We embed the open point $X := \{\eta_{k}\}$ into $P :=\mathbb A^{\mathrm{b}}_{\mathcal V}$, and we let $V := \mathbb D^{\mathrm{b}}_{K}$.
The tube of the \emph{closed} point of $P$ is the open unit disk $\mathbb D^{-}_{K}$ and it follows that
\[
\,]X[_{V} = \mathbb D^{\mathrm{b}}_{K} \setminus \mathbb D^{-}_{K} = \overline {\{v\}} = \{v,v^-\},
\]
where $v$ is the \emph{Gauss point}.
A basis of neighborhoods of $\,]X[_{V}$ is given by the subsets $V_{n} := \mathbb D^{\mathrm{b}}_{K} \setminus \mathbb D_{K}(0, p^{\frac 1n})$.
It follows that
\[
\Gamma(\,]X[_{V}, i^{-1}\mathcal O_{V}) = \mathcal E_K^{\dagger}
\]
is the bounded Robba ring of $K$ (see \cite{LazdaPal16}).
\item (Monsky-Washnitzer setting)
Let $R$ be a noetherian ring and $X$ an affine \emph{scheme} of finite type over $R$.
From a presentation of $X$ over $R$, we obtain a sequence of inclusions $X \subset \mathbb A^n_{R} \subset \mathbb P^n_{R}$.
If $(C \hookrightarrow S \leftarrow O)$ is an overconvergent space and $S \to \mathrm{Spec}(R)$ is a morphism of formal schemes, then we may consider the overconvergent space
\[
X_{C} \hookrightarrow \mathbb P^n_{S} \leftarrow X_{O}.
\]
Then we have
\[
\,]X_{C}[_{X_{O}} = \,]\mathbb A^{n}_C[_{ \mathbb A^n_{O}} \cap X_{O} \subset \mathbb A^{n}_{O}.
\]
Classically, we choose $S = \mathrm{Spf}(R^{/I})$ and $C = \mathrm{Spec}(R/I)$ where $I \subset R$ is some ideal (and $R^{/I}$ denotes $R$ endowed with the $I$-adic topology).
We know from the amazing article \cite{Arabia01} of Alberto Arabia that reduction modulo $I$ is essentially surjective and full on smooth algebras.
In particular, any smooth affine scheme over $C$ will then have the form $X_C= \mathrm{Spec}(A/I)$ where $A$ is a smooth $R$-algebra.
Note also that, in this situation, we have $\,]X_{C}[_{X_{O}} = \,]X[_{X_{O}}$.
\item
We have the following decomposition (as disjoint union) of the adic projective line :
\[
\mathbb P = \underbrace{\mathbb D(0,1) \sqcup \{v_{p}^+, p\ \mathrm{prime}\}}_{]\mathbb A[} \sqcup\ \underbrace{\mathbb D^-(\infty, 1)}_{] \infty[}.
\]
\end{enumerate}
\end{xmp}

We can now generalize definition \ref{fiber} to a general overconvergent space:

\begin{dfn}
The \emph{fiber} of an overconvergent space $(X \hookrightarrow P \leftarrow V)$ is the subset
\[
X_V := \overline X_V \cap ]X_V[
\]
where $\overline X$ is the (formal) schematic closure of $X$.
\end{dfn}

Note that this is consistent with definition \ref{fiber} because
\[
X_V = X^{\mathrm{ad}} \times_{P^{\mathrm{ad}} } V \subset P^{/X,\mathrm{ad}} \times_{P^{\mathrm{ad}} } V \simeq ]X[_V
\]
when $X$ is closed in $V$.
It is also not difficult to see that $X_V = ]X[_{\overline X_V}$.
In particular, if $X$ is an open formal subscheme of $P$, then $X_V = ]X[_V$.
As it was the case with the tube, the fiber is not an adic space in general and we need to consider the prepseudo-adic space $(\overline X_V, X_V) = \overline X_V \cap (V, ]X[_V)$.

\subsection{Local description}

We will show here that the tube behaves well under boolean operations and give an explicit description in the local situation.

In order to lighten the notations, when we are given a fixed morphism of adic spaces $\lambda : V \to W$, some $v \in V$ and some function $f$ defined in a neighborhood of $\lambda(v)$ in $W$, we will simply write $v(f)$ instead of $v(\lambda^{-1}(f))$.

\begin{lem} \label{loctub}
Let $P := \mathrm{Spf}(A)$ be an affine formal scheme, $X \subset P$ the closed formal subscheme defined by an ideal $\mathfrak a \subset A$ and $\lambda : V \to P^{\mathrm{ad}}$ any morphism of adic spaces.
Then, we have
\[
\,]X[_{V}^{\mathrm{naive}} = \{v \in V \colon \forall f \in \mathfrak a, v(f) > 0\},
\]
\[
X_V = \{v \in V \colon \forall f \in \mathfrak a, v(f) = + \infty\},
\]
and
\[
\,]X[_{V} = \{v \in V \colon \forall f \in \mathfrak a,v(f^n) \to + \infty\}.
\]
\end{lem}

\begin{proof}
The first two cases immediately follow from the definitions.
Now, we have
\[
P^{/X,\mathrm{ad}} = \mathrm{Spa}(A^{/\mathfrak a}) \subset \mathrm{Spa}(A) = P^{\mathrm{ad}}.
\]
By definition, a point $v \in V$ belongs to $\,]X[_{V}$ if and only if $\lambda(v)$, which is a point in $P^{\mathrm{ad}}$, falls inside $P^{/X,\mathrm{ad}}$.
It means that $\lambda(v)$ is not only continuous for the topology of $A$ but also for the $\mathfrak a$-adic topology.
It exactly means that for any $f \in \mathfrak a$, we will have $v(f^n) \to + \infty$.
\end{proof}

It will sometimes be convenient to switch to multiplicative notation, or better, work directly inside the residue field, and we want to recall that
\[
v(f^n) \to + \infty \Leftrightarrow |f^n(v)| \to 0 \Leftrightarrow f^n(v) \to 0.
\]
In other words, $\{v(f^n)\}_{n \in \mathbb N}$ is unbounded in $G_{v}$ if and only if $f(v)$ is topologically nilpotent in $\kappa(v)$.

Lemma \ref{loctub} shows that $\,]X[_{V}\ \subset \,]X[_{V}^{\mathrm{naive}} $ when $X$ is \emph{closed}.
Unlike in Tate or Berkovich theory, this inclusion is usually \emph{strict} (as we shall see).
As a consequence, there is \emph{no} inclusion in general when $X$ is only assumed to be locally closed.
Some points of $\,]X[_{V}$ might therefore specialize \emph{outside} $X$ in general (but not too far as we shall also see).

\begin{xmp}
We let $X = \mathrm{Spf}(R)$ and consider the zero-section $X \hookrightarrow P := \mathbb A_{X}$.
If $v \in X^{\mathrm{ad}}$ is a non trivial valuation and $v^-$ denotes the corresponding valuation on $R[T]$ that specializes inside the unit disk, then we have $v^-(T) = (0, 1) > (0, 0)$ but, if the valuation $v$ is not trivial, we have $v^-(T^n) = (0, n) \not \to +\infty$ because $(0,n) < (g,0)$ whenever $g >0$ (we use the lexicographical order).
In particular, $]X[_{P}$ is strictly smaller than $\,]X[_{V}^{\mathrm{naive}}$.
\end{xmp}

From their local description in the closed case, we can deduce the boolean properties of the tubes:

\begin{prop} \label{oper}
Let $(X \hookrightarrow P \stackrel \lambda\leftarrow V)$ be an overconvergent space.
Then, we have the following:
\begin{enumerate}
\item if $Y \subset P$ is a formal subscheme such that $X \subset Y$, then $\,]X[_{V} \subset \,]Y[_{V}$,
\item if $X = X_{1} \cap X_{2}$ where $X_{1}, X_{2}$ are two other formal subschemes, then $\,]X[_{V} = \,]X_{1}[_{V} \cap \,]X_{2}[_{V}$,
\item if $X = X_{1} \cup X_{2}$ where $X_{1}, X_{2}$ are two other formal subschemes, then $\,]X[_{V} = \,]X_{1}[_{V} \cup \,]X_{2}[_{V}$.
\end{enumerate}
\end{prop}

\begin{proof}
Clearly, it is sufficient to consider the case $V = P^{\mathrm{ad}}$.
Let us first check that these assertions hold for \emph{closed} subspaces.
This is a local question on $P$ and we may therefore rely on lemma \ref{loctub}.
Assertion 1) should then be clear.
For the other two, we may assume that $P = \mathrm{Spf}(A)$ and $X, X_{1}, X_{2}$ are defined by $\mathfrak a, \mathfrak a_{1}, \mathfrak a_{2}$, respectively. 
In assertion 2) (resp.\ 3)), we suppose that $\mathfrak a = \mathfrak a_{1} + \mathfrak a_{2}$ (resp.\ $\mathfrak a = \mathfrak a_{1}\mathfrak a_{2}$).
It is then sufficient to notice that if $f, g \in A$, then we have
\[
v\left((f+g)^n\right) \to + \infty \Leftrightarrow (v(f^n) \to + \infty \ \mathrm{and}\ v(g^n) \to + \infty)
\]
and
\[
v\left((fg)^n\right) \to + \infty \Leftrightarrow (v(f^n) \to + \infty \ \mathrm{or}\ v(g^n) \to + \infty).
\]
Note that, for the second equivalence, it is necessary to use the fact that our valuations are \emph{non-negative} on $A$.

We derive now from the closed case that if $U$ and $Z$ are open and closed complements in the topology of $P$, then $\,]U[_{P}$ and $\,]Z[_{P}$ are complements in $P^{\mathrm{ad}}$ (this is not part of the definition).
Equivalently, we have to show that $\,]\overline U[_{P} \cap \,]Z[_{P} = \,]\infty_{U}[$.
Since $\overline U \cap Z = \infty_{U}$, this equality follows from the closed case.
As a consequence, we obtain that the proposition also holds in the open case.

As an intermediate step, we now prove that if $Z$ is closed, $U$ open and $X \subset U \cap Z$, then we have $\,]X[_{P} \subset \,]U[_{P} \cap \,]Z[_{P}$.
From the closed case, we have $\,]\overline {X}[_{P} \subset \,]Z[_{P}$ and therefore also $\,]X[_{P} \subset \,]Z[_{P}$.
In order to prove that $\,]X[_{P} \subset \,]U[_{P}$, we introduce a closed complement $F$ for $U$ so that $]F[_P$ and $]U[_P$ are complements in $P^{\mathrm{ad}}$.
Now, we are reduced to show that $\,]\overline X[_{ P} \cap \,]F[_{P} \subset \,]\infty_{X}[_{P}$, and this follows from the closed case again since $\overline X \cap F \subset \infty_{X}$.

Assume that we actually have an equality $X = U \cap Z$ and let us show that we also have an equality $\,]X[_{P} = \,]U[_{P} \cap \,]Z[_{P}$ on the tubes.
We only have to prove the reverse inclusion $\,]U[_{P} \cap \,]Z[_{P} \subset \,]X[_{P}$.
Equivalently, we have to show that $\,]Z[_{P} \subset \,]X[_{P} \cup \,]F[_{P}$ (where $F$ is a closed complement for $U$ as before), and this follows again from the closed case since $Z \subset X \cup F$.

It is now easy to finish the proof of assertions 1) and 2).
For assertion 1), we may write $Y = U \cap Z$ with $U$ open and $Z$ closed.
If we assume that $X \subset Y = U \cap Z$, then we know that $\,]X[_{P} \subset \,]U[_{P} \cap \,]Z[_{P}$ but, from what we just proved, we also have $\,]Y[_{P} = \,]U[_{P} \cap \,]Z[_{P}$.
For assertion 2), we can write $X_{i} = U_{i} \cap Z_{i}$ with $U_{i}$ open and $Z_{i}$ closed, and use the stability by intersection for open or closed that we already know.

Assertion 3) also results from the closed case because closure commutes with union, and therefore, $\overline X = \overline X_{1} \cup \overline X_{2}$.
Let us be more precise. Thanks to assertion 1), only the direct inclusion needs a proof and this is equivalent to $\,]\overline X[_{P} \subset \,]X_{1}[_{P} \cup \,]X_{2}[ _{P}\cup \,]\infty_{X}[_{P}$.
Actually, since, for $i=1, 2$, we have $\infty_{X_{i}} \subset \infty_{X}$, the same is true for the tubes, and we are reduced to check that $\,]\overline X[_{P} \subset \,]\overline X_{1}[_{P} \cup \,]\overline X_{2}[ _{P}\cup \,]\infty_{X}[_{P}$.
Thus, as promised, this again follows from the closed case.
\end{proof}

We may then improve a bit on lemma \ref{loctub}.

\begin{cor}
Let $(X \hookrightarrow P \stackrel \lambda\leftarrow V)$ be an overconvergent space.
If $P := \mathrm{Spf}(A)$ is affine and $X$ is defined modulo an ideal of definition by
\[
\forall i \in \{ 1, \ldots r\}, f_{i}(x) = 0 \quad \mathrm{and} \quad \exists j \in \{1, \ldots, s\}, g_{j}(x) \neq 0,
\]
then $\,]X[_{V}^{\mathrm{naive}}$ is defined in $V$ by
\[
\forall i \in \{ 1, \ldots r\}, v(f_{i}) > 0 \quad \mathrm{and} \quad \exists j \in \{1, \ldots, s\}, v(g_{j}) \leq 0,
\]
and $\,]X[_{V}$ is defined by
\[
\forall i \in \{ 1, \ldots r\}, v(f_{i}^n) \to + \infty \quad \mathrm{and} \quad \exists j \in \{1, \ldots, s\}, v(g_{j}^n) \not \to + \infty.
\]
\end{cor}

\begin{proof}
When $X$ is a closed subset defined by an open ideal $\mathfrak a$, our hypothesis means that $\mathfrak a = I + (f_{1}, \ldots, f_{r})$ for some ideal of definition $I$ of $A$.
Since we always have $v(f^n) \to + \infty$ for $f \in I$, the assertion follows from proposition \ref{loctub}.
In general, $X$ is the intersection of a closed subset and an open subset and we may use assertion 2) of proposition \ref{oper}.
\end{proof}

\begin{prop}
Let $(X \hookrightarrow P \stackrel \lambda\leftarrow V)$ be an overconvergent space.
If $X = Y \cap U$ as a \emph{formal scheme}, with $Y$ (resp. $U$) closed (resp. open) in $P$, then
\[
X_V = Y_V \cap ]U[_V.
\]
\end{prop}

\begin{proof}
After pulling everything back along the monomorphism $P^{/\overline X,\mathrm{ad}} \to P$, we can assume that the inclusion $\overline X \hookrightarrow P$ is a formal thickening. It follows that the map $\overline X \hookrightarrow Y$ is a birational formal thickening and therefore an isomorphism.
\end{proof}

As a consequence, we see that if $P := \mathrm{Spf}(A)$ is affine and $X$ is defined by
\[
\forall i \in \{ 1, \ldots r\}, f_{i}(x) = 0 \quad \mathrm{and} \quad \exists j \in \{1, \ldots, s\}, g_{j}(x) \neq 0,
\]
then $X_V$ is defined by
\[
\forall i \in \{ 1, \ldots r\}, v(f_{i}) =+\infty \quad \mathrm{and} \quad \exists j \in \{1, \ldots, s\}, v(g_{j}^n) \not \to + \infty.
\]

\subsection{Standard properties}

We translate the properties of completion into the language of tubes and derive some consequences.

\begin{prop} \label{proptub}
Let $(X \hookrightarrow P \stackrel \lambda\leftarrow V)$ be an overconvergent space.
Then,
\begin{enumerate}
\item
the tube $\,]X[_{V}$ only depends on the underlying subspace of $X$ (and not on the structure of formal scheme).
\item
if $P = \bigcup_{i \in I} P_{i}$ is an open covering, $X_{i} := X \cap P_{i}$ and $V_{i} := \lambda^{-1}(P_{i}^{\mathrm{ad}})$, then we have an open covering
\[
\,]X[_{P,V} = \bigcup_{i \in I} \,]X_{i}[_{P_{i},V_{i}},
\]
\item \label{pultub}
if $u \colon Q \to P$ is a morphism of formal schemes and $\lambda$ factors through $Q^{\mathrm{ad}}$, then
\[
\,]X[_{P,V} =\,]u^{-1}(X)[_{Q,V}.
\]
\end{enumerate}
\end{prop}

\begin{proof}
Thanks to proposition \ref{oper}, we may assume that $X$ is closed in $P$ in which case everything follows from proposition \ref{formp} and the standard properties of the functor $(-)^\mathrm{ad}$.
\end{proof}

Be careful that, when $Q$ is an \emph{open} formal subscheme of $P$ containing $X$, then $\,]X[_{Q} \neq \,]X[_{P}$ in general: we may \emph{not} replace $P$ with a neighborhood of $X$ (this is a striking difference with the notion of a tube in Tate or Berkovich theory).
On the contrary, if $Y$ is a closed formal subscheme of $P$ that contains $X$ and $Q := P^{/Y}$, then we do have
\[
\,]X[_{P} = \,]X[_{Q}.
\]
By choosing $Y = \overline X$, we see that we may always reduce to the case where $X$ is a dense open subset in $P$ (this is again a striking difference with the classic situation).

As a consequence of assertion \eqref{pultub}), we see that the tube is functorial in the sense that any morphism of overconvergent spaces
\begin{equation}
\xymatrix{Y \ar@{^{(}->}[r] \ar[d]^f & Q\ar[d]^{v} & W \ar[l]_{\mu} \ar[d]^-{u} \\ X \ar@{^{(}->}[r] & P & V \ar[l]_\lambda}
\end{equation}
induces a morphism of locally topologically ringed spaces
\[
]f[_{u} \colon \,]Y[_{W} \to \,]X[_{V}
\]
(this is not built into the definition).
We may simply write $]f[$ when $u$ is understood from the context but we may as well write $u$ when this is $f$ that plays a secondary role.
The functor $(X \hookrightarrow P \leftarrow V) \mapsto \,]X[_{V}$ is clearly continuous and cocontinuous but it is not left exact in a naive sense because the underlying space of a product is not the product of the underlying spaces.
Be careful also that the (composite) functor $(X \hookrightarrow P) \mapsto \,]X[_{P}$ is still continuous, but it is not cocontinuous whatever topology we put on the left hand side.

However, the tube is left exact in the following sense:

\begin{prop} \label{leftex}
Assume that we are given two morphisms of formal embeddings $(X_{i} \hookrightarrow P_{i}) \to (X \hookrightarrow P)$ for $i =1, 2$ and that $ P_{1} \times_{P} P_{2}$ is representable by a locally noetherian formal scheme.
Then, if we are given $\lambda : V \to P_{1}^{\mathrm{ad}} \times_{P^{\mathrm{ad}}} P_{2}^{\mathrm{ad}}$, we have
\[
\,]X_{1} \times_{X} X_{2}[_{V} = \,]X_{1}[_{V} \cap \,]X_{2}[_{V}.
\]
\end{prop}

\begin{proof}
It follows from assertion \eqref{pultub} in proposition \ref{proptub} that
\begin{align*}
\,]X_{1} \times_{X} X_{2}[_{V} &= ](X_{1} \times_{P} P_{2}) \cap (P_{1} \times_{P} X_{2}) [_{V}
\\& = \,]X_{1} \times_{P} P_{2}[_V \cap ]P_{1} \times_{P} X_{2} [_{V}
\\& = \,]X_{1}[_{V} \cap \,]X_{2}[_{V}. \qedhere
\end{align*}
\end{proof}

We may also consider the prepseudo-adic space $(V,\,]X[_V)$.
Then, proposition \ref{leftex} may be rephrased by saying that the functor $(X \hookrightarrow P \leftarrow V) \mapsto (V,\,]X[_V)$ is left exact.
The left exactness of the tube may also be expressed in the following very convenient form:

\begin{cor}
If
\[
\xymatrix{Y \ar@{^{(}->}[r] \ar[d]^f & Q\ar[d]^{v} & W \ar[l]_{\mu} \ar[d]^-{u} \\ X \ar@{^{(}->}[r] & P & V \ar[l]_\lambda}
\]
is a morphism of overconvergent spaces, then
\[
]v^{-1}(X)[_W = u^{-1}\left(\,]X[_V\right).
\]
\end{cor}

\begin{proof}
Immediate consequence of assertion \eqref{pultub} in proposition \ref{proptub} -- which is then simply the case $W=V$.
\end{proof}

We will need at some point to understand the behaviour of the support of a coherent sheaf (recall that we can identify a formal scheme $P$ with the subset of trivial points in $P^{\mathrm{ad}}$):

\begin{prop} \label{supsh}
Let $P$ be a formal scheme.
\begin{enumerate}
\item
If $X$ is a closed formal subscheme of $P$, then $X^{\mathrm{ad}} \cap P = X$ and $X^{\mathrm{ad}} \subset \,]X[_{P}$.
\item
If $\mathcal F$ is a coherent $\mathcal O_{P}$-module and $Z$ (resp.\ $T$) denotes the support of $\mathcal F$ (resp.\ $\mathcal F^{\mathrm{ad}}$), then $T \cap P = Z$ and $T \subset ]Z[_{P}$.
\end{enumerate}
\end{prop}

\begin{proof}
For the second assertion, we may assume that $\mathcal F = \mathcal O_{P}/\mathcal I$ where $\mathcal I$ is a coherent ideal and we are therefore reduced to proving the first one.
Now, the first equality is already known and the second one results from the factorization $X \hookrightarrow P^{/X} \hookrightarrow P$.
\end{proof}

Unfortunately, the following fibration theorem will not be very useful but it is worth mentioning however, as an immediate consequence of the formal fibration theorem:

\begin{prop}[Weak fibration theorem] \label{weakft}
Assume that we are given two formal embeddings $X \hookrightarrow P$ and $X' \hookrightarrow P'$.
If a differentially smooth morphism $u \colon P' \to P$ induces an isomorphism $X' \simeq X$ and an isomorphism $\overline X' \simeq \overline X$, then, it induces locally on $P$, an isomorphism
\[
\,]X'[_{P'} \simeq \mathbb D^{-,n} \times \,]X[_{P}.
\]
\end{prop}

\begin{proof}
We may assume that $X$ and $X'$ are closed in $P$ and $P'$ respectively, in which case our assertion follows directly from the formal fibration theorem \ref{forfib}.
\end{proof}

One would like to relax our hypothesis in proposition \ref{weakft} and only assume that $u$ induces an isomorphism $X' \simeq X$ (and not necessarily $\overline X' \simeq \overline X$).
Unfortunately, the conclusion will not hold anymore as the case of an open embedding $P' \hookrightarrow P$ immediately shows (with $X = X' = P'$).
The tube inside $P$ will be strictly bigger than the tube inside $P'$ in general.

\subsection{Topology of the tubes}

A tube has no reason to be locally closed (even for the constructible topology) in general as the example of the absolute open unit disk inside the absolute closed unit disk shows.
The situation will be much nicer when we stick to analytic spaces (as we soon shall see) or if we consider the naive tube:

\begin{prop}
Assume that $X$ is a locally closed (resp.\ a closed, resp.\ an open) formal subscheme of a formal scheme $P$ and let $\lambda : V \to P^{\mathrm{ad}}$ be any morphism of adic spaces.
Then, $\,]X[_{V}^{\mathrm{naive}} $ is a locally closed (resp.\ a closed, resp.\ an open) locally constructible subset of $V$.
\end{prop}

\begin{proof}
Only constructibility needs a proof because specialization is continuous.
We may assume that $P = \mathrm{Spf}(A)$ and that $X$ is the closed subset defined by an ideal $\mathfrak a$.
If $f_1, \ldots, f_r$ are generators of $ \mathfrak a$, then
\[
\,]X[_{V}^{\mathrm{naive}}  = \left\{v \in V \colon \min_{i=1}^r v(f_i) > 0 \right\}
\]
is a (closed) constructible subset of $V$ (finite intersection of complements of rational open subsets which are quasi-compact).
\end{proof}

We can now describe how far the genuine tube differs from the naive tube.
In particular, we will see that the points of the tube of $X$ cannot specialize too far from $X$ itself (we will denote respectively by $\mathring W$ and $\overline W$ the interior and the closure of a subset $W$):

\begin{prop} \label{antub}
If $(X\to P \leftarrow V)$ is an overconvergent space, then
\[
\widering {]X[}_{V}^{\mathrm{naive}} \subset \,]X[_{V} \subset \overline {]X[}_{V}^{\mathrm{naive}}.
\]
Moreover, if $V$ is \emph{analytic}, then $\,]X[_{V} = \widering {]X[}_{V}^{\mathrm{naive}}$ (resp.\ $\,]X[_{V} = \overline {]X[}_{V}^{\mathrm{naive}}$) when $X$ is closed (resp.\ open) in the topology of $P$.
\end{prop}

\begin{proof}
Recall first that the naive tube sends closed to closed and preserves standard set operations.
On the other hand, we know from proposition \ref{oper} that the tube also preserves these operations (although it does not send closed to closed).
Assume for a while that we know the conclusion in the closed case.
If we write $X = Z \cap U$ with $Z$ closed and $U$ open in the topology of $P$, we will have
\[
\widering {]Z[}_{V}^{\mathrm{naive}} \subset \,]Z[_{V} \subset ]Z[_{V}^{\mathrm{naive}}
\]
and, by considering a closed complement of $U$, we deduce that
\[
]U[_{V}^{\mathrm{naive}} \subset \,]U[_{V} \subset \overline {]U[}_{V}^{\mathrm{naive}}.
\]
Since $\,]Z\cap U[_{V} = \,]Z[_{V} \cap \,]U[_{V}$, this will imply that
\[
\widering {]Z \cap U[}_{V}^{\mathrm{naive}} = \widering {]Z[}_{V}^{\mathrm{naive}} \cap ]U[_{V}^{\mathrm{naive}} \subset \,]Z \cap U[_{V} \subset ]Z[_{V}^{\mathrm{naive}} \cap \overline {]U[}_{V}^{\mathrm{naive}} = \overline {]Z \cap U[}_{V}^{\mathrm{naive}}
\]
and we will be done (because the last assertion of the proposition also clearly follows from the closed case).

Thus, we may assume from now on that $X$ is closed in $P$.
Since the question is local on $V$, we may also assume that $V = \mathrm{Spa}(B, B^+)$ with $(B, B^+)$ complete, that $P = \mathrm{Spf}(A)$ and that $X$ is defined by some ideal $\mathfrak a$.
Then $\,]X[_{V}$ is defined as a subset of $V$ by the conditions $v(f^n) \to + \infty$ and $]X[_{V}^{\mathrm{naive}}$ is defined by the conditions $v(f) > 0$ for $f \in \mathfrak a$.
We proved in lemma \ref{prlem} that both conditions are stable under specialization.
Now, any $v \in V$ has a maximal generalization $w \in V$ of height $\leq 1$ and it follows that
\[
w(f) > 0 \Leftrightarrow w(f^n) \to + \infty \Rightarrow v(f^n) \to + \infty.
\]
Now, since $]X[_{V}^{\mathrm{naive}}$ is a (closed) constructible subset of $V$, one knows that $v \in \widering {]X[}_{V}^{\mathrm{naive}}$ if and only if any generalization of $v$ belongs to $]X[_{V}^{\mathrm{naive}}$. 
This implies the first assertion and we may now assume that $V$ is analytic.
In this case, $w$ has height exactly $1$ and this is the maximal (vertical) generalization of $v$.
Moreover, there exists an isomorphism $\mathcal H(w) \simeq \mathcal H(v)$ of the completed residue fields, and it follows that
\[
w(f) > 0 \Leftrightarrow w(f^n) \to + \infty \Leftrightarrow v(f^n) \to + \infty
\]
(because this last condition is purely topological).
Thus, we see that $w \in ]X[_{V}^{\mathrm{naive}}$ when $v \in \,]X[_{V}$.
Actually, the same holds for any other generalization of $v$ because it is necessarily a specialization of $w$.
\end{proof}

As a consequence of the proposition, we see that the tube is anticontinuous in the case of an analytic overconvergent space.
In other words:

\begin{cor}
The tube of a locally closed (resp.\ closed, resp.\ open) formal subscheme in an \emph{analytic} space is a locally closed (resp.\ open, resp.\ closed) subset. \qed
\end{cor}

In particular, the tube of a \emph{closed} formal subscheme $X$ into an \emph{analytic} overconvergent space $V$ has a natural structure of analytic space (as an open subset of $V$).
In other words, in this case, we can replace the prepseudo-adic space $(V, \,]X[_V)$ with a genuine adic space.

Assume that $X \hookrightarrow P$ is a formal embedding such that $X$ is \emph{open} in the topology of $P$.
Then, there exists a unique \emph{open} formal subscheme $Q$ of $P$ having the same underlying space as $X$ and we have $]Q[_{V}^{\mathrm{naive}} = Q^{\mathrm{ad}}$ which is an open subset of $P^{\mathrm{ad}}$.
It follows that $\,]X[_{P^{\mathrm{an}}}$ is then the same thing as the closure of the open subset $Q^\mathrm{an}$ of $P^{\mathrm{an}}$.
More generally, if we are given $\lambda \colon V \to P^{\mathrm{ad}}$ with $V$ analytic, then $\,]X[_{V}$ will be the closure of the open subset $X_V = \lambda^{-1}(Q^{\mathrm{ad}})$ of $V$.

Recall that we denote by $[V]$ the set of points of height $1$ in an analytic space $V$ and by
\[
\mathrm{sep} : V \to [V], \quad v \mapsto [v]
\]
the canonical retraction sending a point to its maximal generalization (the projection onto the Berkovich quotient).

\begin{prop}
If $(X \hookrightarrow P \leftarrow V)$ is an \emph{analytic} overconvergent space and $v \in V$, then
\[
v \in \,]X[_{V} \Leftrightarrow \mathrm{sp}_{V}([v]) \in X.
\]
\end{prop}

\begin{proof}
We may assume that $X$ is open in the topology of $P$ - and use complement and intersection in order to deduce the general case.
Then, we know from proposition \ref{antub} that $\,]X[_{V} = \overline {]X[}_{V}^{\mathrm{naive}}$.
Since $]X[_{V}^{\mathrm{naive}}$ is a constructible subset, we will have $v \in \overline {]X[}_{V}^{\mathrm{naive}}$ if and only if $v$ has a generalization $v'$ in $]X[_{V}^{\mathrm{naive}}$.
Since $]X[_{V}^{\mathrm{naive}}$ is open, the maximum generalization $[v]$ of $v'$ will also be in $]X[_{V}^{\mathrm{naive}}$ and we may therefore assume that $v' = [v]$ is the maximal generalization of $v$.
In other words, we have $v \in \overline {]X[}_{V}^{\mathrm{naive}}= \,]X[_{V}$ if and only if $[v] \in ]X[_{V}^{\mathrm{naive}}$ as asserted.
\end{proof}

As a consequence of this proposition, when $V$ is analytic, we can recover the genuine tube from the naive one through specialization and separation.
More precisely, if we set $]X[_{[V]} := ]X[_{V}^{\mathrm{naive}} \cap [V]$, then we have
\[
\,]X[_{V} = \mathrm{sep}^{-1}(]X[_{[V]})
\]
(our tube is the inverse image of the Berkovich tube).
In particular, we see that, when $V$ is analytic, the tube $\,]X[_{V}$ is stable both under specialization and generalization\footnote{This is sometimes called an \emph{overconvergent} subspace but we want to avoid this terminology here.}.

\subsection{Tubes of finite radii}

We will consider now the notion of a tube of finite radii.
We do not want to develop a full theory but only describe the situation locally.

\begin{prop} \label{loctub1}
Let $(X \hookrightarrow P \stackrel \lambda \leftarrow V)$ be a \emph{convergent} space.
Assume that $P$ is affine and $X$ is defined by $f_{1}= \cdots = f_{r} = 0$ in $P$.
Let $V \to \mathrm{Spa}(B, B^+)$ be an \emph{adic} morphism and $(g_{1}, \ldots, g_{s})$ an ideal of definition for $B$.
We set
\[
[X]_{V,n} = \left\{v \in V \colon \min_{i=1}^r v(f^n_{i}) \geq \min_{j=1}^s v(g_{j}) \neq + \infty \right\}.
\]
Then, for all $n \in \mathbb N$, $[X]_{V,n}$ is a retrocompact open subset of $V$ such that
\[
\overline{[X]}_{V,n} \subset [X]_{V,n+1}.
\]
Moreover, if $V$ is \emph{analytic}, then
\[
\,]X[_{V} = \bigcup_{n \in \mathbb N} [X]_{V,n}.
\]
\end{prop}

\begin{proof}
We may clearly assume that $V = \mathrm{Spa}(B, B^+)$.
Then, $[X]_{V,n}$ is the finite union of the rational subsets
\[
\left\{v \in V \colon \min_{i,j=1}^{r,s} \{v(f^n_{i}), v(g_{j})\} \geq v(g_{k}) \neq + \infty \right\}
\]
for $k = 1, \ldots, s$.
This is therefore a retrocompact open subset of $V$ as asserted.
In particular, $[X]_{V,n}$ is constructible, and in order to prove the second assertion, it is sufficient to check that if $w \in [X]_{V,n}$ specializes to $v \in V$, then $v \in [X]_{V,n+1}$.
At this point, it seems more natural to work in the residue fields and use the multiplicative notation (even if the absolute value might have higher height).
Let us assume that $|f^n(w)| \leq |g(w)| \neq 0$ with $f \in \mathfrak a$ and $g \in J$.
Then necessarily $f(w)$ is topologically nilpotent (because $g(w)$ is) and $\left|\frac {f^n(w)}{g(w)}\right| \leq 1$.
It follows that $\frac {f^{n+1}(w)}{g(w)}$ is also topologically nilpotent.
Since $\mathcal H(v)$ is homeomorphic to $\mathcal H(w)$, we also have $g(v) \neq 0$ and $\frac {f^{n+1}(v)}{g(v)}$ topologically nilpotent.
In particular, we must have $\left|\frac {f^{n+1}(v)}{g(v)}\right| \leq 1$ and therefore $|f^{n+1}(v)| \leq |g(v)| \neq 0$.

For the last assertion, it is sufficient to check that, for $f \in \mathfrak a$ and $v \in V$, we have
\[
\left(v(f^n) \to + \infty \right) \quad \Leftrightarrow \quad \left( \exists n \in \mathbb N, \exists g \in J, v(f^n)\geq v(g) \neq + \infty \right).
\]
First of all, we know that
\[
v \in V \quad \Rightarrow \quad \exists g \in J, v(g) \neq + \infty ,
\]
because $J$ is an ideal of definition.
Moreover, once we know that $v(g) \neq + \infty$, then the equivalence
\[
v(f^n) \to + \infty \quad \Leftrightarrow \quad \exists n \in \mathbb N, v(f^n)\geq v(g)
\]
follows from the fact that $v(g^n) \to + \infty$ (because $g \in J$ and our valuations are continuous).
\end{proof}

The open subset $[X]_{V,n}$ of proposition \ref{loctub1} is the \emph{(closed) tube of radius $n \in \mathbb N$} of $X$ in $V$.
One can show that it is essentially independent of the choices.
It is then possible to define the tubes fo finite radii in the more general case of an overconvergent space.
Moreover, there exists a notion of \emph{open} tube of finite radii.
 We shouldn't need all of this.

\begin{xmp}
\begin{enumerate}
\item
Assume $V$ is Tate affinoid with topologically nilpotent unit $\pi$.
Then
\[
[X]_{V,n} = \left\{v \in V \colon \min_{i=1}^{r} \{v(f^n_{i})\} \geq v(\pi)\right\}
\]
is a rational subset (and in particular an affinoid open subset) of $V$.
\item
Assume $O$ is Tate with topologically nilpotent unit $\pi$.
We can consider the analytic overconvergent space (zero section)
\[
\xymatrix{\mathrm{Spec}(\mathbb Z) \ar@{^{(}->}[r] & \mathbb A & \mathbb D_{O} \ar[l]}.
\]
Then, we have
\[
\,]\mathrm{Spec}(\mathbb Z)[_{\mathbb D_{O}} = \mathbb D^-_{O} 
\]
(even as adic spaces, as we shall see below), and for each $n \in \mathbb N$,
\[
[\mathrm{Spec}(\mathbb Z)]_{\mathbb D_{O},n} = \mathbb D_{O}(0, \pi^{\frac 1n}) = \left\{v \in \mathbb A_{O} \colon v(T^n) \geq v(\pi) \right\}
\]
is a closed disk.
The formula of the proposition reads
\[
\mathbb D^-_{O} = \bigcup_{n \in \mathbb N} \mathbb D_{O}(0, \pi^{\frac 1n}).
\]
Recall that the open unit disk does not contain the $v^-$-points (with $v$ a Gauss valuation).
\end{enumerate}
\end{xmp}

We can also introduce the \emph{standard neighborhoods} of Berthelot:

\begin{prop}
Let $(X \hookrightarrow P \leftarrow V)$ be a \emph{convergent} space.
Assume that $V$ is Tate quasi-compact with topologically nilpotent unit $\pi$ and $P$ is affine.
Let $U$ be an open subset of $X$ with closed complement $Z := \mathrm V(h_{1}, \ldots, h_{r})$.
If $m \in \mathbb N$, we set
\[
V_m = \left\{v \in \,]X[_V \colon v(\pi) \geq \min_{i=1}^r v(h^m_{i}) \neq + \infty \right\}.
\]
Next, if $\varphi : \mathbb N \to \mathbb N$ is any map, we set
\[
V_\varphi := \bigcup_{n \in \mathbb N} \left( V_{\varphi(n)} \cap [X]_{V,n}\right).
\]
Then, $\{V_\varphi\}_{\varphi \in \mathbb N^{\mathbb N}}$ is a cofinal system of open neighborhoods of $]U[_V$ in $\,]X[_V$.
\end{prop}

\begin{proof}
Let us first check that for all $m \in \mathbb N$, we have
\[
]U[_V \subset V_m \quad \mathrm{and} \quad [Z]_{V,m} \cap V_{m+1} = \emptyset.
\]
We may assume that $V = \mathrm{Spa}(B, B^+)$.
Assume that $v \in \,]X[_V$ but $v \notin V_m$.
Then, for $i = 1, \ldots, r$, either $v(h_i^m) = + \infty$ or $v(\pi) < v(h_i^m)$.
In any case, $v(h_i^m) \to +\infty$ and $v \in ]Z[_V$.
It follows that $]U[_V \subset V_m$ as asserted.
Let us show now that $[Z]_{V,m} \cap V_{m+1} = \emptyset$.
Assume that there exists some $v \in \,]X[_V$ such that
\[
\min_{i=1}^r v(h^{m}_{i}) \geq v(\pi) \geq \min_{i=1}^r v(h^{m+1}_{i}) \neq + \infty.
\]
Then, there exists $i \in \{1, \ldots, r\}$ such that $v(h_i) \leq 0$.
And therefore, $v(\pi) \leq 0$.
Contradiction.

Given any $\varphi : \mathbb N \to \mathbb N$, since we assume $V$ analytic, proposition \ref{loctub1} implies that
\[
]U[_V = ]U[_V \cap \bigcup_{n \in \mathbb N} [X]_{V,n} = \bigcup_{n \in \mathbb N} \left(]U[_V \cap [X]_{V,n} \right) \subset \bigcup_{n \in \mathbb N} \left( V_{\varphi(n)} \cap [X]_{V,n}\right) = V_{\varphi}.
\]
Conversely, let $\mathcal U$ be an open neighborhood of $]U[$ in $\,]X[$.
If we fix $n \in \mathbb N$, then
\[
[X]_{V,n} \subset \,]X[_V = \mathcal U \cup ]Z[_V = \mathcal U \cup \bigcup_{m \in \mathbb N} [Z]_{V,m}
\]
since $V$ is analytic again.
But we also know that $[X]_{V,n}$ is retrocompact in $V$ which we assumed to be quasi-compact and $[X]_{V,n}$ is therefore also quasi-compact.
It follows that there exists $m \in \mathbb N$ such that $[X]_{V,n} \subset \mathcal U \cup [Z]_{V,m}$.
Since $[Z]_{V,m} \cap V_{m+1} = \emptyset$, if we let $\varphi(n) = m+1$, we see that $V_{\varphi(n)} \cap [X]_{V,n} \subset \mathcal U$ and taking union, we get $V_{\varphi} \subset \mathcal U$.
\end{proof}

It will sometimes be convenient to rely the following weaker version:

\begin{cor} \label{locneigh}
Let $(X \hookrightarrow P \leftarrow V)$ be an overconvergent space.
Assume that $V$ is Tate quasi-compact with topologically nilpotent unit $\pi$, $P$ is affine and $X$ is the open complement in $P$ of the closed subspace $g_1 = \cdots = g_r = 0$.
Then, the subspaces
\[
V_{m} := \{v \in V \colon \quad v(\pi) \geq \max_{i=1}^rv(g_i^m) \}
\]
form a cofinal system of open neighborhoods of $\,]X[_{V}$ in $V$ and $\,]X[_{V} = \bigcap_{m \in \mathbb N} V_{m}$. \qed
\end{cor}

\begin{xmp}
\begin{enumerate}
\item Let $(X \hookrightarrow P \leftarrow V)$ be an overconvergent space where $V$ is Tate affinoid with topologically nilpotent unit $\pi$ and $X$ is the (support of the) complement of a hypersurface $h = 0$ in $P$ affine.
Then the rational open subsets
\[
V_m := \{v \in V, v(\pi) \geq v(h^m)\}
\]
form a cofinal system of rational open neighborhoods of $\,]X[_V$ in $V$.
\item
Assume that $O$ is Tate with topologically nilpotent unit $\pi$.
We consider the analytic overconvergent space
\[
\xymatrix{\mathbb A \ar@{^{(}->}[r] & \mathbb P & \mathbb A_{O} \ar[l]}.
\]
Then,
\[
\,]\mathbb A[_{\mathbb A_{O}} = \bigcap_{n \in \mathbb N} \mathbb D_{O}(0, \pi^{-\frac 1n})
\]
where
\[
\mathbb D_{O}(0, \pi^{-\frac 1n}) := \left\{v \in \mathbb A_{O} \colon v(\pi T^n) \geq 0 \right\}
\]
is an affinoid open subset for each $n \in \mathbb N$.
The tube is slightly bigger than the closed unit disk $\mathbb D_{O}$: it contains the $v^\pm$-points (this is the proper unit disk $\overline {\mathbb D}_{O}$).
\item (Monsky-Washnitzer setting)
We let $R$ be a noetherian ring and $S \to \mathrm{Spec}(R)$ be any morphism of formal schemes.
Then, we consider an $R$-algebra of finite type $A$ and let $X := \mathrm{Spec}(A)$.
Also, we give ourselves a Tate ring $\Lambda$ with topologically nilpotent unit $\pi$, a morphism $O = \mathrm{Spa}(\Lambda, \Lambda^+) \to S^{\mathrm{ad}}$ and a formal \emph{thickening} $C \hookrightarrow S$.
Then (using $\dagger$ to denote the weak completion of a ring), we have
\[
\Gamma(\,]X_{C}[_{X_{O}}, i^{-1}\mathcal O_{X_{O}}) = A_\Lambda^\dagger := (\Lambda_{0} \otimes_{R} A)^\dagger[1/\pi]
\]
if $i : \,]X_C[_{X_O} \hookrightarrow X_O$ denotes the inclusion map.
\end{enumerate}
\end{xmp}

Using corollary \ref{locneigh}, one can show the following:

\begin{prop}
Let
\[
(X \hookrightarrow P \leftarrow V) \to (C \hookrightarrow S \leftarrow O) \quad \mathrm{and} \quad (Y \hookrightarrow Q \leftarrow W) \to (C \hookrightarrow S \leftarrow O)
\]
be two formal morphisms of analytic overconvergent spaces.
When $V'$ and $W'$ run through all the neighborhoods of $\,]X[_{V}$ in $V$ and $]Y[_{W}$ in $W$ respectively, then the open subsets $V' \times_{O} W'$ of $V \times_O W$ form a cofinal system of neighborhoods of $\,]X\times_C Y[_{V \times_O W}$.
\end{prop}

\begin{proof}
Since the question is local everywhere, we can assume that $V$, $W$ and $C$ are affinoid, and $P$, $Q$ and $S$ are affine.
We can also complete along the closure of $X$ and $Y$ and blow up the locus at infinity (see proposition \ref{bup2} below for example).
We may therefore assume that $X$ (resp.\ $Y$) is the complement in $P$ (resp.\ $Q$) of a hypersurface $g=0$
(resp.\ $h=0$).
Then, $X \times_C Y$ is the complement in $P \times_S Q$ of the hypersurface $g \otimes h = 0$.
Our assertion follows from proposition \ref{locneigh} because, if we set
\[
V_{n} := \{v \in V \colon \quad v(\pi) \geq v(g^n) \} \quad \mathrm{and} \quad W_{n} := \{v \in W \colon \quad v(\pi) \geq v(h^n) \},
\]
then we have
\begin{align*}
V_{2n} \times_V W_{2n} &= \{v \in V \times_O V \colon \quad v(\pi) \geq v(g^{2n} \otimes 1) \ \mathrm{et}\ v(\pi) \geq v(1 \otimes h^{2n} )\}
\\ &\subset \{v \in V \times_O V \colon \quad v(\pi) \geq v((g \otimes h)^n)\}.\qedhere
\end{align*}

\end{proof}

We finish with the following adic variant of the famous proposition 0.2.7 of \cite{Berthelot96c*}:

\begin{prop} \label{carthat}
If $(X \hookrightarrow P \stackrel\lambda\leftarrow V)$ is an \emph{analytic convergent} space, then there exists an isomorphism of \emph{adic} spaces:
\[
P^{/X,\mathrm{ad}} \times_{P^{\mathrm{ad}}} V \simeq \,]X[_{V}.
\]
\end{prop}

Note that we already know that there exists such a homeomorphism.

\begin{proof}
The question is local.
We may therefore assume that $P = \mathrm{Spf}(A)$ is affine with $X$ defined by an ideal $\mathfrak a = (f_{1}, \ldots, f_{r})$ and that $V = \mathrm{Spa}(B, B^+)$ with $B$ a complete Tate ring with topologically nilpotent unit $\pi$.
Recall that, in this situation, $P^{/X} = \mathrm{Spa}(A^{/\mathfrak a})$ where $A^{/\mathfrak a}$ is identical to $A$ as a ring but its topology is defined by $\mathfrak a + I$ if $I$ denotes an ideal of definition for $A$.
Proposition \ref{loctub} tells us that $\,]X[_{V} = \bigcup_{n \in \mathbb N} [X]_{V,n}$ with
\[
[X]_{V,n} = \left\{x \in V \colon \min_{i=1}^{r} \{v(f^n_{i})\} \geq v(\pi) \neq + \infty \right\},
\]
and we need to show that the natural morphism $[X]_{V,n} \to P^{\mathrm{ad}}$ factors canonically through $ P^{/X,\mathrm{ad}}$.
After replacing each $f_i$ by its $n$th power, we may assume that $n=1$.
If we write $[X]_{V,1} = \mathrm{Spf}(C, C^+)$ with $(C, C^+)$ complete, we need to show that the composite map $A \to B \to C$ is continuous for the $\mathfrak a$-adic topology.
Thus, we have to check that $f_{i}^n \to 0$ in $C$ for all $i = 1, \ldots, r$ (we may re-use the letter $n$).
This follows from the fact that, by definition, ${f_{i}} / \pi$ is power-bounded in $C$ and $\pi$ is topologically nilpotent.
\end{proof}

As a consequence of this proposition, we see that, in an \emph{analytic} overconvergent space $(X \hookrightarrow P \leftarrow V)$, we can replace $P$ with its completion along $\overline X$ and then replace $X$ with the open formal subschemes $Q$ of $P$ having the same underlying space as $X$, without modifying the tube.
In practice, we are therefore often reduced to the case when $X = Q$ is an open formal subscheme of $P$.

Note that it is necessary to assume $V$ analytic in proposition \ref{carthat} in order to consider $\,]X[_{V}$ as an adic space (as an open subset of $V$).
Anyway, as already mentioned, the map $P^{/X,\mathrm{ad}} \times_{P^{\mathrm{ad}}} V \to \,]X[_{V}$ is not even an isomorphism of locally topologically ringed spaces in general.

As a consequence of the proposition, one can show that he fiber of an analytic convergent space only depends on the tube:

\begin{cor} \label{fibtub}
If $(X \hookrightarrow P \stackrel\lambda\leftarrow V)$ is an analytic convergent space, then $X_V \simeq X_{]X[_V}$.
\end{cor}

\begin{proof}
Since we have a factorization into a closed embedding $X \hookrightarrow P^{/X}$ followed by a monomorphism $P^{/X} \to P$, we can replace $V$ with $P^{/X,\mathrm{ad}} \times_{P^{\mathrm{ad}}} V = \,]X[_{V}$.
\end{proof}

In particular, we see that, in the analytic convergent case, the fiber $X_V$ is a closed adic subspace of the tube $]X[_V$.
This extends to the analytic overconvergent case if we consider prepseudo-adic spaces.

In order to generalize the constructions of this section to the case where $P$ is \emph{not} locally noetherian, it is necessary to assume that $X$ is \emph{locally radically finitely presented}:
it means that locally, $X$ is defined modulo an ideal of definition by an ideal $\mathfrak a$ such that $\sqrt \mathfrak a = \sqrt{(S)}$ for some \emph{finite} set $S \in A$.
One may then proceed in the same way as above, using generalized adic spaces of Scholze and Weinstein instead of the honest adic spaces of Huber.
We will not work out the details here.

\begin{xmp}
If $\mathcal V$ is a (non discrete) valuation ring with residue field $k$ and $X \hookrightarrow P$ is a locally closed embedding of a $k$-variety into a formal $\mathcal V$-scheme which is locally finitely presented over $\mathcal V$, then $X$ is locally radically finitely presented.
This follows from the fact that, if $\mathfrak m$ denotes the maximal ideal of $\mathcal V$, then we have $\sqrt \mathfrak m = \mathfrak m = \sqrt{(\pi)}$ for any topologically nilpotent unit $\pi$.
\end{xmp}

\section{The overconvergent site}

In this section, we relax our condition on (formal) morphisms of overconvergent spaces in order to make the role of the ambient formal scheme secondary and to allow the replacement of the adic space by some neighborhood of the tube.
As usual, formal schemes are always assumed to be locally noetherian and adic spaces are always assumed to be locally of noetherian type.

\subsection{Strict neighborhoods}

We introduce the notion of a strict neighborhood and list some elementary properties.

\begin{dfn} \label{strneig}
A \emph{strict neighborhood} is a formal morphism
\[
\xymatrix{X' \ar@{^{(}->}[r] \ar[d]^f & P'\ar[d]^{v} & V' \ar[l] \ar[d]^-{u} \\ X \ar@{^{(}->}[r] & P & V. \ar[l]}
\]
where
\begin{enumerate}
\item $f$ is an isomorphism,
\item $v$ is locally noetherian,
\item $u$ is an open embedding,
\item
$\,]f[_{u}$ is surjective.
\end{enumerate}
\end{dfn}

Note that a strict neighborhood induces an isomorphism
\[
\,]f[_{u} \colon \,]X'[_{V'} \simeq \,]X[_{V}
\]
of topologically locally ringed spaces (for the induced structures).
In particular, the last two conditions of definition \ref{strneig} may be rephrased by saying that the morphism of prepseudo-adic spaces $(V', \,]X[_{V'}) \to (V, \,]X[_V)$ is a strict neighborhood (of prepseudo-adic spaces).

It is also very likely that the strict neighborhood of overconvergent spaces induces a strict neighborhood $(V',X'_{V'}) \simeq (V,X_V)$ between the fibers but I have not checked it.

Up to an isomorphism, we may always assume that $X' = X$, that $f = \mathrm{Id}_X$ and that $u$ is the inclusion of an open subset so that $\,]X[_{V'} = \,]X[_{V}$, and we would then write
\[
\xymatrix{ X \ar@{^{(}->}[r] \ar@{=}[d]& P' \ar[d]^{v}& V' \ar[l] \ar@{^{(}->}[d] \\
X \ar@{^{(}->}[r] & P & V. \ar[l]}
\]

\begin{xmp}
Assume that we are given a formal scheme $S$ and a morphism $S^{\mathrm{ad}} \leftarrow O$ with $O$ \emph{analytic}.
Then there exists a sequence of strict neighborhoods
\[
(S \hookrightarrow \mathbb A^-_{S} \leftarrow \mathbb D^{-}_{O}) \to (S \hookrightarrow \mathbb A_{S} \leftarrow \mathbb D_{O}) \to (S \hookrightarrow \mathbb P_{S} \leftarrow \mathbb P_{O})
\]
in which $S$ is embedded using the zero-section
(this is not true however when $O$ is not analytic).
\end{xmp}

\begin{prop}
Any strict neighborhood
\[
\xymatrix{X' \ar@{^{(}->}[r] \ar[d]^f & P'\ar[d]^{v} & V' \ar[l] \ar[d]^-{u} \\ X \ar@{^{(}->}[r] & P & V. \ar[l]}
\]
is the composition of a strict neighborhood with $u = \mathrm{Id}_{V}$ and another strict neighborhood with both $f = \mathrm{Id}_{X}$ and $v = \mathrm{Id}_{P}$.
\end{prop}

\begin{proof}
Follows from proposition \ref{decom}.
\end{proof}

We may always pull a strict neighborhood back:

\begin{prop}
If
\[
\xymatrix{X' \ar@{^{(}->}[r] \ar[d] & P'\ar[d] & V' \ar[l] \ar[d] \\ X \ar@{^{(}->}[r] & P & V \ar[l]}
\]
is a strict neighborhood and we pull-back along a morphism of formal schemes $Q \to P$, then we obtain a strict neighborhood
\[
\xymatrix{Y' \ar@{^{(}->}[r] \ar[d] & Q'\ar[d] & W' \ar[l] \ar[d] \\ Y \ar@{^{(}->}[r] & Q & W. \ar[l]}
\]
\end{prop}

\begin{proof}
Thanks to assertion \eqref{pultub} of proposition \ref{proptub}, the morphism of overconvergent spaces
\[
\xymatrix{X' \times_P Q \ar@{^{(}->}[r] \ar[d] & P' \times_P Q \ar[d] & V' \times_{P^{\mathrm{ad}}} Q^{\mathrm{ad}} \ar[l] \ar[d] \\ X\times_P Q \ar@{^{(}->}[r] & Q & V \times_{P^{\mathrm{ad}}} Q^{\mathrm{ad}} . \ar[l]}
\]
clearly satisfies the conditions of definition \ref{strneig}.
\end{proof}

Strict neighborhoods will allow us to enlarge (or shrink) $P$ as the next proposition shows.

\begin{prop} \label{opst1}
If $(X \hookrightarrow P \leftarrow V)$ is an overconvergent space and $P \hookrightarrow Q$ is a formal embedding, then we have a strict neighborhood:
\[
\xymatrix{ X \ar@{^{(}->}[r] \ar@{=}[d] & P \ar@{^{(}->}[d] & V \ar[l] \ar@{=}[d]\\
X \ar@{^{(}->}[r] & Q & V. \ar[l]}
\]
\end{prop}

\begin{proof}
Since the left hand square is cartesian, we have $\,]X[_{Q} \cap P^{\mathrm{ad}} = \,]X[_{P}$.
Everything follows because a formal embedding is a locally noetherian morphism.
\end{proof}

Be careful however that a morphism of the form
\[
\xymatrix{ X \ar@{^{(}->}[r] \ar@{=}[d] & P \ar@{^{(}->}[d] & P^{\mathrm{ad}} \ar@{=}[l] \ar@{^{(}->}[d] \\
X \ar@{^{(}->}[r] & Q & Q^{\mathrm{ad}}, \ar@{=}[l]}
\]
in which the middle map is an open embedding,
is \emph{not} a strict neighborhood in general.
In the case $X = P = \mathbb A$ and $Q = \mathbb P$, we have $\,]X[_{P} = \mathbb D \neq ]\mathbb A[_{\mathbb P} = \,]X[_{Q}$.
In the statement of proposition \ref{opst1}, it is therefore important to keep the same $V$ on the right.

Using the next property, it will always be possible to assume that the locus at infinity is a Cartier divisor:

\begin{prop} \label{bup2}
Let $(X \hookrightarrow P \leftarrow V)$ be an overconvergent space.
If $v \colon P' \to P$ is a blowing up of a usual subscheme centered outside $X$, then it extends to a strict neighborhood
\[
\xymatrix{ X \ar@{^{(}->}[r] \ar@{=}[d] & P' \ar[d]^{v}& V' \ar[l]_-{\lambda'} \ar[d]^{u}\\
X \ar@{^{(}->}[r] & P & V. \ar[l]_\lambda}
\]
\end{prop}

\begin{proof}
We may assume that $V = P^{\mathrm{ad}}$.
Now, the assertion results from proposition \ref{bup}.
More precisely, if we denote by $Z$ the center of the blowing up and let $Z' = v^{-1}(Z)$ so that $X \cap Z = \emptyset$ and $X \cap Z' = \emptyset$, then $v$ induces an isomorphism $ P'^{\mathrm{ad}} \setminus Z'^{\mathrm{ad}}\simeq P^{\mathrm{ad}} \setminus Z^{\mathrm{ad}}$ between an open neighborhood of $\,]X[_{P'}$ and an open neighborhood of $\,]X[_{P}$.
We may then choose $V' := P'^{\mathrm{ad}} \setminus Z'^{\mathrm{ad}}$.
\end{proof}

As a consequence, we obtain Chow's lemma for overconvergent spaces:

\begin{cor} \label{cholem}
Assume that we are given a formal morphism
\[
\xymatrix{ Y \ar@{^{(}->}[r] \ar[d]^f & Q \ar[d]^v & W \ar[l] \ar[d]^u\\
X \ar@{^{(}->}[r] & P & V \ar[l]}
\]
in which $v : Q \to P$ is separated of finite type \emph{around} $Y$ (see definition \ref{formcl}).
If $Y$ is reduced, $f$ is quasi-projective and $\overline X$ quasi-compact, then there exists a strict neighborhood
\[
\xymatrix{Y' \ar@{^{(}->}[r] \ar[d] & Q'\ar[d] & W' \ar[l] \ar[d] \\ Y \ar@{^{(}->}[r] & Q & W. \ar[l]}\]
such that the composite map $v' : Q' \to P$ is quasi-projective in the neighborhood of $Y'$.
\end{cor}

\begin{proof}
We may assume that $X$ also is reduced.
We are in the situation to apply the precise Chow's lemma (corollary I.5.7.14 of \cite{GrusonRaynaud71}): there exists a blowing-up $Z \to \overline Y$ centered outside $Y$ such that the composite map $Z \to \overline X$ is projective.
This blowing up is always induced by some blowing up $Q' \to Q$ and we may apply proposition \ref{bup} (we have $Y' = Y$ and $\overline {Y'} = Z$).
\end{proof}

In the analytic case, we also have the following:

\begin{prop}
If $(X \hookrightarrow P \leftarrow V)$ is an \emph{analytic} overconvergent space, then there exists a strict neighborhood 
\[
\xymatrix{ X \ar@{^{(}->}[r] \ar@{=}[d] & P' \ar[d]^{v}& V' \ar[l]_-{\lambda'} \ar[d]^{u}\\
X \ar@{^{(}->}[r] & P & V. \ar[l]_\lambda}
\]
with $X$ open and dense in the topology of $P'$ and $\,]X[_{V'}$ closed in $V'$.
\end{prop}

\begin{proof}
Simply choose $P' = P^{/\overline X}$ and $V' = \,]\overline X[_{V}$.
The result is then an immediate consequence of proposition \ref{carthat}.
\end{proof}

\subsection{The overconvergent site}

We will now invert strict neighborhoods in order to obtain the overconvergent site.
The next two propositions will allow us to prove that they form a right multiplicative system.

\begin{lem} \label{pulcri}
Strict neighborhoods are stable under pullback.
\end{lem}

\begin{proof}
Immediate consequence of proposition \ref{leftex}.
\end{proof}

\begin{lem} \label{opst}
If a strict neighborhood 
\[
\xymatrix{X' \ar@{^{(}->}[r] \ar[d]^f & P'\ar[d]^{v} & V' \ar[l] \ar[d]^-{u} \\ X \ar@{^{(}->}[r] & P & V. \ar[l]}
\]
has a section, then this section is also a strict neighborhood.
\end{lem}

\begin{proof}
Condition 1) and 4) of definition \ref{strneig} are clearly satisfied because both $f$ and $]f[_{u}$ are actually isomorphisms and a section of an isomorphism is necessarily an isomorphism.
The same holds for condition 3) because an open embedding with a section is also an isomorphism.
Finally, condition 2) follows from the fact that a section of a locally noetherian map is locally noetherian.
\end{proof}

\begin{prop}
Overconvergent spaces and formal morphisms admit right calculus of fractions with respect to strict neighborhoods.
\end{prop}

\begin{proof}
This is a formal consequence of lemmas \ref{pulcri} and \ref{opst} as lemma \ref{right} below shows.
\end{proof}

\begin{lem} \label{right}
Let $\mathcal C$ be a category and $S$ a set of morphisms which contains identities, is stable by composition, is stable by pull-back and is stable by taking sections.
Then, $\mathcal C$ admits right calculus of fractions with respect to $S$.
\end{lem}

\begin{proof}
According to definition 2.3 of \cite{GabrielZisman67}, there are four (sometimes called after Ore) conditions to check, the first two of which coincide with our first two.
For the third condition, we have to show that any diagram
\[
\xymatrix{Y' \ar@{-->}[r]^{f'} \ar@{-->}[d]^{t} & X' \ar[d]^s \\ Y \ar[r]^f & X
}
\]
with $s \in S$ may be completed into a commutative diagram with $t \in S$.
We may simply choose for $t$ the pullback of $s$ along $f$.
For the fourth condition, we must show that any commutative diagram
\[
\xymatrix{
Y' \ar@{-->}[r]^{t} & Y \ar@<2pt>[r]^f \ar@<-2pt>[r]_{g}& X \ar[r]^s & X'
}
\]
with $s \in S$ may be completed into a commutative diagram with $t \in S$.
First of all, if we pull $s$ back along itself, we see that $X \times_{X'} X$ is representable and that the projection $p : X \times_{X'} X \to X$ belongs to $S$.
It follows that the diagonal map $\delta : X \hookrightarrow X \times_{X'} X$ also belongs to $S$ because this is a section of $p$.
We may then pull $\delta$ back along the map $(f, g) : Y \to X \times_{X'} X$ in order to obtain $t : Y' \to Y$ (which is actually the kernel of $f$ and $g$).
\end{proof}

\begin{dfn}
The \emph{(absolute) overconvergent site} is the category of overconvergent spaces localized with respect to strict neighborhoods and endowed with the image topology.
\end{dfn}

In other words, an object of the overconvergent site is an overconvergent space and a morphism is, up to equivalence, a diagram
\[
\xymatrix{
(Y' \hookrightarrow Q' \leftarrow W') \ar[d] \ar[dr] \\ (Y \hookrightarrow Q \leftarrow W) & (X \hookrightarrow P \leftarrow V)
}
\]
where the vertical map is a strict neighborhood.
It means that we may always replace $W$ with some neighborhood of the tube and modify $Q$ almost as we wish.

Our category is endowed with the \emph{image topology} (the coarsest topology making continuous the localization map).
This is the topology generated by the pretopology made of families of formal morphisms
\[
\{(X \hookrightarrow P \leftarrow V_{i}) \to (X \hookrightarrow P \leftarrow V)\}_{i \in I}
\]
where, for each $i \in I$, $V_{i}$ is open in $V$, and $\,]X[_{V} = \bigcup_{i \in I}\,]X[_{V_{i}}$.
Again, it is subcanonical.
Giving a sheaf (or a presheaf) on the overconvergent site is equivalent to giving a sheaf (or a presheaf) with respect to \emph{formal} morphisms which induces an isomorphism on strict neighborhoods.

If necessary, we will call this topology the \emph{adic} topology.
We may also endow our category with the \emph{Zariski} topology (image of the Zariski topology) or even with the \emph{Zariski-adic} topology which is the coarsest topology finer than the two others, for example.
Many other choices are possible.

Since the formal scheme $P$ plays a very loose role in the theory, we will usually denote by $(X,V)$ an object of the overconvergent site.
Also, since we will often have to replace $V$ with some neighborhood $V'$ of $\,]X[_{V}$ in $V$, it will be convenient to simply call such a $V'$ \emph{a neighborhood of $X$ in $V$} instead of a neighborhood \emph{of the tube} of $X$ in $V$.

The localization functor
\[
(X \hookrightarrow P \leftarrow V) \mapsto (X,V)
\]
from the category of overconvergent adic spaces and formal morphism to the overconvergent adic site commutes with finite limits (\cite{GabrielZisman67} again) and is continuous (by definition).
Let us remark that, as a consequence, the functor $(X \hookrightarrow P) \mapsto (X, P^{\mathrm{ad}})$ from the category of formal embeddings to the category of overconvergent sites also commutes with finite limits.
It is continuous for the Zariski or Zariski-adic topology but not for the adic-topology (see corollary \ref{loczar} however).
Also, the functor $(X, V) \mapsto X$ commutes with all limits because it has an adjoint $X \mapsto (X, \emptyset)$:
\[
\mathrm{Hom}((Y,\emptyset), (X,V)) \simeq \mathrm{Hom}(Y, X).
\]
And it is also continuous.
It also follows from (the comments after) proposition \ref{leftex} that the functor $(X \hookrightarrow P \leftarrow V) \mapsto (V, \,]X[_V)$ induces a functor $(X,V) \mapsto \,]X[^\dagger_V$ from the overconvergent site to the category of \emph{germs} of adic spaces, which is continuous, cocontinuous and left exact.
Finally, there exists a fiber functor $(X,V) \mapsto X_V$ and a natural map $X_V \hookrightarrow ]X[_V^\dagger$.

\subsection{Local isomorphisms of overconvergent spaces}

The topology of the adic site is very coarse in the sense that it is not local with respect to any topology on the formal schemes side.
We will show however that, in practice, one can often localize with respect to the Zariski topology.

We insist on the fact that we usually write $(X \hookrightarrow P \leftarrow V)$ for an overconvergent space seen as an object before localization (with formal morphisms) and $(X, V)$ when we see it as an object of the overconvergent site (up to strict neighborhoods).

\begin{prop} \label{isostrict}
A formal morphism
\[
(Y \hookrightarrow Q \leftarrow W) \to(X \hookrightarrow P \leftarrow V)
\]
induces an isomorphism $(Y, W) \simeq (X, V)$ if and only if there exists a commutative diagram
\[
\xymatrix{
&(Y' \hookrightarrow Q' \leftarrow W') \ar[dl] \ar[dr] \\ (Y \hookrightarrow Q \leftarrow W) \ar[rr] && (X \hookrightarrow P \leftarrow V)
}
\]
where both diagonal arrows are strict neighborhoods.
\end{prop}

Note that we may always assume that $Y' = Y$, that the corresponding map is the identity and that $W'$ is an open subset of $W$.

\begin{proof}
Only the direct implication needs a proof.
According to proposition 7.1.20. (i) of \cite{KashiwaraSchapira06} (be careful that they call left what we call right) or section 3.5 of chapter I in \cite{GabrielZisman67}, there exists a commutative diagram of formal morphisms:
\[
\xymatrix{
(Y' \hookrightarrow Q' \leftarrow W') \ar[r] \ar[d] & (X' \hookrightarrow P' \leftarrow V') \ar[d] \ar[dl] \\
(Y \hookrightarrow Q \leftarrow W) \ar[r] & (X \hookrightarrow P \leftarrow V)
}
\]
with strict neighborhoods as vertical maps.
It is sufficient to prove that the upper map is a strict neighborhood.
Since the composite map $Q' \to P' \to Q$ is locally noetherian, the first one $Q' \to P'$ is also necessarily locally noetherian (see \cite{Crew17*}) and condition 2) holds. 
The proofs that the other three conditions hold are very much the same and we will only do condition 3) which is the one that requires some care.
Let us denote by $V''$ the inverse image of $W'$ in $V'$ through the diagonal map, and identify $W'$ and $V''$ with their images in $W$ and $V$ respectively.
Then, there exists a commutative diagram
\[
\xymatrix{
W' \ar[r] \ar@{=}[d] & V'' \ar@{=}[d] \ar[dl] \\
W' \ar[r] & V''
}
\]
The diagonal map has both a section and a retraction and must be an isomorphism.
It follows that the upper map is also an isomorphism and we are done.
\end{proof}

Be careful that the map $Q \to P$ in proposition \ref{isostrict} is not necessarily locally noetherian as the following example shows:
\[
(\emptyset \subset \mathbb A^{\mathrm{b}} \leftarrow \emptyset) \to (\emptyset \subset \mathrm{Spec}(\mathbb Z) \leftarrow \emptyset).
\]

We will mostly be interested in the following consequence of proposition \ref{isostrict}:

\begin{cor} \label{isom}
Assume that we are given two overconvergent spaces
\[
(X_{i} \hookrightarrow P_{i} \leftarrow V_{i})
\]
for $i = 1, 2$.
Then, $(X_{1}, V_{1}) \simeq (X_{2}, V_{2})$ if and only if there exists a common strict neighborhood
\[
\xymatrix{
&(X \hookrightarrow P \leftarrow V) \ar[dl] \ar[dr] \\ (X_{1} \hookrightarrow P_{1} \leftarrow V_{1}) && (X_{2} \hookrightarrow P_{2} \leftarrow V_{2}). \qed
}
\]
\end{cor}

Note that when $X_{2} = X_{1}$, we may also assume that $X = X_{1} = X_{2}$ and that the maps at this level are the identity maps.

Most of the time, we will work over a given overconvergent space $(C \hookrightarrow S \leftarrow O)$.
It is important to notice that, by construction, the diagonal maps in corollary \ref{isom} will be defined over $(C \hookrightarrow S \leftarrow O)$ as well.

It will be convenient to call a formal morphism $(X \hookrightarrow P \leftarrow V) \to (C \hookrightarrow S \leftarrow O)$ \emph{locally noetherian} when the morphism $P \to S$ is locally noetherian.
 
\begin{prop} \label{prodcr}
Assume that we are given two locally noetherian overconvergent spaces $(X_{i} \hookrightarrow P_{i} \leftarrow V_{i})$ over $(C \hookrightarrow S \leftarrow O)$ for $i = 1, 2$.
Then, $(X_{1}, V_{1}) \simeq (X_{2}, V_{2})$ if and only if the projections extend to strict neighborhoods
\[
\xymatrix{
&(X \hookrightarrow P_{1} \times_{S} P_{2} \leftarrow V) \ar[dl] \ar[dr] \\ (X_{1} \hookrightarrow P_{1} \leftarrow V_{1}) && (X_{2} \hookrightarrow P_{2} \leftarrow V_{2}).
}
\]
\end{prop}

Again, when $X_{2} = X_{1}$, we may also assume that $X = X_{1} = X_{2}$ and that the corresponding maps are the identity maps.

\begin{proof}
In the situation of corollary \ref{isom}, we may always assume (as we indicate after the statement of the corollary) that the morphisms are compatible with the structural maps.
Then the common strict neighborhood factors as
\[
\xymatrix{&(X \hookrightarrow P \leftarrow V) \ar[d] \\
&(X \hookrightarrow P_{1} \times_{S} P_{2} \leftarrow V) \ar[dl] \ar[dr] \\ (X_{1} \hookrightarrow P_{1} \leftarrow V_{1}) && (X_{2} \hookrightarrow P_{2} \leftarrow V_{2}).
}
\]
It should be clear that all maps in this diagram are strict neighborhoods.
For example, the composite map
\[
\,]X[_{P,V} \to \,]X[_{P_{1} \times_{S} P_{2},V} \to \,]X_{1}[_{P_{1},V_{1}}
\]
being surjective, the second map must also be surjective and so on.
\end{proof}

The following obvious consequence of proposition \ref{prodcr} will often make it possible to assume that $X$ is a usual scheme when considering an overconvergent space $(X, V)$.

\begin{cor} \label{redred}
Assume that $(X, V)$ and $(X, V')$ are two locally noetherian overconvergent spaces over $(C, O)$.
Then we have
\[
(X, V) \simeq (X, V') \Leftrightarrow (X_{\mathrm{red}}, V) \simeq (X_{\mathrm{red}}, V'). \qed
\]
\end{cor}

Note that this is the same $X$ on both sides and we implicitly assume that the map induced at this level is the identity.

There is no chance for the forgetful functor $(X, V) \to X$ to be continuous (for the Zariski topology of $X$) because we choose to use the coarse topology on the algebraic side when we defined the topology on the overconvergent site.
However, we have the following consequence of proposition \ref{prodcr}:

\begin{cor} \label{loczar}
Let $(X \hookrightarrow P \stackrel \lambda\leftarrow V)$ and $(X \hookrightarrow P' \stackrel {\lambda'}\leftarrow V')$ be two locally noetherian overconvergent spaces over some $(C \hookrightarrow S \leftarrow O)$.
Assume that there exists two open coverings $P = \bigcup_{i \in I} P_{i}$ and $P' = \bigcup_{i \in I} P'_{i}$ such that, for each $i \in I$, $X \cap P_{i} = X \cap P'_{i}$ and, locally,
\[
(X \cap P_{i}, \lambda^{-1}(P_{i}^{\mathrm{ad}}) \simeq (X \cap P'_{i}, \lambda^{'-1}(P_{i}^{'\mathrm{ad}})
\]
Then, locally, $(X, V) \simeq (X, V')$.
\end{cor}

\begin{proof}
Let us write $X_{i} := X \cap P_{i} = X \cap P'_{i}$, $V_{i} := \lambda^{-1}(P_{i}^{\mathrm{ad}})$ and $V'_{i} := \lambda'^{-1}(P_{i}'^{\mathrm{ad}})$.
We assume that there exists, for each $i \in I$, a family of open subsets $\{V_{ij}\}_{j \in J_{i}}$ of $V_{i}$ such that $\,]X_{i}[_{V_{i}} = \bigcup_{j \in J_{i}}\,]X_{i}[_{V_{ij}}$, a family of open subsets $\{V'_{ij}\}_{j \in J_{i}}$ of $V'_{i}$ such that $\,]X'_{i}[_{V'_{i}} = \bigcup_{j \in J_{i}}\,]X'_{i}[_{V'_{ij}}$, and an isomorphism $(X_{i}, V_{ij}) \simeq (X_{i}, V'_{ij})$.
Then, proposition \ref{prodcr} tells us that for each $i \in I, j \in I_{j}$, the projections extend to a common strict neighborhood
\[
\xymatrix{
&(X_{i} \hookrightarrow P_{i} \times_{S} P'_{i} \leftarrow W_{ij}) \ar[dl] \ar[dr] \\ (X_{i} \hookrightarrow P_{i} \leftarrow V_{ij}) && (X_{i} \hookrightarrow P'_{i} \leftarrow V'_{ij}).
}
\]
It follows that the projections
\[
\xymatrix{
&(X \hookrightarrow P \times_{S} P' \leftarrow W_{ij}) \ar[dl] \ar[dr] \\ (X \hookrightarrow P \leftarrow V_{ij}) && (X \hookrightarrow P' \leftarrow V'_{ij})
}
\]
also provide a common strict neighborhood and $(X, V)$ is locally isomorphic to $(X, V')$.
\end{proof}

In other words, this statement says that if $(X,V)$ and $(X,V')$ are locally isomorphic for the Zariski-adic topology, then they are already locally isomorphic for the adic topology.

\subsection{Analytic persistence of formal properties}

Recall that overconvergent spaces were first made into a category using formal morphisms and that we turned strict neighborhoods into isomorphisms in order to obtain the overconvergent site. In the process, the formal scheme which is used as a link between the algebraic and the analytic world almost disappears.
We will see that, however, many properties of the formal morphism are somehow transferred to the (genuine) morphism.

\begin{dfn}
\begin{enumerate}
\item
A formal morphism
\[
\xymatrix{Y \ar@{^{(}->}[r] \ar[d]^f & Q\ar[d]^{v} & W \ar[l] \ar[d]^-{u} \\ X \ar@{^{(}->}[r] & P & V. \ar[l]}
\]
of overconvergent spaces is said to
\begin{enumerate}
\item be \emph{right cartesian} if $W$ is a neighborhood (of the tube) of $Y$ in $(v^{\mathrm{ad}})^{-1}(V)$,
\item \emph{satisfy a property $\mathcal P$} (of formal schemes) if it is right cartesian and $v$ satisfies the property $\mathcal P$ around $Y$.

\end{enumerate}
\item
A morphism of overconvergent spaces $(Y, W) \to (X, V)$ is said to \emph{satisfy a property $\mathcal P$} (of adic spaces) if there exists a neighborhood $V'$ of $X$ in $V$ and a neighborhood $W'$ of $Y$ in $W$ such that the induced map $W' \to V'$ satisfies the property $\mathcal P$.
\end{enumerate}
\end{dfn}

Recall from definitions \ref{formop} and \ref{formcl} that the expression ``around $Y$'' has two different meanings depending on the ``open'' or ``closed'' nature of the property.

\begin{prop} \label{etet}
If a formal morphism
\[
\xymatrix{Y \ar@{^{(}->}[r] \ar[d]^f & Q\ar[d]^{v} & W \ar[l]_\mu \ar[d]^-{u} \\ X \ar@{^{(}->}[r] & P & V. \ar[l]_{\lambda}}
\]
of \emph{analytic} overconvergent spaces is formally locally of finite type and formally unramified (resp.\ and formally smooth, resp.\ and formally \'etale), then it induces an unramified (resp.\ a smooth, resp.\ an \'etale) morphism
\[
(Y, W) \to (X, V).
\]
\end{prop}

\begin{proof}
After completing along $\overline X$ (resp.\ $\overline Y$), we may assume that $X$ (resp.\ $Y$) is open in the topology of $P$ (resp.\ $Q$) so that $v$ is formally locally of finite type (not only around $Y$).

If $v$ is formally unramified around $Y$, then the support $Z$ of the coherent sheaf $\Omega^1_{Q/P}$ is a (closed) subspace of $Q$ that does not meet $Y$.
We showed in proposition \ref{supsh} that the support $T$ of $\Omega^1_{Q^{\mathrm{ad}}/P^{\mathrm{ad}}} \simeq \Omega^{1,\mathrm{ad}}_{Q/P}$ is a closed subset of $Q^{\mathrm{ad}}$ which is contained in $]Z[_{P}$.
In particular, $T$ does not meet $]Y[_{P}$, and it follows that the support $\mu^{-1}(T)$ of $\Omega^1_{W/V} \simeq \mu^*\Omega^1_{Q^{\mathrm{ad}}/P^{\mathrm{ad}}}$ is a closed subset of $W$ which does not meet $\,]Y[_{V}$.
The open complement $W'$ of $\mu^{-1}(T)$ is a neighborhood of (the tube of) $Y$ in $W$ such that $\Omega^1_{W'/V} = 0$.
Moreover, since $v$ is locally of finite type, we may assume thanks to lemma \ref{parpro} below that $W' \to V$ is locally of finite type.
It follows that $u$ is unramified in the neighborhood of $Y$.

We proceed in an analogous way for smoothness using the jacobian criterion.
Since the question is local and $v$ is formally locally of finite type, we may assume that there exists a closed embedding $Q \hookrightarrow \mathbb A^{\pm,N}_{P}$ defined by an ideal $(f_{1}, \ldots, f_{r})$ and that the minor $\det[\partial f_{i}/\partial T_{j}]_{i,j =1}^{r}$ is invertible outside a closed subset $Z$ not meeting $Y$.
It follows that this minor stays invertible outside some closed subset of $W$ not meeting $]Y[_{W}$ and we may consider its open complement $W'$.
By construction, the induced map $W' \to V$ is smooth.
\end{proof}

In the previous proof, we used the following consequence of theorem \ref{parpr}:

\begin{lem} \label{parpro}
If a formal morphism
\[
(Y \hookrightarrow Q \leftarrow W) \to (X \hookrightarrow P \leftarrow V)
\]
of \emph{analytic} overconvergent spaces is formally locally of finite type (resp.\ formally locally of finite type and separated, resp.\ partially proper), then the associated morphism
\[
(Y, W) \to (X, V)
\]
is locally of finite type (resp.\ separated, resp.\ partially proper).
\end{lem}

\begin{proof}
Simply replace $V$ and $W$ with $V' := \,]\overline X[_{V}$ and $W' := \,]\overline Y[_{W}$ and apply theorem   \ref{parpr}.
\end{proof}

\subsection{The strong fibration theorem}

We will now mix open (such as formally smooth) and closed (such as partially proper) conditions.
This will lead us to the strong fibration theorem.

\begin{thm} \label{etis}
If a formal morphism of \emph{analytic} overconvergent spaces
\[
\xymatrix{ X' \ar@{^{(}->}[r] \ar[d]^{f} & P' \ar[d]^{v}& V' \ar[l]_-{\lambda'} \ar[d]^{u} \\
X \ar@{^{(}->}[r] & P & V \ar[l]_\lambda}
\]
is partially proper and formally \'etale, and induces an isomorphism $f \colon X' \simeq X$, then it induces an isomorphism of overconvergent spaces
\[
(X', V') \simeq (X, V).
\]
\end{thm}

\begin{proof}
Before doing anything else, let us remark that, thanks to proposition \ref{etet} and lemma \ref{parpro}, we can assume that the map $u \colon V' \to V$ is partially proper and \'etale (and in particular is locally quasi-finite).
Next, after completing along $\overline X$ (resp.\ $\overline X'$), we may assume as usual that $X$ (resp.\ $X'$) is open in the topology of $P$ (resp.\ $P'$).
If we denote for the moment by $\mathcal U$ (resp.\ $\mathcal U'$) the formal open subscheme of $P$ (resp.\ $P'$) having the same underlying space as $X$ (resp.\ $X'$), then the induced morphism $\mathcal U' \to \mathcal U$ is formally \'etale and induces an isomorphism on the maximal reduced subschemes.
This is necessarily an isomorphism.
We may therefore replace $X$ (resp.\ $X'$) by $\mathcal U$ (resp.\ $\mathcal U'$) and assume from now on that $X$ (resp.\ $X'$) is a formal open subscheme of $P$ (resp.\ $P'$).
Since the right hand square of our diagram is by definition cartesian in the neighborhood of the tubes, the morphism $u$ induces an isomorphism of analytic spaces $]X'[_{V'}^{\mathrm{naive}} \simeq ]X[_{V}^{\mathrm{naive}}$.
It follows that the map $]f[_u : \,]X'[_{V'} \to \,]X[_V$ induced by $u$ on the tubes is ``birational'': it induces an isomorphism on dense open subsets.
This map is actually bijective.
In order to prove that, we may assume that $V = \mathrm{Spa}(K, K^+)$ is an analytic Huber point and $\,]X[_V \neq \emptyset$.
But then, the morphism $u : V' \to \mathrm{Spa}(K, K^+)$ being partially proper locally quasi-finite and birational is necessarily an isomorphism (use proposition 1.5.6 of \cite{Huber96}).
Thus, we see that the canonical map is a homeomorphism $]f[_u : \,]X'[_{V'} \simeq \,]X[_V$.
Now, any $v' \in \,]X'[_{V'}$ generalizes to some $w' \in ]X'[_{V'}^{\mathrm{naive}}$ and since $V'$ is analytic, the canonical map $\mathcal O_{v'} \to \mathcal O_{w'}$ (is local and) induces an isomorphism $\mathcal H(v') \simeq \mathcal H(w')$ on the completed residue fields.
For the same reason, we have $\mathcal H(u(v')) \simeq \mathcal H(u(w'))$.
On the other hand, the isomorphism $]X'[_{V'}^{\mathrm{naive}} \simeq ]X[_{V}^{\mathrm{naive}}$ provides us with an isomorphism $\mathcal H(u(w')) \simeq \mathcal H(w')$ and it follows that $\mathcal H(u(v')) \simeq \mathcal H(v)$.
We may therefore apply proposition 2.3.7 of \cite{Huber96} (analytic variant) which tells us that $u$ induces an isomorphism between the \'etale topos of the pseudo-adic space $(V', \,]X'[_{V'})$ and the \'etale topos of the pseudo-adic space $(V, \,]X[_{V})$.
Since $u \colon V' \to V$ is \'etale, then, necessarily, $u$ induces an isomorphism between a neighborhood of $\,]X'[_{V}$ in $V'$ and a neighborhood of $\,]X[_{V}$ in $V$.
\end{proof}

Note that the \'etale site of a pseudo-adic space $(V, T)$ is not subcanonical because any open embedding $V' \hookrightarrow V$ with $T \subset V'$ will induce an isomorphism between the \'etale topos of $(V', T)$ and $(V, T)$.
This is why we need to shrink $V$ and $V'$ once more at the end of the proof.
Alternatively, one could introduce the \'etale site of germs of adic spaces (which is then subcanonical).

\begin{thm}[Strong fibration theorem]
If a formal morphism of \emph{analytic} overconvergent spaces
\[
(X' \hookrightarrow P' \leftarrow V') \to (X \hookrightarrow P \leftarrow V)
\]
is partially proper and formally smooth, and induces an isomorphism $X' \simeq X$, then there exists, \emph{locally} in the overconvergent site, an isomorphism
\[
(X',V') \simeq (X, \mathbb D^{n-}_{V}).
\]
\end{thm}

Here, we denote by $(X, \mathbb D^{n-}_{V})$ the overconvergent space $(X \hookrightarrow \mathbb A^{n-}_{P} \leftarrow \mathbb D^{n-}_{V})$ obtained by using the zero-section.
Recall that this is the same thing as $(X \hookrightarrow \mathbb A^n_{P} \leftarrow \mathbb D_V^{n})$ or $(X \hookrightarrow \mathbb P^n_{P} \leftarrow \mathbb P_V^{n})$ (up to a strict neighborhood).

\begin{proof}
The strategy is due to Berthelot and we follow essentially the proof of theorem 4.1.3 in \cite{LeStum11}.
First of all, thanks to corollary \ref{redred}, we may assume that both $X$ and $X'$ are reduced.
Also, using proposition \ref{bup2}, we may assume that $X'$ is the complement of a divisor in $\overline X'$.
Now, since the question is local for the adic topology of the overconvergent site, we may assume that $V$ is affinoid.
Moreover, thanks to corollary \ref{loczar}, the question is local on $X$ and therefore also on $P$ that we may both assume to be affine.
Also thanks to proposition \ref{opst1}, we may assume that $P'$ is quasi-compact.
We are then in the situation of applying our version of Chow's lemma (corollary \ref{cholem}) and we may therefore assume that the map $\overline X' \to \overline X$ induced on the compactifications is projective.
Next, proposition \ref{smet2} grants us the existence of a projective morphism of formal embeddings $(X' \hookrightarrow Q) \to (X \hookrightarrow P)$ which is \'etale around $X'$ and such that the closure of $X'$ in $Q$ is identical to the closure of $X'$ in $P'$.
The formal morphism $(X' \hookrightarrow Q \leftarrow Q^{\mathrm{ad}} \times_{P^\mathrm{ad}} V) \to (X \hookrightarrow P \leftarrow V)$ is right cartesian, partially proper and \'etale around $X'$ and induces therefore, thanks to theorem \ref{etis}, an isomorphism in the overconvergent site.
After pulling back along this formal morphism, we may therefore assume that $v$ actually induces an isomorphism $\overline f : \overline X' \simeq \overline X$.
We may now use lemma \ref{locnei} and apply theorem \ref{etis} once again.
\end{proof}

The theorem has an immediate consequence in terms of germs: there exists locally on the germ $\,]X[_V^\dagger$ an isomorphism of germs
\[
\,]X[_{V'}^\dagger \simeq \,]X[_V^\dagger \times \mathbb D^{n-}.
\]

\addcontentsline{toc}{section}{References}
\printbibliography

\Addresses

\end{document}